\documentclass{amsart}
\include{macros}
\renewcommand{\today}{\ifcase \month \or January\or February\or March\or April\or May\or June\or July\or August\or September\or October\or November\or December\fi \space \number \day, \number \year}

\title[Generalized Busemann functions for RWRP]{Existence of generalized Busemann functions and Gibbs measures for random walks in random potentials}

\author[S.~Groathouse]{Sean Groathouse}
\address{Sean Groathouse\\ University of Utah\\  Department of Mathematics\\ 155 S 1400 E\\   Salt Lake City, UT 84112}
\email{sean.groathouse@utah.edu}
\urladdr{http://www.math.utah.edu/~sean/}
\thanks{S.\ Groathouse and F.\ Rassoul-Agha were partially supported by National Science Foundation grants DMS-1811090 and DMS-2054630}

\author[C.~Janjigian]{Christopher Janjigian}
\address{Christopher Janjigian\\ Purdue University\\   Department of Mathematics\\  150 N University St\\  West Lafayette, IN 47907}
\email{cjanjigi@purdue.edu}
\urladdr{http://www.math.purdue.edu/~cjanjigi}
\thanks{C.\ Janjigian was partially supported by National Science Foundation grant DMS-2125961}

\author[F.~Rassoul-Agha]{Firas Rassoul-Agha}
\address{Firas Rassoul-Agha\\ University of Utah\\  Department of Mathematics\\ 155 S 1400 E\\   Salt Lake City, UT 84112}
\email{firas@math.utah.edu}
\urladdr{http://www.math.utah.edu/~firas}
\thanks{F.\ Rassoul-Agha  was partially supported by MPS-Simons Fellowship grant 823136}

\keywords{Busemann functions, cocycle, directed polymer in a random environment, first-passage percolation, Gibbs measure, harmonic function, last-passage percolation, Martin boundary, random polymer, random walk in a random environment, random growth, shape theorem}
\subjclass[2020]{60K35, 60K37} 

\date{\today}

\begin{document}

	\begin{abstract}
	We establish the existence of generalized Busemann functions and Gibbs-Dobrushin-Landford-Ruelle measures for a general class of lattice random walks in random potentials with finitely many admissible steps. This class encompasses directed polymers in random environments, first- and last-passage percolation, and elliptic random walks in both static and dynamic random environments in all dimensions and with minimal assumptions on the random potential.
	\end{abstract}

	\maketitle
 
	\section{Introduction}
The model of a random walk interacting with a random potential (RWRP) has been a major topic of research in probability over the last half-century. Through various specializations, it encompasses random walks in both static and dynamic random environments, directed polymers in random environments, as well as zero-temperature models of deterministic walks in random environments like directed and undirected first- and last-passage percolation. The objects we investigate, called generalized Busemann functions, have been previously studied in specific instances through the structure of positive harmonic functions \cite{Yil-09-aop,Yil-09-cpam,Yil-11-aop} and the associated Martin boundary \cite{Bal-Ras-Sep-19}, 
infinite volume Gibbs-Dobrushin-Landford-Ruelle measures and geodesics \cite{Cat-Pim-11,Bus-Sep-Sor-24,Sep-Sor-23-pmp, Alb-Ras-Sim-20,New-95,Dam-Han-14, Dam-Han-17, Geo-Ras-Sep-17-ptrf-2,Jan-Ras-20-aop,Bak-Li-19,Bak-Cat-Kha-14,Bak-13,Bak-16,Jan-Ras-Sep-22-1F1S-}, increment stationary distributions of directed polymers and stochastic Hamilton-Jacobi equations \cite{Bus-Sep-Sor-24,Sep-Sor-23-pmp,Alb-Ras-Sim-20,Cat-Pim-12, Bak-Cat-Kha-14, Bak-13,Bak-16,Bak-Li-19, Jan-Ras-20-jsp,Jan-Ras-Sep-22-1F1S-}, solutions to variational formulas for the limiting free energy and shape function \cite{Bal-Ras-Sep-19,Yil-09-cpam,Ras-Sep-11,Ras-Sep-Yil-13,Ras-Sep-Yil-17-ber}, and, in the related stochastic Hamilton-Jacobi setting, correctors for the stochastic homogenization problem \cite{Car-Sou-17,Jan-Ras-Sep-22-1F1S-}. 

In this introductory section, we will motivate the model and questions we consider informally, ignoring technical complications. A careful treatment follows in the main body of the text. Our model begins with a time-homogeneous reference random walk with finitely many admissible steps on an integer lattice $\bbZ^d$. We denote the set of admissible steps by $\range$, the associated transition mass function by $p$, and the path law started at $x \in \bbZ^d$ by $\rwP_x$. The random walk interacts with a random potential $\{V(\w,z) : z \in \range\}$ through its position and its increment.  

At positive temperature, the \emph{quenched {\rm(}unrestricted-length{\rm)} point-to-point polymer measure} in \emph{environment $\w$} is a Gibbsian measure on paths $X_{\aabullet}$ which emanate from a site $x$, defined via the Radon-Nikodym derivative
\begin{align}
    \frac{d\Qunr{x}{y}{\beta, \w}}{d\rwP_x}(X_{\aabullet}) = \frac{e^{-\beta\sum_{k=0}^{\stoppt{y}-1} \pote(\T_{X_k} \omega, X_{k+1}-X_k)}}{\prtunr{x}{y}{\beta,\w}}\one_{\{\stoppt{y}<\infty\}}(X_{\aabullet}). \label{eq:unrq}
\end{align}
In the previous expression, $T_x\w$ is a shift of the environment $\w$ by $x$, $\tau_y$ is the first strictly positive time the path reaches site $y$, and the \emph{partition function} is
\begin{align}
\prtunr{x}{y}{\beta,\w} = \rwE_{x}\Bigl[e^{-\beta\sum_{k=0}^{\stoppt{y}-1} \pote(\T_{X_k} \omega, X_{k+1}-X_k)}\one_{\{\stoppt{y}<\infty\}}\Bigr]. \label{eq:unrpart}
\end{align}
Here $\rwE_{x}$ is the expectation with respect to $\rwP_x$, and $\beta \in (0,\infty)$ is interpreted as the \emph{inverse temperature}.

A straightforward computation checks that these measures are Markovian with one-step transition probabilities from $x \neq y$ to $x+z$, $z \in \range$, given by
\begin{align}
p(z)e^{-\beta \pote(\T_x\w,z)} \frac{\prtunr{x+z}{y}{\beta,\w}}{\prtunr{x}{y}{\beta,\w}} = p(z) e^{-\beta(\pote(\T_x\w,z) + \frac{1}{\beta}(\log\prtunr{x}{y}{\beta,\w} - \log\prtunr{x+z}{y}{\beta,\w}) )}.\label{eq:unrtrans}
\end{align}

As $\beta \to \infty$, one should expect the measure in \eqref{eq:unrq} to converge to a distribution supported on paths which minimize the potential along the path; i.e., the measure is asymptotically supported on paths which optimize 
\begin{align}
-\min_{x_{\abullet} \in \supp \rwP_x} \biggl\{ \sum_{k=0}^{\tau_y-1} \pote(\T_{x_k}\w, x_{k+1}-x_{k})\biggr\} = \freeunr{x}{y}{\infty,\w} = \lim_{\beta\to\infty} \frac{1}{\beta}\log \prtunr{x}{y}{\beta,\w}.\label{eq:fpplim}
\end{align}
This ``zero temperature" RWRP model is known in the literature as \emph{first-passage percolation} (FPP), with the standard undirected Euclidean model arising if $\rwP_x$ is any random walk with the same admissible steps as simple symmetric random walk on $\bbZ^d$. In that setting $\freeunr{x}{y}{\infty,\w}$ is called the \emph{passage time}. The statement that the transition probabilities in \eqref{eq:unrtrans} sum to one becomes the following local control problem, which can be used to construct minimizing paths, called \emph{geodesics} in FPP:
\begin{align*}
    0 &= \min_{z\in \range}\{\pote(\T_x\w, x+z) + \freeunr{x}{y}{\infty,\w} - \freeunr{x+z}{y}{\infty,\w} \}.
\end{align*}

If one replaces stopping on reaching $y$ with stopping after $n$ steps, then the positive-temperature polymer measure defined through \eqref{eq:unrq} is what we call the \emph{quenched restricted-length point-to-level polymer measure}, which is defined via the Radon-Nikodym derivative 
\begin{align}
    \frac{d\QnStep{x}{n}{\beta, \w}}{d\rwP_{x,n}}(X_{\aabullet})=\frac{e^{-\beta\sum_{k=0}^{n-1} \pote(\T_{X_k} \omega, X_{k+1}-X_k)}}{\prtnstep{x}{n}{\beta, \w}}.\label{eq:reslen}
\end{align}
In the previous expression, $\rwP_{x,n}$ is the restriction of $\rwP_{x}$ to the first $n$ steps of the walk and 
\[\prtnstep{x}{n}{\beta}(\omega) = \rwE_{x,n}\Bigl[e^{-\beta\sum_{k=0}^{n-1} \pote(\T_{X_k} \omega, X_{k+1}-X_k)}\Bigr]. \]
Again, these measures are Markovian with one-step transition probabilities from $x$ to $x+z$, where $z \in \range$, given by
\begin{align}
p(z)e^{-\beta V(T_x\w,z)} \frac{\prtnstep{x+z}{n-1}{\beta}(\omega) }{\prtnstep{x}{n}{\beta}(\omega) } = p(z) e^{-\beta(V(T_x\w,z) + \frac{1}{\beta}(\log\prtnstep{x}{n}{\beta} - \log\prtnstep{x+z}{n-1}{\beta} ))}.\label{eq:restrans}
\end{align}

Denote by $\bbP$ the law of the random environment $\w$. In general, the above quenched point-to-level polymer measures are not Kolmogorov-consistent as $n$ varies unless it happens to be the case that $\bbP$-almost surely for all $x$ and all $n$,
\begin{align}
\prtnstep{x}{n}{\beta}(\omega) = 1\qquad \text{ or, equivalently, }\qquad \sum_{z\in \range} p(z)e^{-\beta \pote(\T_x \w, z)} =1.\label{eq:RWRE}
\end{align}
If this holds, then the restricted point-to-level model reduces to that of a \emph{random walk in a random environment} (RWRE).

As previously noted, except in the special case where the model is an RWRE, the measures $(\QnStep{x}{n}{\beta, \w})_{n}$ and $(\Qunr{x}{y}{\beta,\w})_y$ discussed above are not in general Kolmogorov-consistent.  The domain Markov property satisfied by another family of measures (restricted length point-to-point, defined further down in \eqref{finitePathMeasures}) gives Dobrushin-Landford-Ruelle (DLR) equations which characterize infinite volume Gibbs measures in this setting.
 One typically expects such measures to arise in the \emph{thermodynamic limit} as the terminal condition ($y$ or $n$, respectively) tends to infinity. Such a limit is equivalent to the convergence of the transition probabilities defined through \eqref{eq:unrtrans} (or \eqref{eq:restrans}). Writing $\beta^{-1} \log \prtunr{x}{y}{\beta,\w} = \freeunr{x}{y}{\beta,\w},$ this, in turn, is equivalent to the convergence of limits of the form
\begin{align}
    \B{\beta,\w}(x,y) = \lim_u \Bigl\{ \freeunr{x}{u}{\beta,\w} -\freeunr{y}{u}{\beta,\w}\Bigr\},\label{eq:intbusf}
\end{align}
where the limit is taken as $u\to\infty$ in an appropriate sense. In the zero-temperature setting of FPP, the analogous limits  
\begin{align}
    \B{\infty,\w}(x,y) = \lim_u \Bigl\{\freeunr{x}{u}{\infty,\w} - \freeunr{y}{u}{\infty,\w}\Bigr\} \label{eq:intbusi}
\end{align}
are known as \emph{Busemann functions}, which is a name inherited from their interpretation in metric geometry. We keep this terminology in the general model.

The existence of such limits is highly non-trivial to prove in general, with most arguments, even in the planar $d=2$ setting, relying on strong and generally unproven hypotheses about the limiting \emph{free energy} ($\beta<\infty$) or \emph{limit shape} ($\beta=\infty$). These quantities are defined respectively (in the point-to-point setting) through the limits
\begin{align*}
    \ppShapeUnr{\beta}(\xi) = \lim_{n \rightarrow \infty} \frac{1}{n} \freeunr{\orig}{x_n}{\beta,\w} \qquad \text{ and }\qquad     \ppShapeUnr{\infty}(\xi) = \lim_{n \rightarrow \infty} \frac{1}{ n} \freeunr{\orig}{x_n}{\infty,\w},
\end{align*}
where $x_n/n \to \xi\in \bbR^d$. In the setting of RWRE, the corresponding object has also been studied extensively in the guise of the quenched large deviation rate function for the time $n$ position of the path. 

If the limits in \eqref{eq:intbusf} and \eqref{eq:intbusi} exist, then these Busemann functions satisfy four key properties, the first three of which are inherited immediately from the pre-limit  structure of the model:
\begin{enumerate}[label={\rm(\roman*)}, ref={\rm\roman*}]   \itemsep=3pt
\item \textup{(Recovery)} For  $\beta<\infty,$
\begin{align*}
    1= \sum_{z\in\range} p(z) e^{-\beta(V(\w,z) + \frac{1}{\beta}\B{\beta,\w}(\orig,z))},
\end{align*}
and for $\beta=\infty$,
\begin{align*}
    0 &= \min_{z\in \range}\{\pote(\w, z) + \B{\infty,\w}(\orig,z) \}.
\end{align*}
\item \textup{(Cocycle)} 
\[\B{\beta,\w}(x,y)+\B{\beta,\w}(y,z) = \B{\beta,\w}(x,z).\]
\item \textup{(Shift covariance)} 
\[
\B{\beta,T_z\w}(x,y) = \B{\beta,\w}(x+z,y+z).
\]
\item \textup{(Duality)} There exists $m \in \partial \ppShapeUnr{\beta}(\xi)$ for some $\xi$ for which
\begin{align}\label{duality}
\bbE[\B{\beta,\w}(x,y)] = m\cdot(x-y).
\end{align}
\end{enumerate}
 In the duality expression, $\partial\ppShapeUnr{\beta}(\xi)$ denotes the superdifferential of the concave function $\ppShapeUnr{\beta}$. With some convex analysis, one can generally show that duality is actually a consequence of the previous three conditions. We include this condition as part of our (informal) definition to simplify the discussion below. Random fields which satisfy these four properties are called \emph{generalized Busemann functions}. Where such objects exist, they can serve essentially the same role as true Busemann functions defined through the limits in \eqref{eq:intbusf} and \eqref{eq:intbusi}. In particular, in the setting of polymer models (resp.\ FPP/LPP), the recovery property of the generalized Busemann functions can be used to construct semi-infinite Markovian path measures (resp.\ paths). Then, the cocycle property implies that these measures (resp.\ paths) satisfy the aforementioned DLR equations (resp.\ are geodesics). Gibbs measures (resp.\ geodesics) built in this way have extra structure and can be shown to also be consistent with the unrestricted-length point-to-point measures (resp.\ geodesics) which are discussed above.

As the reader will see in the body, because of how generally we work in this project, the careful statements of our main results are fairly technical. For this reason, we give informal statements in this introduction, with precise statements to follow. Our main technical result, stated carefully in Theorem \ref{thm:Cocycles} below, is that generalized Busemann functions essentially always exist.
\begin{itheorem}\label{ithm:Cocycles}
For appropriate choices of the distribution $\P$ of the environment $\w$ and all choices of the reference walk with finitely many admissible steps, for each $\xi$ and $m \in \partial  \ppShapeUnr{\beta}(\xi)$, there exist random variables $\B{\beta,\w}(x,y)$ satisfying the duality \eqref{duality} and the recovery, cocycle, and covariance properties above.
\end{itheorem}

In the statement above, what is meant by ``appropriate" is that the environment needs to be appropriately ergodic or mixing and satisfy mild moment hypotheses. Our hypotheses include most models with finitely many admissible steps studied in the literature, with two significant exceptions being degenerate (i.e.~not elliptic) random walks in random environments and walks on percolation clusters.

The fourth condition in our informal definition of the generalized Busemann function was called duality because it expresses a Legendre-Fenchel duality (through the concave function $\ppShapeUnr{\beta}$) between the mean vector $m$ of the Busemann function and a direction $\xi$ for which $m \in \partial\ppShapeUnr{\beta}(\xi)$ which appears when one uses the Busemann function to construct Gibbs measures or infinite geodesics. 

 The study of the Gibbs measures and semi-infinite geodesics has been a focus of significant recent attention. We discuss some of these related works below and refer the reader to the introduction of \cite{Jan-Ras-Sep-23} for a more comprehensive discussion. The main result of this paper establishes the existence of such measures and geodesics in a broad setting, while also detailing some of their fundamental properties. The precise statement of our existence result is Theorem \ref{thm:CocycleMeasuresFace} below. Informally, we can summarize that result as follows.
 \begin{itheorem}\label{ithm:Gibbs}
Under the same hypotheses as Informal Theorem \ref{ithm:Cocycles}, for each $x \in \Z^d$, $\xi$, and $m \in \partial\ppShapeUnr{\beta}(\xi)$, there exists a probability measure $\Qinf{x}{\beta, \m, \w}$ on infinite paths starting at $x$. This measure satisfies Gibbs consistency with the restricted-length point-to-point measures and, in an appropriate sense, with the unrestricted-length measures. Furthermore, it corresponds to the tilt $m$ through the following dualities:

For almost all $\w$, $\Qinf{x}{\beta,\m,\w}$-almost surely, the limit points of $X_n/|X_n|_1$ lie in the set of directions $\zeta$ such that $m \in \partial\ppShapeUnr{\beta}(\zeta)$. Similarly, the limit points of $X_n/n$ lie in the set of velocities satisfying the analogous condition with respect to the restricted-length free energy.
\end{itheorem}

The need for understanding consistency with the unrestricted-length measures ``in an appropriate sense" in the statement above is because the unrestricted-length measures are not in general self-consistent in the Gibbs sense. We discuss this point further in Appendix \ref{app:consistency}.

Appendix A.1 of the Ph.D.\ thesis \cite{Gro-23} contains some further preliminary results on the general structure of infinite volume Gibbs measures in this general setting, including their equivalence to recovering cocycles.
 
 
What we are calling generalized Busemann functions have previously appeared with other interpretations in other settings. Suppose that the model is an irreducible RWRE as in \eqref{eq:RWRE} and that $h$ solves for all $x \in \bbZ^d$ and some $c\in\bbR$,
\begin{align}
h(x) = \sum_{z}p(z)e^{-\pote(\T_x\w,z)-c} h(x+z).\label{eq:harmonic}
\end{align}
Then the above equation says that $h(X_n)e^{cn}$ is a (quenched) martingale for the RWRE. When $c=0$, this becomes the usual notion of a harmonic function. We call the $c\neq 0$ case a \textit{time-dependent} harmonic function. 

If $\B{}$ satisfies the recovery and cocycle properties discussed above, then the function $h(x) = e^{\B{1,\w}(0,x)}$  is harmonic (with $c=0$). More generally, we will see that the (space-time) generalized Busemann functions (see Remark \ref{rmk:spaceTimeCocycle} below for a careful definition) coming from the restricted-length measures in \eqref{eq:reslen} define time-dependent harmonic functions as in \eqref{eq:harmonic}. Such time-dependent harmonic functions have previously been used to study RWRE in \cite{Yil-09-aop, Yil-09-cpam, Yil-11-aop}. One could conversely start with a covariant positive harmonic function with $c=0$ and construct a generalized Busemann function by setting $B(x,y) = \log h(x) - \log h(y).$ A similar statement holds for space-time generalized Busemann functions in the case of $c \neq 0.$

Through this connection, generalized Busemann functions can be seen to be closely related to the Martin boundary theory of the RWRE and generalizations to the RWRP model. In this interpretation, the infinite volume polymer measures discussed in Informal Theorem \ref{ithm:Gibbs} are the Markov processes generated by Doob $h$-transforms of the RWRE  with respect to this harmonic $h$. Through this connection, it is our hope that the study of the objects constructed in this work may shed some light on the long-open questions concerning zero-one laws for RWRE models previously studied in \cite{Kal-81,Ras-05,Zer-Mer-01,Zer-07,Slo-21-,Tou-15}. 

As mentioned in the opening paragraph of this manuscript, many other connections are present in related settings. Next, we briefly comment on two of these connections.
Generalized Busemann functions have been shown in \cite{Jan-Ras-20-jsp} to be equivalent to translation invariant stationary distributions for directed polymer models (in the sense that a realization of one induces a realization of the other). See also \cite{Jan-Ras-Sep-22-1F1S-}. These translation-invariant stationary distributions play a prominent role in the KPZ scaling theory, which predicts the values of the non-universal constants needed to center and scale to see the universal distributions in the KPZ class \cite{Kru-Mea-Hal-92,Spo-14}.

In the closely related setting of stochastic Hamilton-Jacobi equations, generalized Busemann functions correspond to globally defined solutions to the corrector equation, as discussed in \cite{Car-Sou-17,Jan-Ras-Sep-22-1F1S-}. The construction of such objects is a key step in one approach to the problem of proving stochastic homogenization for such models, beginning with the pioneering work \cite{Lio-Pap-Var-87}. We refer the reader to the discussion in \cite{Bak-Kha-18} as well as that in \cite{Jan-Ras-Sep-22-1F1S-} for connections between the problems studied here and the general stochastic Hamilton-Jacobi setting.

Before concluding this section on motivation, it is important to clarify that our focus in this work lies in the simultaneous construction of the generalized Busemann functions for a countable dense set of vectors in the superdifferential of the limiting free energy. Our scope does not encompass the complete construction of the Busemann process, meaning a stochastic process indexed by the whole superdifferential. Thus far, such a process has only been successfully constructed generally in two-dimensional directed nearest-neighbor settings \cite{Jan-Ras-20-aop, Jan-Ras-Sep-23}, where the path structure imposes a monotonicity that allows for the construction of the Busemann process based on its values at a dense countable set of directions.


We close this introductory section by remarking that this work leaves open many of the usual questions, such as those considered in Newman's seminal work \cite{New-95}, regarding the structure of the semi-infinite Gibbs measures and geodesics. In particular, it is natural to wonder whether all semi-infinite Gibbs measures and geodesics are directed into faces of the unrestricted-length limit shape.  One might also ask for conditions under which one can show the uniqueness of these semi-infinite measures and geodesics or the non-existence of bi-infinite measures and geodesics. See \cite{Cha-Kri-19, Cha-Kri-21} for some recent relevant results. The existence of the full Busemann process and its connection to the non-uniqueness of Gibbs measures and infinite geodesics as in \cite{Jan-Ras-20-aop,Jan-Ras-Sep-23}  would also be of interest.

Another open problem of interest is to determine the conditions under which paths have an asymptotic law of large numbers (LLN) velocity. In the context of standard FPP, this is commonly referred to as the problem of asymptotic geodesic length. Our results show that a sufficient condition for paths to possess an almost sure asymptotic velocity is for the restricted-length limiting free energy to be strictly concave. However, it is known that this condition does not hold universally, as exemplified by marginally nestling RWRE, where the limiting free energy features a linear segment \cite[Theorems 7.4 and 8.1]{Var-03-cpam}. Nevertheless, even for such models, it has been proven that in certain special cases the paths do have an asymptotic velocity \cite{Ras-03-aop,Szn-Zer-99,Com-Zei-04}, but the general scenario remains unresolved. The results in \cite{Bat-24} and \cite{Kri-Ras-Sep-23} suggest that in the standard FPP model, for a given asymptotic direction, there is a large class of weight distributions for which there are multiple asymptotic speeds and, conversely, there is a large class of weight distributions for which there is a unique asymptotic speed.  Consequently, it is natural to question whether or not having positive temperature forces the LLN to hold, as is conjectured, for example, in the case of a uniformly elliptic RWRE with i.i.d.\ transition probabilities.

\subsection{Methods and related work}\label{sec:MRW}
In the polymer and percolation literature, two main approaches have been widely utilized to establish the existence of Busemann functions. The earliest approach can be traced back to the pioneering work of Newman and collaborators \cite{New-95,Lic-New-96,How-New-01} on first-passage percolation. 

In this approach, quantitative estimates on the strict convexity (or concavity, depending on sign conventions)
of the limiting shape are used to prove geodesic coalescence, from which the existence of Busemann functions follows as a consequence. These ideas were later carried over to other percolation \cite{Bak-13,Bak-Cat-Kha-14,Cat-Pim-12,Cat-Pim-13} and polymer \cite{Bak-Li-19} models.

However, it is worth noting that obtaining the required curvature bounds for this approach, except for a few specific cases where the shape function can be explicitly computed, remains a major open problem in the field. More importantly, from our perspective, such conditions are provably false in full generality.  It was shown by H\"aggstrom and Meester \cite{Hag-Mee-95} that in two dimensions, every compact convex shape with all the symmetries of the lattice appears as the limit shape of some stationary FPP model. See also the polygonal examples constructed in \cite{Ale-Ber-18,Bri-Hof-21}. Moreover, even in the FPP setting, geodesic coalescence in dimensions three or higher continues to be an unresolved problem, with even its validity in high dimensions remaining unclear.

The approach we employ in this work is based on the connection between generalized Busemann functions and invariant measures. It broadly follows the idea of C\'esaro averaging distributions of Markov processes to produce stationary distributions, a concept that traces back at least to the classical Krylov-Bogoliouboff theorem \cite{Kry-Bog-37}. The essential technical difficulty in this approach in RWRP models is constructing generalized Busemann functions that satisfy the duality condition \eqref{duality} for a rich set of vectors $m$. There are two main steps to such an argument: constructing tight approximate Busemann functions with means converging to $m$ and then establishing uniform integrability.

A method for proving the existence of generalized Busemann functions with the correct mean structure was developed (mostly in the context of FPP)  over a series of works tracing from  Liggett's proof of the subadditive ergodic theorem \cite[Theorem 1.10]{Lig-85}, to Garet and Marchand \cite{Gar-Mar-05} to Gou\'er\'e \cite{Gou-07}, Hoffman \cite{Hof-08}, and finally culminating with Damron and Hanson \cite{Dam-Han-14}. This approach was subsequently applied by Cardaliaguet and Souganidis \cite{Car-Sou-17} to construct correctors in a class of stochastic Hamilton-Jacobi equations.


 The uniform integrability requirement is simplified in settings like FPP where loops are possible and easy  upper and lower bounds on the approximate Busemann functions immediately imply the result. In settings where loops are not allowed in some directions, only one-sided bounds follow quickly from the model structure, and one needs more involved tools to obtain the mean convergence. In the queueing literature, this issue was also encountered in the construction of stationary $\abullet$/G/1/$\infty$ queues in \cite{Mai-Pra-03}, which also follows the general Krylov-Bogoliouboff approach mentioned above. This construction was later applied in \cite{Geo-Ras-Sep-17-ptrf-1} to show the existence of Busemann functions in the two-dimensional directed last-passage percolation model.  A generally applicable method was introduced in \cite{Jan-Ras-20-aop} in the context of a $1+1$ dimensional nearest-neighbor directed polymer in an i.i.d.\ random environment, where  a variational representation of $\ppShapeUnr{\beta}$, originally proven in that setting in \cite{Geo-Ras-Sep-16}, was used to show the uniform integrability.

Implementing these previous ideas in this general setting is complicated by the variety of admissible paths and the structure of the limiting free energy and shape function in the range of models we study. Our approach is a hybrid of that of \cite{Dam-Han-14} and \cite{Jan-Ras-20-aop}, but with some significant differences from both. In \cite{Dam-Han-14}, the approximate Busemann functions corresponding to a direction $\xi$ are increments of point-to-(randomized-) hyperplane passage times, where the hyperplane is chosen to be tangent to the limit shape in direction $\xi$. These point-to-hyperplane passage times can be viewed as a random discrete approximation to the dual norm of the limit shape. One sees immediately from this construction how the relationship to a vector in the subdifferential of the (convex) limit shape arises. In the restricted-length setting of \cite{Jan-Ras-20-aop}, by contrast, the approximate Busemann functions are increments of point-to-(randomized-)time-level passage times, with an external field added to the potential. These tilted point-to-level passage times are a random discrete approximation of the Legendre transform of the shape function. Standard convex analytic arguments then connect the external field to the superdifferential of the (concave) limiting free energy.

In our setting, which encompasses models with characteristics of both directed and undirected models, neither of these approaches appears to work on its own. Notably, there are cases where the tangent to the limiting free energy can be the zero vector, or the limiting free energy itself can be zero in non-trivial directions (e.g., in certain RWRE models), which poses challenges in implementing the approach of \cite{Dam-Han-14} in a general setting. Furthermore, when loops are permitted in the model, introducing external fields to unrestricted-length point-to-level models may lead to infinite partition functions. To address these challenges, we employ a hybrid approximation utilizing a point-to-hyperplane free energy with an external field, which proves successful except in purely directed cases where all hyperplanes become the zero hyperplane. In such instances, our approach closely resembles the one adopted in \cite{Jan-Ras-20-aop}. 

With generalized Busemann functions constructed, the existence of Gibbs-Dobrushin-Landford-Ruelle measures follows essentially immediately from the cocycle and recovery properties. We show that these measures are simultaneously consistent with both the restricted and unrestricted length point-to-point measures, as well as with the Green's function.  The next natural questions are if the paths under these measures are directed (meaning that $X_n/|X_n|_1$ converges) and if they  satisfy a law of large numbers (meaning that $X_n/n$ converges).

The argument that the paths under the semi-infinite measures have an asymptotic direction requires some care because our models allow for the possible existence of traps that constrain paths to bounded sets. We give an exact characterization of when this occurs and provide easy-to-check conditions on the environment that ensure the paths do not become trapped.  We then show that in the absence of traps, the paths are strongly directed into the set of directions that are in a Legendre-Fenchel duality (defined through the concave function $\ppShapeUnr{\beta}$) with the mean of the Busemann function. This is the set of directions $\xi$ such that $m \in \partial \ppShapeUnr{\beta}(\xi)$, where $m$ is the mean vector of the Busemann function. This essentially follows by combining the recovery property of the generalized Busemann functions with the shape theorems for both the free energy and the generalized Busemann functions themselves.
This argument is similar to the approach taken in \cite{Dam-Han-14,Geo-Ras-Sep-17-ptrf-2}, with certain details modified in the positive temperature case.

We also prove a large deviation principle for the velocity of the path under the semi-infinite polymer measure, using methods similar to \cite{Jan-Ras-20-aop,Ras-Sep-14}. The possible law of large numbers limiting velocities are then included in the zero set of the rate function, which turns out to be the set of vectors $\xi$ that are dual to the mean of the Busemann function through the concave function $\ppShapeUnr{\beta}$, i.e., such that $m \in \partial \ppShapeUnr{\beta}(\xi)$.


We conclude this section with a brief technical remark. For simplicity, the preceding discussion focused solely on the free energy $\ppShapeUnr{\beta}$ and its superdifferential as a function on the cone generated by the admissible steps $\range$. In actuality, when working with a direction $\xi$ in this cone, we restrict attention to the unique face of the cone containing $\xi$ in its relative interior. This restriction arises because paths with $\xi$ as their asymptotic direction can only take steps within this face.
Another subtle technical point is that the free energy $\ppShapeUnr{\beta}$ is not generally known to be continuous up to the boundary of the face. To address this, we instead consider 
the unique continuous extension of $\ppShapeUnr{\beta}$ from the interior to the entire face,
which happens to be given by the upper-semicontinuous regularization of the restriction of $\ppShapeUnr{\beta}$ to the interior of the face.
The Gibbs measures (and geodesics) we construct are then supported on semi-infinite paths that are constrained to steps lying within the corresponding face of the cone.

\subsection*{Outline of paper} 
The outline of the remainder of the paper is as follows. Section \ref{sec:setting} describes our setting, defining the relevant path spaces and RWRP measures,  then collects results about the free energies, and states the shape theorems. Section \ref{sec:main} collects the technical conditions for our results and states our main result concerning the existence of Gibbs measures and geodesics. Section \ref{sec:coc} includes both the statement and proof of our results concerning the existence of the generalized Busemann functions. Section \ref{sec:semiinf} contains the construction of the associated semi-infinite Gibbs measures. Section \ref{sec:MainThmProofs} then concludes the proofs of the results stated in Section \ref{sec:main}. Appendix \ref{app:conv} collects some basic facts we use from linear algebra and convex analysis. Appendix \ref{app:consistency} discusses Gibbs consistency and inconsistency of the various RWRP models we consider. Appendix \ref{sec:shapeThms} proves the shape theorems we state in Section \ref{sec:setting}. 
Appendix \ref{app:resunr} relates the restricted- and unrestricted-length limiting free energies.

\subsection*{Notation}
$\R$ denotes the real numbers and $\ZZ$ denotes the integers. For $a\in\R$, $\R_{\ge a}=[a,\infty)$, $\R_{>a}=(a,\infty)$, $\ZZ_{\ge a}=\R_{\ge a}\cap\ZZ$, and $\R_{\le a}$, $\R_{<a}$, $\ZZ_{>a}$, $\ZZ_{\le a}$, $\ZZ_{<a}$ are similar.
For $x \in \RR^d$, let $\abs{x}_1$ denote the $\ell^1$ norm of $x$, and $\abs{x}_\infty$ denote the $\ell^\infty$ norm of $x$.  A sequence $(x_k)_{n\le k\le m}$ is denoted by $x_{n:m}$. Similarly, $x_{-\infty:n}=(x_k)_{k\le n}$, $x_{n:\infty}=(x_k)_{k\ge n}$, and $x_{-\infty:\infty}=(x_k)_{k\in\ZZ}$. When the index set is understood, we denote the sequence by $x_\abullet$. For $a\in\R$, $a^+=\max(a,0)$ and $a^-=-\min(a,0)$. 

Throughout, we use the convention that $\inf \varnothing = \infty$. $\ri A$ denotes the relative interior of a set $A\subset\R^d$. $\ext A$ denotes the set of extreme points of a convex set $A\subset\R^d$.

Given a metrizable topological space $\sX$, we denote by $\sB(\sX)$ the Borel $\sigma$-algebra of $\sX$. The collection of probability measures on $(\sX,\sB(\sX))$ is denoted by $\sM_1(\sX)$ and equipped with the weak topology.

The symbol $\triangle$ marks the end of a numbered remark.

	\section{Setting}\label{sec:setting}
	In this section, we describe the models of random walks in random potential that we study in this work. 
	The examples at the end of Section \ref{sec:rwrp} show how a large number of 
 models studied elsewhere in the literature are special cases of the models we consider.

	\subsection{The admissible paths}
	We begin by defining the path spaces used by the various models. Let $d\ge 1$ be a positive integer and consider a non-empty finite subset $\range\subset\ZZ^d$ containing at least two points. This is the set of admissible steps. $\range$ may contain the zero vector $\zero \in \ZZ^d$.   
	
	We denote the additive group and semi-group generated by $\range$ by \[\Rgroup= \Bigl\{ \sum_{z \in \range} b_z  z : b_z \in \ZZ \Bigr\} \qquad \text{ and } \qquad \Rsemi = \Bigl\{ \sum_{z \in \range} b_z z : b_z \in \ZZ_{\ge0} \Bigr\}.\] 

  For each $n \in \ZZ_{\ge0}$, the set of points that are accessible from $\zero$ in exactly $n$ admissible steps is
	\[
	\Dn{n} = \Bigl\{ x \in \Rsemi : \exists(b_z)_{z\in\range} \in \ZZ_{\ge0}^\range \text{ with } \sum_{z \in \range} b_z = n \text{ and } x = \sum_{z \in \range} b_z z \Bigr\}.
	\]
	
	We now define several collections of admissible paths.  The set of all admissible paths of length $n\in\ZZ_{\ge0}$ from $x\in\Z^d$ is denoted by
	\[
	\PathsNStep{x}{n} = \bigl\{x_{0:n} : x_0 = x, x_i - x_{i-1} \in \range \text{ for all } 1 \leq i \leq n \bigr\}.
	\]
	For integers $j \leq k$, the set of admissible paths from $x\in\Z^d$ at time $j$ to $y\in\Z^d$ at time $k$ is denoted by
	\[
	\PathsPtPResKM{x}{y}{j}{k} = \bigl\{x_{j:k} : x_j = x, x_k = y, x_{i+1} - x_{i} \in \range \text{ for all } j \leq i \leq k-1 \bigr\}.
	\]
	This set is empty unless $y-x \in \Dn{k-j}$.  When $j = 0$, we abbreviate the set of admissible paths from $x$ to $y$ of length $n$ by 
	\[
	\PathsPtPRes{x}{y}{n} = \PathsPtPResKM{x}{y}{0}{n}.
	\]
	
	The set of all admissible paths from $x\in\Z^d$ to $y\in\Z^d$ is 
	\[
	\PathsPtPUnr{x}{y} = \bigcup_{n =0}^\infty  \PathsPtPRes{x}{y}{n}.
	\]
	This set is empty unless $y-x \in \Rsemi$.  
	Define the space of paths from $x\in\Z^d$ that reach $y\in\Z^d$ for the first time in exactly $k\in\ZZ_{\ge0}$ steps
	\[
	\PathsPtPKilledRes{x}{y}{k} =  \bigl\{x_{0:k} \in \PathsNStep{x}{k} : x_{i} \neq y \text{ for all } 0 \leq i < k \text{ and } x_k = y\bigr\}.
	\]
	Define the space of killed paths from $x\in\Z^d$ to $y\in\Z^d$, i.e.\ paths of arbitrary length that start at $x$ and end when they reach $y$:
	\[
	\PathsPtPKilled{x}{y} = \bigcup_{k=0}^\infty \PathsPtPKilledRes{x}{y}{k}.
	\]

	The set of all semi-infinite paths rooted at $x\in\Z^d$ is denoted by 
	\[
	\PathsSemiInf{x} = \bigl\{x_{0:\infty} : x_0 = x, x_i - x_{i-1} \in \range \text{ for all } i \geq 1\bigr\}.
	\]
  $\PathsSemiInf{x}$ and the spaces above, which can be embedded in it, come equipped with the natural filtration generated by the coordinate projections. We denote the natural stochastic process on these path spaces by $X_{\aabullet}$.
 For $y \in \ZZ^d$, define the stopping time $\stoppt{y} = \inf \{k \geq 0 : X_k = y\}$.  

 All of the path spaces discussed above are compact metrizable spaces in the product-discrete topology.

When we wish to emphasize the dependence on the set $\range$ of admissible steps, we will write $\Dn{n}(\range)$, $\PathsNStep{x}{n}(\range)$, $\PathsSemiInf{x}(\range)$, and so on.

		
		Let $p : \range \rightarrow (0,1]$ be a probability kernel, i.e. $\sum_{z \in \range} p(z) = 1$.  We assume, without loss of generality, $p(z) > 0$ for all $z \in \range$.  Let $\rwP_x$ (with expectation $\rwE_x$) be the classical random walk starting at $x$ and having i.i.d.\ increments with distribution $p$. 
		This random walk is referred to as the \emph{reference walk}, and $\rwP_x$ is the  \emph{reference measure}.

    \subsection{Random walks in random potentials}\label{sec:rwrp}
	Let $(\Omega, \Sig)$ be a Polish space endowed with its Borel $\sigma$-algebra. We assume this measurable space is equipped with a commutative group of continuous bijections $\T = \{\T_x : \Omega \rightarrow \Omega : x \in \ZZ^d\}$ such that $\T_{\orig}$ is the identity map, and for all $x, y \in \ZZ^d$, $\T_x \circ \T_y = \T_y \circ \T_x = \T_{x+y}$. A generic element in $\Omega$ will be denoted by $\w$ and is referred to as an \emph{environment}.
	Let $\P$ (with expectation $\E$) be a probability measure on $(\Omega, \Sig)$, called the \emph{environment measure}. We assume $\P$ is invariant under $\T_x$ for all $x \in \ZZ^d$. 	For a subset $S \subset \ZZ^d$, we say $(\Omega,\Sig,\P,\{\T_z : z \in S\})$ is ergodic if for every event $A\in\Sig$ satisfying $\T^{-1}_z A = A$ for all $z \in S$, $\P(A) \in \{0, 1\}$. 
	

	Let $\pote:\Omega \times \range\to \R$ be a measurable function, which we will call a \emph{potential}.  For $q \geq 1$, we say $\pote \in L^q$ if $\bbE[\abs{\pote(\w,z)}^q] < \infty$ for all $z \in \range$.  

	Let $\beta$ be a positive real number, called the \emph{inverse temperature}.  
	 For $x,y\in\ZZ^d$ such that $y-x\in \Rsemi$, $n\in\Z_{>0}$, and $u,v\in\ZZ^d$ and $n \in \bbN$ such that $v-u\in\Dn{n}$,
	the \emph{restricted and unrestricted-length point-to-point partition functions} at inverse temperature $\beta$ 
	are, respectively,
	\begin{align}
		&\prtres{u}{v}{n}{\beta}(\w) = \rwE_u\Bigl[ e^{- \beta\sum_{k=0}^{n-1} \pote(\T_{X_k} \omega, X_{k+1}-X_k)}\one_{\{X_n=v\}}\Bigr]
		\quad\text{and} \label{def:prtPtPRes}\\
		&\prtunr{x}{y}{\beta}(\w)= \rwE_x\Bigl[ e^{- \beta\sum_{k=0}^{\stoppt{y}-1} \pote(\T_{X_k} \omega, X_{k+1}-X_k)}\one_{\{\stoppt{y}<\infty\}}\Bigr]. \label{def:prtPtPUnr}
	\end{align}
		Throughout the paper, depending on what notation is most convenient in a particular expression, we will at times place the dependence on $\w$ into a superscript, writing expressions like $\prtres{u}{v}{n}{\beta,\w}$ or omit the dependence on $\w$ entirely and write expressions like $\prtres{u}{v}{n}{\beta}$.
        Since $\range$ is finite, $\prtres{u}{v}{n}{\beta}(\w)$ is always finite. $\prtunr{x}{y}{\beta}(\w)$, on the other hand, can be infinite if the potential $\pote$ can take negative values and the paths can make loops. 
        When loops are present, \eqref{finiteZ} below suffices to ensure that 
        $\prtunr{x}{y}{\beta}(\w)$ is finite. 
        Define $\prtres{u}{v}{n}{\beta}(\w) = 0$ when $v-u \not\in \Dn{n}$ and $\prtunr{x}{y}{\beta}(\w) = 0$ when $y-x \not\in \Rsemi$.
		
	The corresponding \emph{restricted and unrestricted-length point-to-point free energies} are, respectively,
	\begin{align*}
		\freeres{u}{v}{n}{\beta}(\w) = \frac{1}{\beta} \log \prtres{u}{v}{n}{\beta}(\w) \quad\text{and}\quad
		\freeunr{x}{y}{\beta}(\w) = \frac{1}{\beta} \log \prtunr{x}{y}{\beta}(\w).
	\end{align*}
    If $v-u \not\in\Dn{n}$, define $\freeres{u}{v}{n}{\beta}(\w) = -\infty$. Similarly, if $y-x \not\in\Rsemi$, then $\freeunr{x}{y}{\beta}(\w) = -\infty$.

\begin{remark}
A more physically correct definition of the free energy includes a minus sign. Many mathematics papers, including many of those on which we rely in this work, omit this sign to avoid proliferating minus signs in the various computations. To maintain consistency and avoid conflicting notation with most of the earlier works we cite, we have also omitted the minus sign. 
\end{remark}

	
	The corresponding \emph{quenched point-to-point polymer measures} are the probability measures on $\PathsPtPResKM{u}{v}{0}{n}$ and $\PathsPtPKilled{x}{y}$, respectively, such that 
	\begin{align}
	\begin{split}
	    &\frac{d\Qres{u}{v}{n}{\beta, \w}}{d\rwP_u}(X_{\aabullet}) =\frac{e^{-\beta\sum_{k=0}^{n-1} \pote(\T_{X_k} \omega, X_{k+1}-X_k)}\one_{\{X_n=v\}}}{\prtres{u}{v}{n}{\beta}(\w)}
		\qquad\text{and}\\
		&\frac{d\Qunr{x}{y}{\beta, \w}}{d\rwP_x}(X_{\aabullet}) = \frac{e^{-\beta\sum_{k=0}^{\stoppt{y}-1} \pote(\T_{X_k} \omega, X_{k+1}-X_k)}\one_{\{\stoppt{y}<\infty\}}}{\prtunr{x}{y}{\beta}(\w)}.
		\end{split}\label{finitePathMeasures}
	\end{align}

	For a \emph{tilt} (or \emph{external field} or \emph{force}) $h \in \RR^d$ and an integer $n\ge1$, define the  tilted \emph{$n$-step partition function}
	\[\prtnstep{x}{n}{\beta, h}(\w)= \rwE_x\Bigl[ e^{-\beta\sum_{k=0}^{n-1} \pote(\T_{X_k} \omega, X_{k+1}-X_k) + \beta h \cdot (X_{n}-x)}\Bigr]\]
	and the corresponding free energy 
	\[
	\freenstep{x}{n}{\beta, h}(\omega) = \frac{1}{\beta} \log  \prtnstep{x}{n}{\beta, h}(\w).
	\]
 Similarly to the point-to-point case, $\prtnstep{x}{n}{\beta,h}(\w)$ is always finite, since $\range$ is finite.
    The corresponding quenched polymer measure is the probability measure on $\PathsNStep{x}{n}$ such that 
	\[\frac{d\QnStep{x}{n}{\beta, h, \w}}{d\rwP_x}(X_{\aabullet})=\frac{e^{-\beta\sum_{k=0}^{n-1} \pote(\T_{X_k} \omega, X_{k+1}-X_k) + \beta h \cdot (X_n - x)}}{\prtnstep{x}{n}{\beta,h}(\w)}\,.\]

	Through most of the manuscript, we work with the unrestricted-length model.  This is without loss of generality because a restricted-length model can be rewritten as an unrestricted-length model, as explained in the following remark.

 \begin{remark}\label{rk:resAsUnr1} (Restricted-length models as unrestricted-length models)
	Consider a restricted-length model in $\ZZ^d$ with steps $z \in \range$, random walk probability kernel $p$, and potential function $V(\w, z)$.  This model can be written as an unrestricted-length model in $\ZZ^{d+1}$ as follows.
	Write $\ZZ^{d+1}=\ZZ^d\times\Z$ and denote elements of $\ZZ^{d+1}$ by $\langle z,n\rangle$, where $z\in\ZZ^d$ and $n\in\ZZ$. The set of admissible steps of the new model is $\overline{\range} = \{ \langle z,1\rangle : z \in \range\}$.  Shifts in the $(d+1)$-st dimension are defined to be the identity shift, i.e., $\overline \T_{\langle z,n\rangle}=\T_z$. 
	The probability kernel is $\overline{p}(\langle z,1\rangle) = p(z)$, and the potential function is $\overline{\pote}(\w, \langle z,1\rangle) = \pote(\w,z)$.  
	Then the restricted-length point-to-point partition function and quenched polymer measure from $x$ to $y$ in $n$ steps are equivalent to the corresponding unrestricted-length objects from $\langle x,0\rangle$ to $\langle y,n\rangle$.   
    %
     Theorem \ref{Thm:ShapeRes} and Remarks \ref{rmk:obtainingSTCocycles} and \ref{rk:Qunr->Qres} are examples of how results for the restricted-length model can be obtained from ones for the unrestricted-length model.  Remark \ref{rk:resAsUnr2} explains how other quantities for the restricted-length model transfer to ones for the unrestricted-length model.
\end{remark}

Taking $\beta\to\infty$ in the above definitions leads to polymer models at {\it zero temperature}, where the free energy becomes a passage time.  For $x, y \in \ZZ^d$ with $y-x \in \Rsemi$ and an integer $k\ge0$, the \emph{restricted and unrestricted-length point-to-point passage times} are given by
\begin{align*}
    &\freeres{x}{y}{k}{\infty} = 
    \sup_{x_{0:k} \in \PathsPtPRes{x}{y}{k}} \sum_{i=0}^{k-1} (-\pote(\T_{x_i}\w, x_{i+1}-x_i)) \quad \text{and}\\
  &\freeunr{x}{y}{\infty} = 
  \sup_{k \in \ZZ_{\ge0}} \sup_{x_{0:k} \in \PathsPtPRes{x}{y}{k}} \sum_{i=0}^{k-1} (-\pote(\T_{x_i}\w, x_{i+1}-x_i)).
\end{align*}
Similar to the positive temperature situation, while $\freeres{x}{y}{k}{\infty}$ is always finite, 
condition \eqref{finiteZ} below suffices to ensure the finiteness of $\freeunr{x}{y}{\infty}$ if loops are allowed. 

At zero temperature, the quenched measures are replaced by (restricted and unrestricted-length) \emph{geodesics}, which are optimal paths that achieve the passage time.  Between two points, there may be multiple geodesics.  One may choose an ordering on $\range$ to order all paths lexicographically, obtaining an ordering of geodesics.  Then, for example, the first geodesic in this ordering is unique.

Next, we give examples of models that appear in the literature and are covered by our setting. 

\begin{example}[Edge and vertex weights]\label{ex:edgeVertexWeights}
We can represent random weights assigned to the vertices of $\Z^d$ by taking $\Omega = \R^{\ZZ^d}$ and $\pote(\w) = \w_\orig$.  
To represent directed edge weights, take $\Omega_0 = \R^\range$, $\Omega = \Omega_0^{\ZZ^d}$.  Then for each $x \in \ZZ^d$, $\w_x = (\w_{(x, x+z)})_{z \in \range}$ is the vector of edge weights out of $x$. The potential function is $\pote(\w, z) = \w_{(\orig,z)}$. The shifts are $(\T_v \w)_{(x, x+z)} = \w_{(x+v, x+v+z)}$ for $v \in \Z^d$. 

For undirected edge weights, assume $\orig\not\in\range$ and define the set of edges $\sE = \{ \{x,y\} \subset \ZZ^d : y-x \in \range \text{ or }x-y\in\range\}$.  Let $\Omega = \R^{\sE}$ and $\pote(\w, z) = \w_{\{\orig, z\}}$. The shifts are $(\T_v \w)_{\{x, x+z\}} = \w_{\{x+v, x+v+z\}}$ for $v \in \Z^d$. 

One can also have a hybrid of all three versions by letting $\Omega=\R^{\Z^d}\times(\R^\range)^{\Z^d}\times\R^{\sE}$. For each $x\in\Z^d$, the weights at $x$ are given by $\w_x=\{\alpha_x,\gamma_{x,x+z},\delta_{\{x,x+z'\}}:z,z'\in\range\}$ and for $z\in\range$, $\pote(\w,z)=\alpha_\orig+\gamma_{\orig,z}+\delta_{\{\orig,z\}}$.
\end{example}

\begin{example}[Product environment]\label{ex:productEnv}
Let $\Omega_0$ be Polish and $\Omega = \Omega_0^{\Z^d}$ equipped with the product topology and Borel $\sigma$-algebra.  We denote points $\w = (\w_x)_{x \in \Z^d}$.  The translations are $(\T_x \w)_y = \w_{x+y}$.  $\bbP$ is a product measure if $\{\w_x\}_{x \in \Z^d}$ are i.i.d.\ random variables under $\bbP$.  $\bbP$ has a \emph{finite range of dependence} if there exists an $r_0 \geq 0$ such that for any $A, B \subset \ZZ^d$ where $\abs{x-y}_1 \geq r_0$ for all $x \in A$ and $y \in B$, $\{\w_x : x \in A\}$ and $\{\w_y : y \in B\}$ are independent.  $\pote$ is \emph{local} if there exists an $L \geq 0$ such that for all $z \in \range$, $\pote(\w,z)$ is measurable with respect to $\sigma\{\w_x : \abs{x}_1 \leq L \}$.  In other words, $\pote$ only depends on finitely many coordinates $\w_x$.
\end{example}

	     \begin{example}\label{examples}
	     Our general setting covers the following models. 
	     
      \begin{enumerate}[label=\rm(\arabic{*}), ref=\rm\arabic{*}]
      
        \item \label{itm:randomGrowth} \emph{First- and last-passage percolation.} At zero temperature, directed \emph{last-passage percolation} is obtained by taking $\range = \{e_1, \ldots, e_d\}$, $\Omega=\R^{\Z^d}$,  setting $\pote(\w,z)=-\w_0$, and taking $\beta = \infty$. 
        The last-passage time is
          \[
          \freeunr{x}{y}{\infty} 
          = \sup_{k \in \ZZ} \sup_{x_{0:k} \in \PathsPtPRes{x}{y}{k}} \sum_{i=0}^{k-1} \w_{x_i}.\]
          Standard \emph{first-passage percolation} is obtained by taking $\range = \{\pm e_1, \ldots, \pm e_d\}$, putting weights on the edges as explained in Example \ref{ex:edgeVertexWeights}, 
          then 
          taking $\beta = \infty$ and putting a minus sign in front of the free energy.
          The first-passage time is
          \begin{align*}
          -\freeunr{x}{y}{\infty} = 
          \inf_{k \in \ZZ} \inf_{x_{0:k} \in \PathsPtPRes{x}{y}{k}} \sum_{i=0}^{k-1} \w_{\{x_i,x_{i+1}\}}.
          \end{align*}
        \item \label{itm:directedPolymers} \emph{Directed polymers in random environments (DPRE).}  This is the case where $\beta<\infty$ and there exists a vector $\uhat\in\R^d$ with $z\cdot\uhat>0$ for all $z\in\range$ or, equivalently, where $0$ does not lie in the convex hull of $\range$.
        Thus, the polymer always moves strictly forward in  direction $\uhat$;
        by \cite[Corollary A.2]{Ras-Sep-Yil-13}, this is equivalent to the condition that there are no loops, i.e., $\PathsPtPUnr{\orig}{\orig} = \{\orig\}$. A commonly studied special case, which we call a \textit{space-time directed polymer}, is one where there exists a vector $\uhat\in\R^d$ such that $z\cdot\uhat=1$ for all $z\in\range$. Then $(X_n-X_0)\cdot\uhat=n$ plays the role of a time coordinate. There is a large literature studying this model; see \cite{Com-17} and the references therein. By Lemma \ref{lm:z.uhat=1}, this is satisfied if and only if, for any two points $x$ and $y$ in $\Z^d$ such that $y-x\in\Rsemi$, all paths from $x$ to $y$ take the same number of steps.     A common example of this is when $\range = \{e_1, e_2, \ldots, e_d\}$.  

          \item \label{itm:RWRE} \emph{Random walk in a random environment (RWRE).} Let $\range$ be general and $\beta = 1$.  Let $\pi_z:\Omega\to[0,1]$, $z \in \range$, be measurable such that $\sum_{z \in \range} \pi_z = 1$, $\P$-almost surely.   Let $\pote(\w, z) = -\log \pi_z(\w) + \log p(z)$.  Then $\prtnstep{x}{n}{1, \w} = 1$ and the measures $\{\QnStep{x}{n}{1,\orig,\w} : n \in \mathbb{N}\}$ are consistent, being the marginal distributions of the Markov chain $X_n$ starting at $x$ which uses transition probabilities $\pi_{y,y+z} = \pi_z(\T_y\w)$. Under the path measure this family induces by Kolmogorov's extension theorem, $\prtunr{x}{y}{\beta}(\w)$ is the probability that the path ever reaches $y$ if started at $x$ and $\Qunr{x}{y}{1,\orig, \w}$ is the distribution of this Markov chain conditioned on reaching $y$, provided that this probability is positive.
          A space-time directed polymer which is also a random walk in a random environment is called a \emph{random walk in a dynamic random environment}.

          An RWRE is called \textit{elliptic} if $\pi_{z}>0$ for all $z\in \range$ and \textit{uniformly elliptic} if there exists a deterministic $c>0$ such that $\pi_{z}\geq c$ almost surely for all $z\in \range.$ If the model is not elliptic, we say it is \textit{degenerate}; such models are not covered by the results of this paper. Note that uniform ellipticity is equivalent to the potential $\pote$ being bounded. 

          An RWRE is called \textit{nestling} if $\orig$ lies in the relative interior of the convex hull of the support of the random vector $\sum_{z\in\range} \pi_{z}z$, \textit{marginally nestling} if $\orig$ lies on the relative boundary of this convex hull and \textit{non-nestling} if $\orig$ does not lie in the convex hull. This concept features in some of the discussion above and in the discussion below. 
        \item\label{itm:stretchedPolymers} \emph{Stretched polymers.} The term $h\cdot(X_n-x)$ in the definition of $\QnStep{x}{n}{\beta, h, \w}$ can be interpreted as an external force that stretches the polymer. A special case that appears in the literature is when $\range = \{\pm e_1, \pm e_2, \ldots, \pm e_d\}$, $\Omega = \R^{\Z^d}$, and $\pote(\w,z) = \w_0$. This model is studied in \cite{Iof-Vel-12}, for example.
        
        \item \label{itm:killedPolymers} \emph{Killed polymers.}  Let $\beta<\infty$ and consider an unrestricted-length model with $\Omega = \R^{\Z^d}$, $\range = \{\pm e_1, \pm e_2, \ldots, \pm e_d\}$,  and $\pote(\w,z) = \w_0$. The Green's function $\greens(x,y)$, defined in \eqref{g:def}, can be interpreted as the expected number of visits of the reference random walk to the site $y$ before it is ``killed'' by the potential. See, for example, \cite[p. 249]{Zer-98-aap}.  \qedhere
  \end{enumerate}
\end{example}


\subsection{Limiting free energies}\label{sub:FE}

Before we state our main result, we need to record a few limit theorems for the various free energies defined in Section \ref{sec:rwrp}. These results require the following definition.

\begin{definition}\label{def:cL} For a finite subset $\range\subset\Z^d$ and $z\in\range\setminus\{\orig\}$, 
a non-negative measurable function $g : \Omega \to \R$ is in class $\classL_{z,\range}$ if 
\[
\varlimsup_{\varepsilon \searrow 0} \varlimsup_{n \rightarrow \infty} \max_{\substack{x \in \Rsemi(\range)\\ \abs{x}_1\le n}} n^{-1} \sum_{0 \leq k \leq \varepsilon n} g(\T_{x+kz}\w) = 0 \quad \text{for } \bbP\text{-a.e. } \w.
\] 
\end{definition}

Membership in class $\classL_{z,\range}$ depends on a trade-off between the moments of $g$ and the degree of mixing of $\bbP$. For example, if $(\bbP,\T_z)$ is exponentially mixing, then $g \in L^q(\bbP)$ for some $q > d$ is sufficient.  For a general $\bbP$, if $g$ is bounded, then $g \in \classL_{z,\range}$. See \cite[Lemma A.4]{Ras-Sep-Yil-13} for a precise formulation of sufficient conditions to ensure $g \in \classL_{z,\range}$.

     For inverse temperature $\beta\in(0,\infty]$ and velocity $\xi \in \Uset$ the \emph{restricted-length limiting quenched point-to-point free energy} is defined by
	\begin{align}
	\ppShapeRes{\beta}(\xi) = \lim_{n \rightarrow \infty} n^{-1}  \freeres{\orig}{\widehat{x}_n(\xi)}{n}{\beta}, \label{ppShapeRes}
	\end{align}
	where $\widehat{x}_n(\xi) \in \Dn{n}$ is a lattice point approximation of $n\xi$ as defined in \cite[(2.1)]{Ras-Sep-14}.

	It was shown in Theorem 2.2 in \cite{Ras-Sep-14} and Theorem 2.4 in  \cite{Geo-Ras-Sep-16} that the limit \eqref{ppShapeRes} exists in $(-\infty,\infty]$, $\bbP$-almost surely, for all $\xi \in \Uset$ simultaneously. By  \cite[Theorem 2.6]{Ras-Sep-14} and the inequalities in \cite[(2.11)]{Geo-Ras-Sep-16}, $\ppShapeRes{\beta}(\xi)$ is either identically infinite for all $\xi\in\ri\Uset$ and $\beta\in(0,\infty]$ or $\ppShapeRes{\beta}$ is bounded on $\Uset$, for each $\beta\in(0,\infty]$. 
    In the latter case,
    \cite[Theorem 2.6]{Ras-Sep-14} and \cite[Remark 2.5]{Geo-Ras-Sep-16} state that with $\P$-probability one, $\ppShapeRes{\beta}$ is lower semicontinuous 
    on $\Uset$ and concave and continuous on $\ri \Uset$.
    Furthermore, the upper semicontinuous regularization of $\ppShapeRes{\beta}$ and its unique continuous
extension from $\ri\Uset$ to $\Uset$ are equal.  
The conditions of Theorem \ref{Thm:ShapeRes} below ensure that $\ppShapeRes{\beta}$ is finite on the relative interior of the face being considered in that result.

	\begin{remark}\label{rk-cL}
	   In  \cite{Geo-Ras-Sep-16} and \cite{Ras-Sep-14}, the authors assume that $\bbP$ is ergodic, but the results we mention above continue to hold for a stationary $\bbP$ by the ergodic decomposition theorem.
	   The authors of these papers also assume that for each $z'\in\range$, $\abs{V(\w,z')}\in L^1(\P)$ and belongs to $\classL_{z,\range}$ for all $z\in\range\setminus{\{\orig\}}$. However, upon closer examination of their proofs, it can be observed that the condition on $\pote$ in Theorem \ref{Thm:ShapeRes} below is sufficient for their results to hold. This can be seen, for instance, in the proof of Lemma \ref{lem:upperBoundCompact}, which demonstrates how the arguments can be adjusted to work under the weaker assumptions.
	  \end{remark}
	  
	  \begin{remark}\label{shape:deterministic}
	     In general, $\ppShapeRes{\beta}(\xi)$ may not be deterministic.
	     For example, if $\orig$ is an extreme point of $\Uset$, then 
	     $\ppShapeRes{\beta}(\orig)=-\pote(\w,\orig)+\beta^{-1}\log p(\orig)$.
	     A sufficient condition to ensure that it is deterministic is to have the ergodicity of $\bigl(\Omega, \Sig^V_{\range'},\bbP,\{\T_z:z\in\range'\}\bigr)$, where $\range'=\Uset'\cap\range$, $\Uset'$ is the unique face of $\Uset$ such that $\xi\in\ri\Uset'$,   $\Sig^\pote_{\range'}=\sigma(\pote(\T_x\w,z):x\in\Rgroup',z\in\range')$, and $\Rgroup'$ is the additive group generated by $\range'$. 
	  \end{remark}

The following example provides a simple model where $\ppShapeRes{1}$ can be computed via simulations. We will also refer to this example throughout the paper to illustrate key results and definitions.

\begin{example}[Nearest-neighbor RWRE on $\bbZ$]\label{ex:nRWRE}
Let $p(1)=p(-1)=1/2$ so that $\range=\{-1,1\}$ and $\Uset=[-1,1]$. Consider the i.i.d.~RWRE model as in Example \ref{examples}\eqref{itm:RWRE}.  Greven and den Hollander \cite{Gre-Hol-94} gave a detailed description of $-\ppShapeRes{1}$, which is the large deviation rate function for the position of the particle. They demonstrated that this function is always differentiable away from $0$ but may have a corner at $0$. It also may have two linear pieces around the corner at $0$ but is strictly convex otherwise.  Figure \ref{fig:RWRE} shows the results of simulations for the model with $\pi_1=1-\pi_{-1}=a_1$ with probability $q$ and $\pi_1=1-\pi_{-1}=a_2$ with probability $1-q$, for the various regimes of the parameters $q,a_1,a_2\in(0,1)$. $N$, in the caption of the figure, is the number of steps we used for the random walk in the simulation used to approximate $-\ppShapeRes{1}$, which is plotted in blue. The unrestricted-length rate function $-\ppShapeUnr{1}$, defined in \eqref{ppShapeUnr} below, is plotted in red. See 
Remark \ref{rk:RWRE_unrShape} for the details involved in the computation of $-\ppShapeUnr{1}$. 
\end{example}

\begin{figure}
\centering
\begin{subfigure}[b]{0.45\textwidth}
    \centering
    \includegraphics[width=\textwidth]{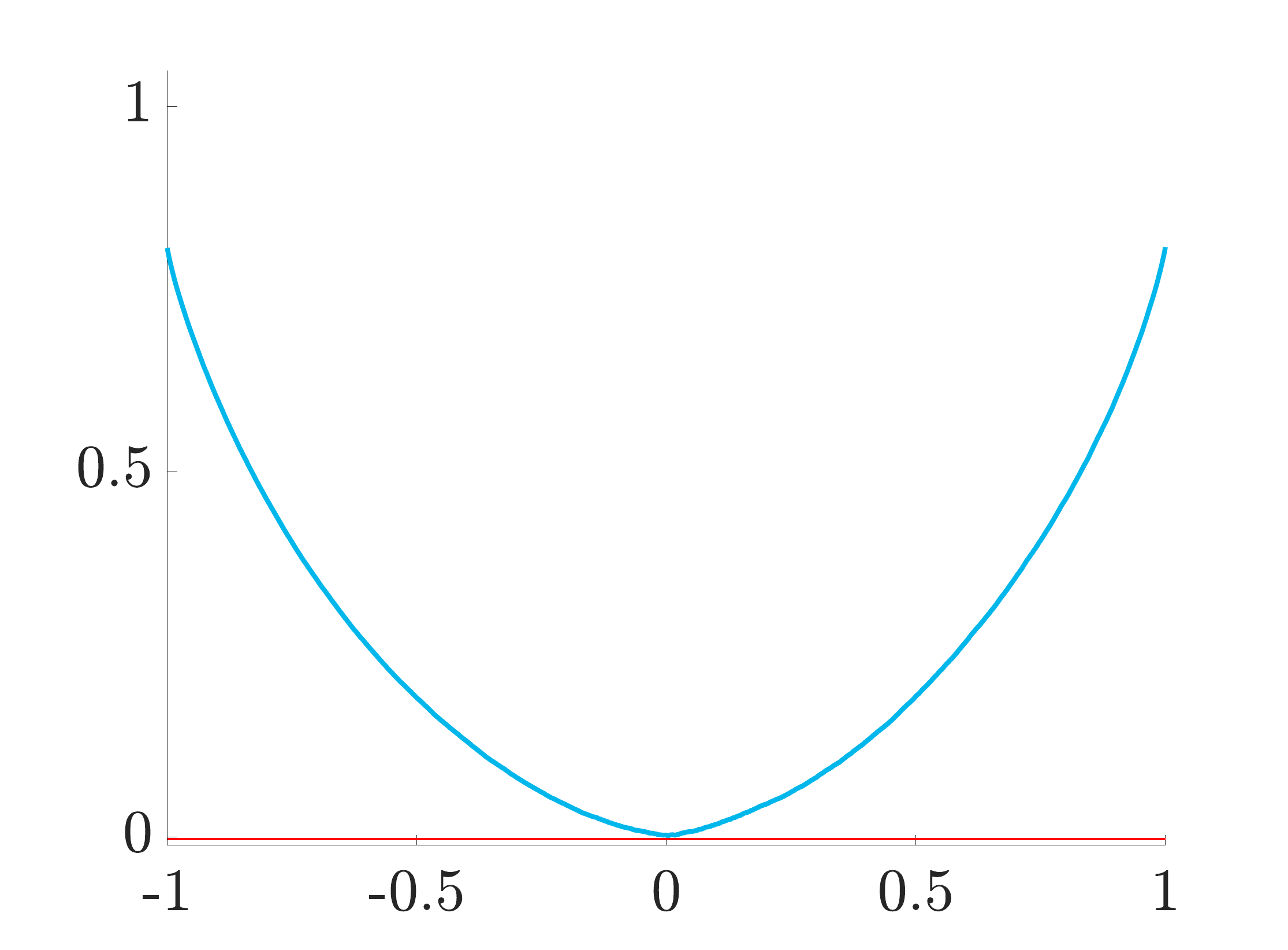}
    \subcaption{(Recurrent) $q=1/3$, $a_1=1/5$, $a_2=2/3$, $N=2,000$}\label{fig:RWRErec}
\end{subfigure}
    \hfill
\begin{subfigure}[b]{0.45\textwidth}
\centering
\includegraphics[width=\textwidth]{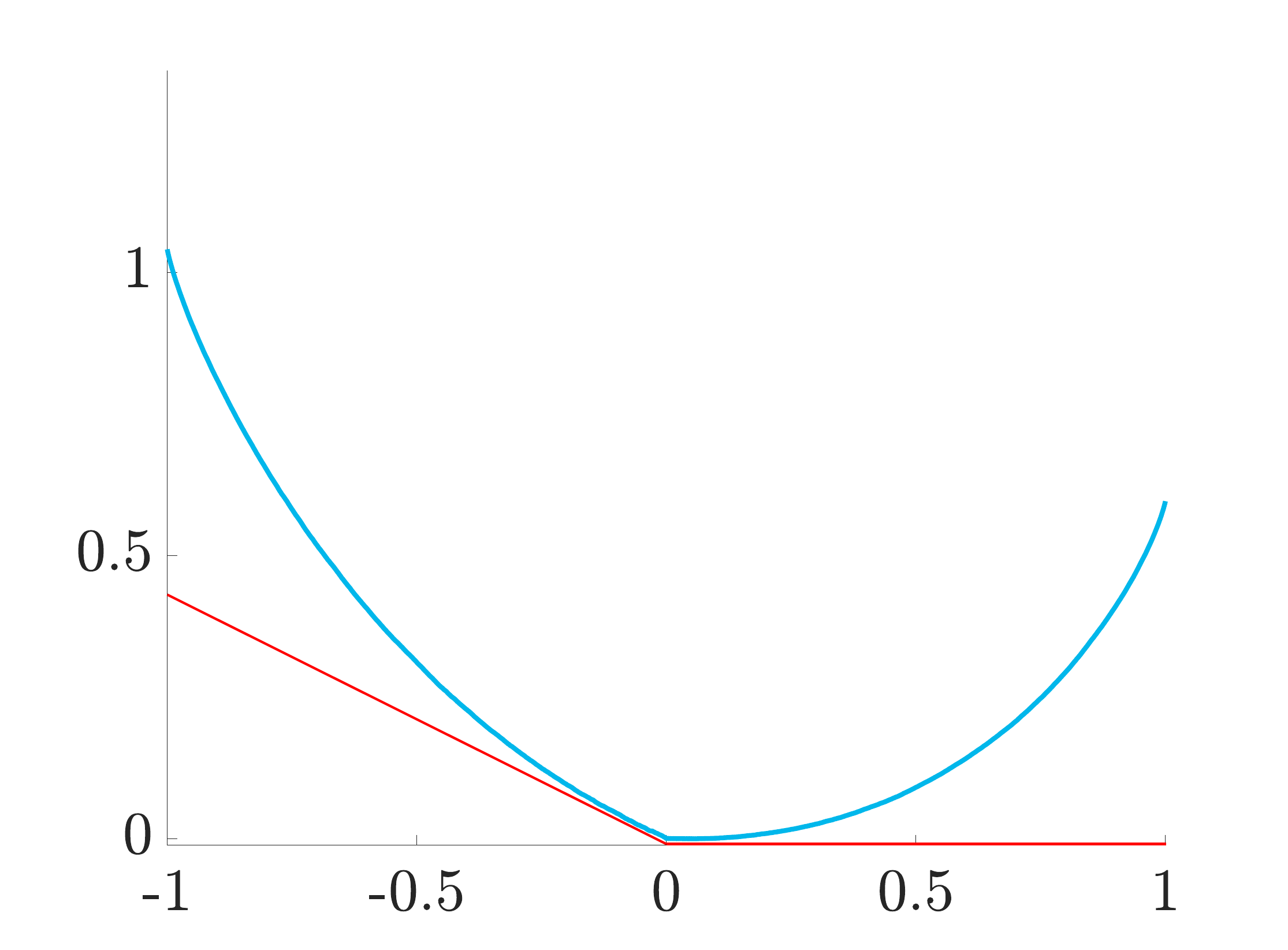}
        \subcaption{(Transient with $0$ velocity) $q=2/3$, $a_1=3/4$, $a_2=3/10$, $N=20,000$}\label{fig:RWREtv0}
\end{subfigure}\\
\begin{subfigure}[b]{0.45\textwidth}
    \centering
    \includegraphics[width=\textwidth]{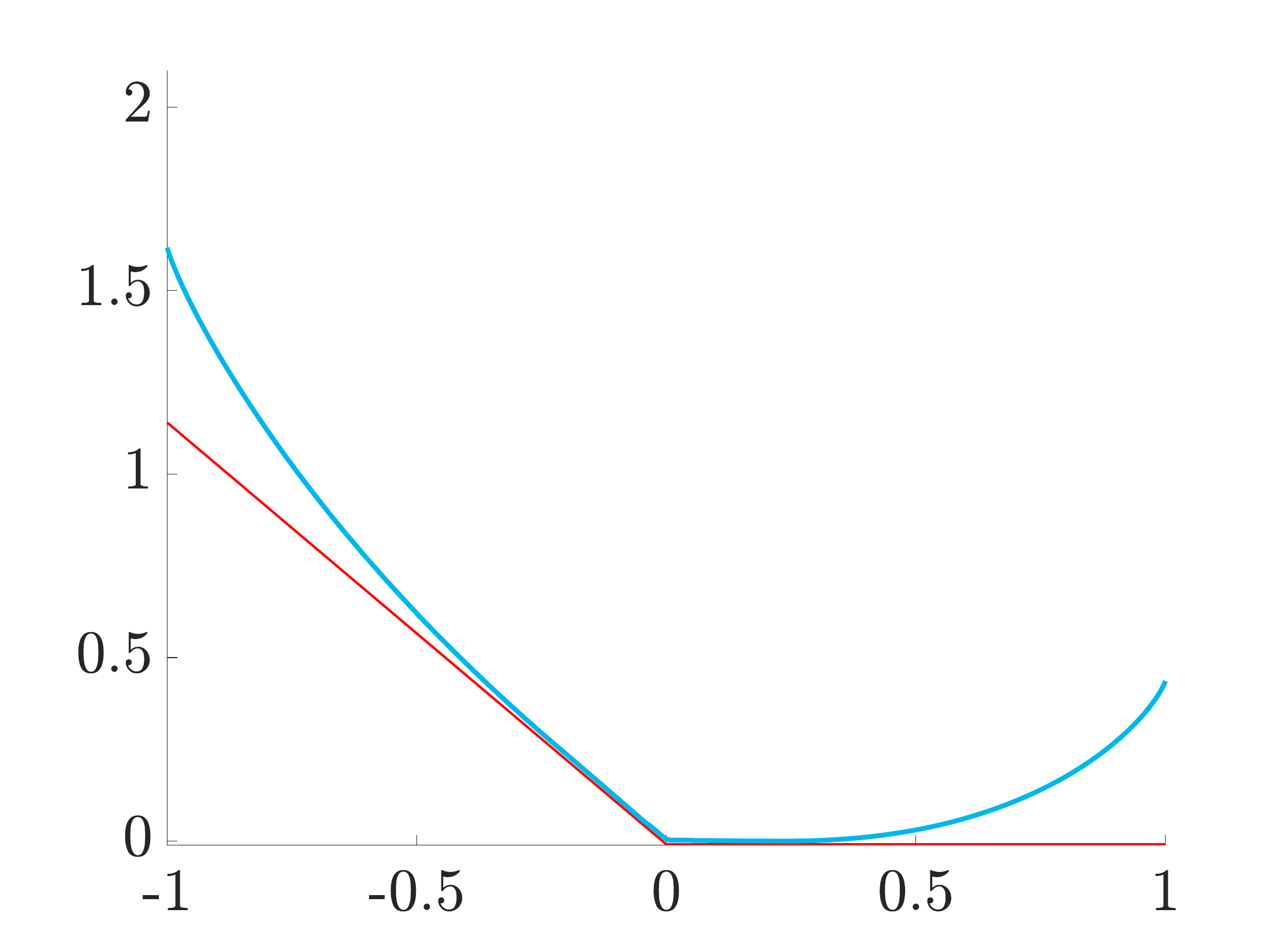}
    \subcaption{(Marginal nestling) $q=3/5$, $a_1=1/2$, $a_2=95/100$, $N=1000,000$}\label{fig:RWREmnest}
\end{subfigure}
    \hfill
\begin{subfigure}[b]{0.45\textwidth}
\centering
\includegraphics[width=\textwidth]{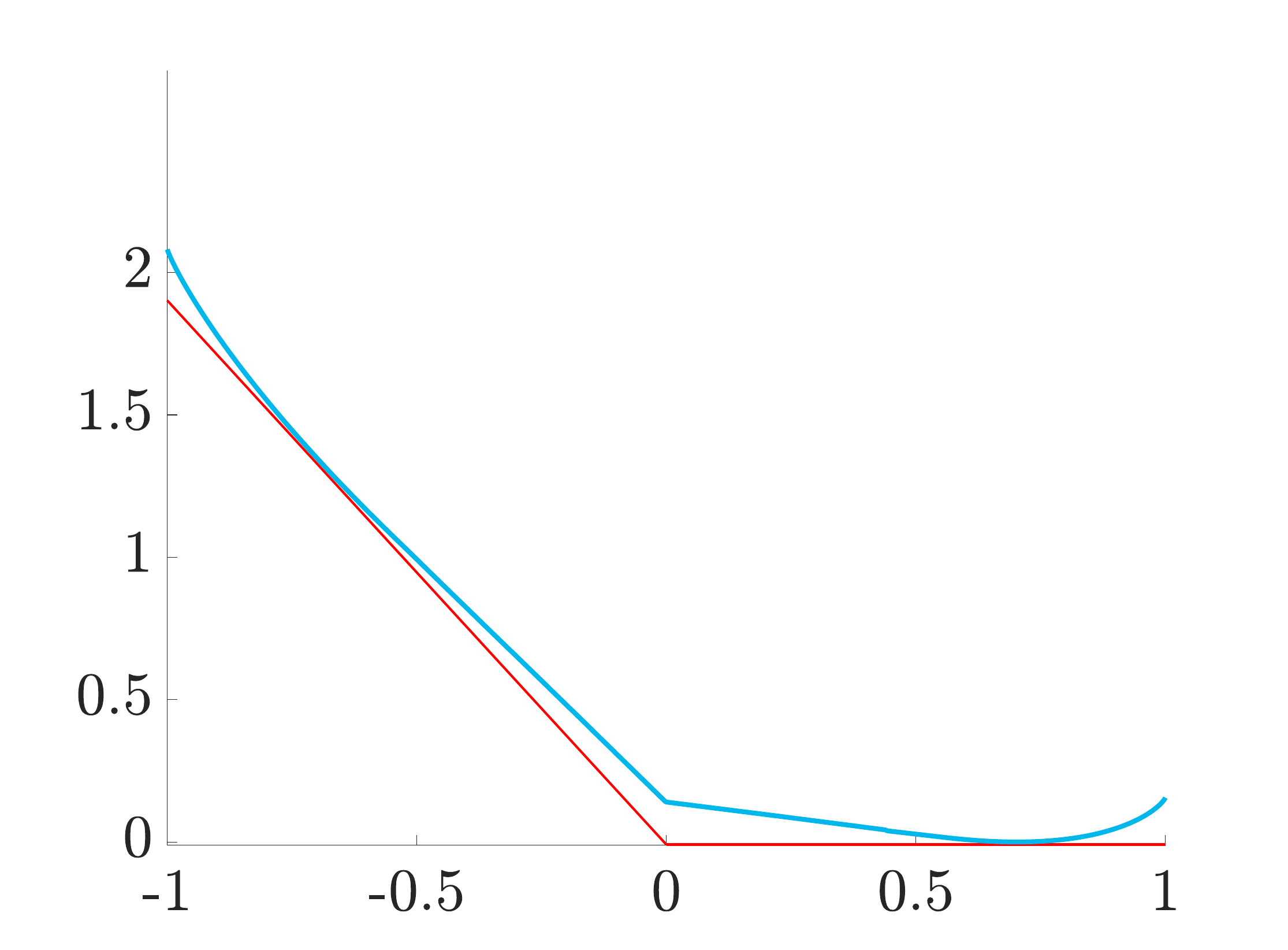}
        \subcaption{(Non-nestling) $q=4/5$, $a_1=9/10$, $a_2=7/10$, $N=200,000$}\label{fig:RWREnnest}
\end{subfigure}
\caption{The thicker blue curves show plots of the approximation of $-\ppShapeRes{1}(t)$ from Example \ref{ex:nRWRE} by $-N^{-1}\freeres{\orig}{\lfloor Nt\rfloor}{N}{1}$ for $t\in[-1,1]$ in simulations. In the recurrent example in Figure \ref{fig:RWRErec}, $-\ppShapeRes{1}$ is differentiable and strictly convex. In the transient example with $0$ velocity in Figure \ref{fig:RWREtv0}, $-\ppShapeRes{1}$ is strictly convex and differentiable except at $0$. In the ballistic marginal nestling example in Figure \ref{fig:RWREmnest}, $-\ppShapeRes{1}$ has a flat segment and a linear segment, meeting at a corner at $0$. Elsewhere, it is differentiable and strictly convex. (The ballistic nestling case looks qualitatively the same.) In the non-nestling example in Figure \ref{fig:RWREnnest}, $-\ppShapeRes{1}$ has two linear segments that meet at a corner at $0$. Elsewhere, it is differentiable and strictly convex. The thinner red piecewise linear graphs show plots of $-\ppShapeUnr{1}(t)$ for $t\in[-1,1]$. In all cases, $\ppShapeUnr{1}(t)=0$ for $t\ge0$. In Figure \ref{fig:RWRErec}, $\ppShapeUnr{1}(t)=0$ for $t\le0$ as well. In Figures \ref{fig:RWREtv0} and \ref{fig:RWREmnest}, the slope of $-\ppShapeUnr{1}$ on $t<0$ matches the slope of the linear segment of  $-\ppShapeRes{1}$ left of $0$. In Figure \ref{fig:RWREnnest}, $-\ppShapeUnr{1}$ is tangent to $-\ppShapeRes{1}$ at a unique point in $(-1,0)$.}\label{fig:RWRE}
\end{figure}
	
%

Define the cone 
	    \[\cone = \Bigl\{\sum_{z \in \range} b_z z : b_z \in \Rnonneg\Bigr\}.\]    
  For a face $\sA$ of $\cone$, let 
	$\rangeA = \range \cap \sA$, $\UsetA = \Uset \cap \sA$, and let 
	$\RgroupA$ and $\RsemiA$ denote, respectively, the group and the semigroup generated by $\rangeA$. For the definition of a face of a convex set, see Appendix \ref{app:conv}.

Define 
	\begin{align}\label{Rid}
	\rangeA^{\id}=\{z\in\rangeA:\T_z\text{ is the identity map}\}
	\end{align}
 to be the set of steps for which the associated shift is the identity. If $\orig \in \rangeA$, then $\orig \in \rangeA^{\id}$. Non-zero steps can also be in this set. As an example, in the restricted-length setting of Remark \ref{rk:resAsUnr1}, if $\orig$ is an admissible step, $\overline \T_{\langle \orig,1\rangle}$ is the identity. It may also be the case that $\rangeA^{\id}=\varnothing$. 
 Let $\faces$ be the set of faces $\sA$ of $\cone$ such that $\rangeA\ne\rangeA^{\id}$.  Note that if $\orig$ is an extreme point of $\Uset$, then the face $\{\orig\}$ is not in $\faces$.
 
 Given $\sA \in \faces$, $\xi \in \ri \sA$, and an inverse temperature $\beta\in(0,\infty]$, the \emph{unrestricted-length limiting quenched point-to-point free energy} is defined by 
	\begin{align}\label{ppShapeUnr}
     \ppShapeUnr{\beta}(\xi) = \lim_{n \rightarrow \infty} n^{-1} \freeunr{\orig}{x_n}{\beta},
	\end{align}
    where the sequence $x_n \in \RsemiA$ satisfies $x_n / n \rightarrow \xi$. 
	 In \cite[Theorem 2.8]{Jan-Nur-Ras-22} it is shown that  
	 the limit exists $\bbP$-almost surely for all $\xi \in \cone$ simultaneously, and is finite, positively $1$-homogeneous, and lower semicontinuous on $\cone$, and concave and continuous on $\ri \cone$.  
	 As in Remark \ref{rk-cL}, although \cite{Jan-Nur-Ras-22} assumes stronger conditions on class $\sL$ membership, minor adjustments of their proofs (along the lines of the proof of our Lemma \ref{lem:upperBoundCompact}) give the same results under the assumptions of Theorem \ref{Thm:ShapeUnr} below. 
	 
	 \begin{remark}
	 Similarly to the restricted-length case, a sufficient condition to guarantee that $\ppShapeUnr{\beta}(\xi)$ is deterministic is to have the ergodicity of $\bigl(\Omega, \Sig^\pote_{\range'},\bbP,\{\T_z : z \in \range' \}\bigr)$, where 
	 $\range'=\sA\cap\range$, $\sA\in\faces$ is the unique face of $\cone$ such that $\xi\in\ri\sA$, and $\Sig^\pote_{\range'}$ is as in Remark \ref{shape:deterministic}.  
	 \end{remark}

\begin{remark}[RWRE on $\Z$]\label{rk:RWRE_unrShape}
We sketch the computation of $\ppShapeUnr{\beta}$ in the case of the one-dimensional RWRE models from Example \ref{ex:nRWRE} which were discussed in Figure \ref{fig:RWRE}. Here, $\cone=\R$ and it is the only face. 


In this model, $\prtunr{0}{x}{1,\w}$ represents the probability that the RWRE, starting at $0$, eventually reaches $x$. In the recurrent case (Figure \ref{fig:RWRErec}), $\prtunr{0}{x}{1,\w} = 1$ almost surely for all $x \in \Z$, which implies that $\ppShapeUnr{1}$ is identically $0$. Similarly, in the other (transient-to-the-right) cases, $\ppShapeUnr{1}(t) = 0$ for all $t \geq 0$. By homogeneity, $\ppShapeUnr{1}(t) = -t\ppShapeUnr{1}(-1)$ for all $t < 0$. Using the results in \cite{Yil-09-cpam}, we identify $-\ppShapeUnr{1}(-1)$, which corresponds to $\bar\lambda(0)$ in the notation of that paper.

In that work, the rate function $I$ corresponds to our $-\ppShapeRes{1}$, and $\bar\lambda(0)$ there matches $-\ppShapeUnr{1}(-1)$ in our framework. For the transient nestling cases (Figures \ref{fig:RWREtv0} and \ref{fig:RWREmnest}), Theorem 1.8 of \cite{Yil-09-cpam} shows that, in the notation of that paper, $r_c = I(0) = -\ppShapeRes{1}(0) = 0$. The arguments in the proof of that theorem (page 1052 of the paper) further establish that the left derivative of $-\ppShapeRes{1}$ at $0$ equals $\bar\lambda(0) = -\ppShapeUnr{1}(-1)$. This explains the behavior depicted in Figures \ref{fig:RWREtv0} and \ref{fig:RWREmnest}.

In the non-nestling case (Figure \ref{fig:RWREnnest}), we have $r_c = I(0) = -\ppShapeRes{1}(0) > 0$, and there exists a point $\bar\xi_c \in (-1, 0)$ such that $r(\bar\xi_c) = r_c$, where $r(\xi)$ is defined in the middle of page 1050 in \cite{Yil-09-cpam}, substituting $\bar\lambda$ for $\lambda$. By \cite[Theorem 1.8]{Yil-09-cpam}, $\bar\xi_c$ is the left endpoint of the interval $[\bar\xi_c, 0]$ where $I = -\ppShapeRes{1}$ is linear, with the derivative of $-\ppShapeRes{1} = I$ at $\bar\xi_c$ matching its left derivative at $0$. Then, referring again to the formula on page 1052 in \cite{Yil-09-cpam}, we get that the left derivative of $-\ppShapeRes{1} = I$ at $0$ equals $\bar\lambda(r(\bar\xi_c)) = \bar\lambda(I(0)) > \bar\lambda(0)$. In particular, $-\ppShapeUnr{1}(-1) < 0$, and the slope of the line representing $-\ppShapeUnr{1}$ on $(-\infty, 0]$ is steeper than that of the linear segment of $-\ppShapeRes{1}$ to the left of $0$. 
Furthermore, a little bit of calculus, using the formula (1.8) in \cite{Yil-09-cpam}, shows that $-\ppShapeUnr{1}$ and $-\ppShapeRes{1}$ intersect at a unique point in $(-1,\bar\xi_c)$. 
This explains the behavior illustrated in Figure \ref{fig:RWREnnest}.  See Remark \ref{ex2:nRWRE} for more.
%
%
\end{remark}
     
The above limits can be strengthened to uniform limits, commonly referred to as ``shape theorems.'' The shape theorem for the unrestricted-length model is established in \cite[Theorem 3.10]{Jan-Nur-Ras-22}, where the uniform limit holds within the interior of each face. Theorem \ref{Thm:ShapeUnr} extends this result to the entire face, while Theorem \ref{Thm:ShapeRes} presents the shape theorem for the restricted-length model. Detailed proofs are provided in Appendix \ref{sec:shapeThms}.
A key observation is that the shape theorem is stated face-by-face, with the proper centering achieved through the upper-semicontinuous regularization of the limiting shape restricted to the interior of each face. The restriction to faces arises because paths starting at the origin and terminating at a specific point remain within the face containing the endpoint in its relative interior (see Lemma \ref{lem:pathrem}). The primary technical point, however, is the necessity of the upper-semicontinuous regularization. This stems from the fact that the continuity of the limiting shape up to the boundaries of the faces is not generally known. Thus, the correct centering requires using the unique continuous extension from the interior of the face to its boundary, which is precisely captured by the upper-semicontinuous regularization.
	
	Recall the upper semicontinuous regularization of a function $f : \sX \rightarrow [-\infty, \infty)$  defined as
    \[
    f^{\usc}(x) = \inf \Bigl\{ \sup_{y \in G} f(y) : G \ni x \text{ and } G \text{ is open}\Bigr\}.
    \] 

Given a face $\sA$ of $\cone$, let $\ppShapeUnrA{\beta}$ be the function that equals $\ppShapeUnr{\beta}$ on $\sA$ and equals $-\infty$ outside of $\sA$. 
Denote the upper semicontinuous regularization of $\ppShapeUnrA{\beta}$ by $\ppShapeUnrA{\beta, \usc}$. Then, under the conditions of Theorem \ref{Thm:ShapeUnr} below, the restriction of  $\ppShapeUnrA{\beta, \usc}$ to $\sA$ is the unique continuous extension of $\ppShapeUnrA{\beta}$ from $\ri\sA$ to all of $\sA$. In particular, $\ppShapeUnrA{\beta, \usc}$ is concave and positively $1$-homogeneous.

\begin{remark}
The shape $\ppShapeUnr{\beta}$ is expected to be continuous on all of $\cone$ in a wide class of natural models.  
In such cases, the restrictions and upper semi-continuous extensions discussed above can be dispensed with. The example following this remark gives examples where this continuity is known to hold.

In the example in Remark \ref{shape:deterministic}, if the distribution of a zero step is continuous, then $\ppShapeUnr{\beta}$ cannot be continuous at zero and so considering passage times seen from the relative interior of a face is necessary.
\end{remark}
\begin{example}[i.i.d.~directed LPP and RWRP] \label{ex:LPPpoly}
In i.i.d.~directed LPP on $\Z^d$ with $\range=\{e_1,e_2, \dots, e_d\}$ (discussed in Example \ref{examples}\eqref{itm:randomGrowth}) under a moment hypothesis slightly weaker than the weights having $>d$ moments, it is shown in \cite{Mar-04} that the shape function 
$\ppShapeUnr{\infty}$ is continuous on all of $\cone = \bbR^d_+ = \{(x_1,\dots, x_d) : x_1,\dots, x_d \geq 0\}$. 

Continuity of $\ppShapeRes{1}$ on $\Uset$ is proven for a wide collection of RWRP models in \cite[Theorem 3.2]{Ras-Sep-14}.
\end{example}

We will reference the next example frequently in the discussion that follows to illustrate our results and the definitions.

\begin{example}[FPP on $\bbZ^3$ with a forbidden step]\label{ex:FPPforb} Consider the standard FPP on $\Z^3$ with i.i.d.\ Exponential(1) weights, but where paths are not allowed to use the $-e_3$ step. We can model this as a zero temperature model with hybrid edge weights as in Example \ref{ex:edgeVertexWeights}. In the notation of that example, let $\range = \{e_1, -e_1, e_2, -e_2, e_3\}$. For $x\in\Z^3$ and $z\in\{-e_1,e_1,-e_2,e_2\}$ let $\alpha_x=\gamma_{x,x+z}=\delta_{\{x,x+e_3\}}=0$ and let $\{\gamma_{x,x+e_3},\delta_{\{y,y+e_1\}},\delta_{\{u,u+e_2\}}:x,y,u\in\Z^d\}$ be i.i.d.\ Exponential(1) random variables.  The first-passage time is given by $-\freeunr{x}{y}{\infty}$.
Then $\cone = \{(x,y,z) \in \bbR^3 : z \geq 0\}$ has two non-empty faces: $\cone$ and its boundary, $\partial \cone = \{(x,y,0) : x,y \in \bbR\}$. Restricted to $\sA=\partial \cone$, this model becomes the classical planar standard FPP with i.i.d.\ Exponential(1) edge-weights and the upper semi-continuous regularization of $-\ppShapeUnrA{\infty}$ is the usual shape function of that planar FPP. See Figure \ref{fig:FPPforbshape} for simulations of the limit shape in this model.
\end{example}

\begin{figure}[h!]
        \centering
    \begin{subfigure}[b]{0.485\textwidth}
        \centering
        \includegraphics[width=3.in,height=3in]{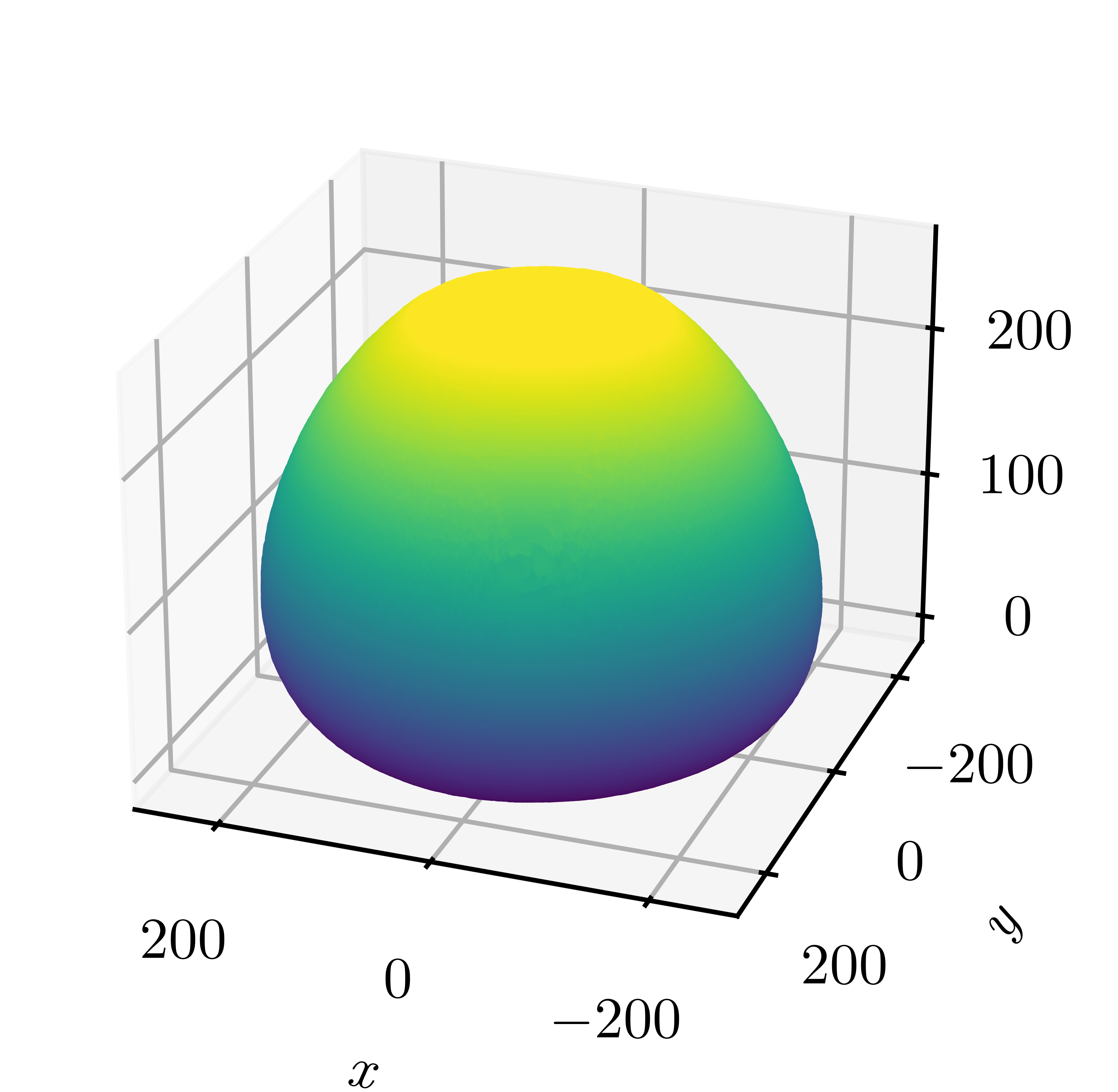}
        \caption{Level 75 sub-level set of an approximation of $-\ppShapeUnrA{\infty}$, $\sA = \cone = \{(x,y,z)\in\bbR^3 :  z \geq 0\}$. Darker colors represent points with lower $z$ coordinates.}\label{fig:FPP3D}
    \end{subfigure}
    \hfill
    \begin{subfigure}[b]{0.485\textwidth}
        \centering
        \includegraphics[width=3in,height=3in]{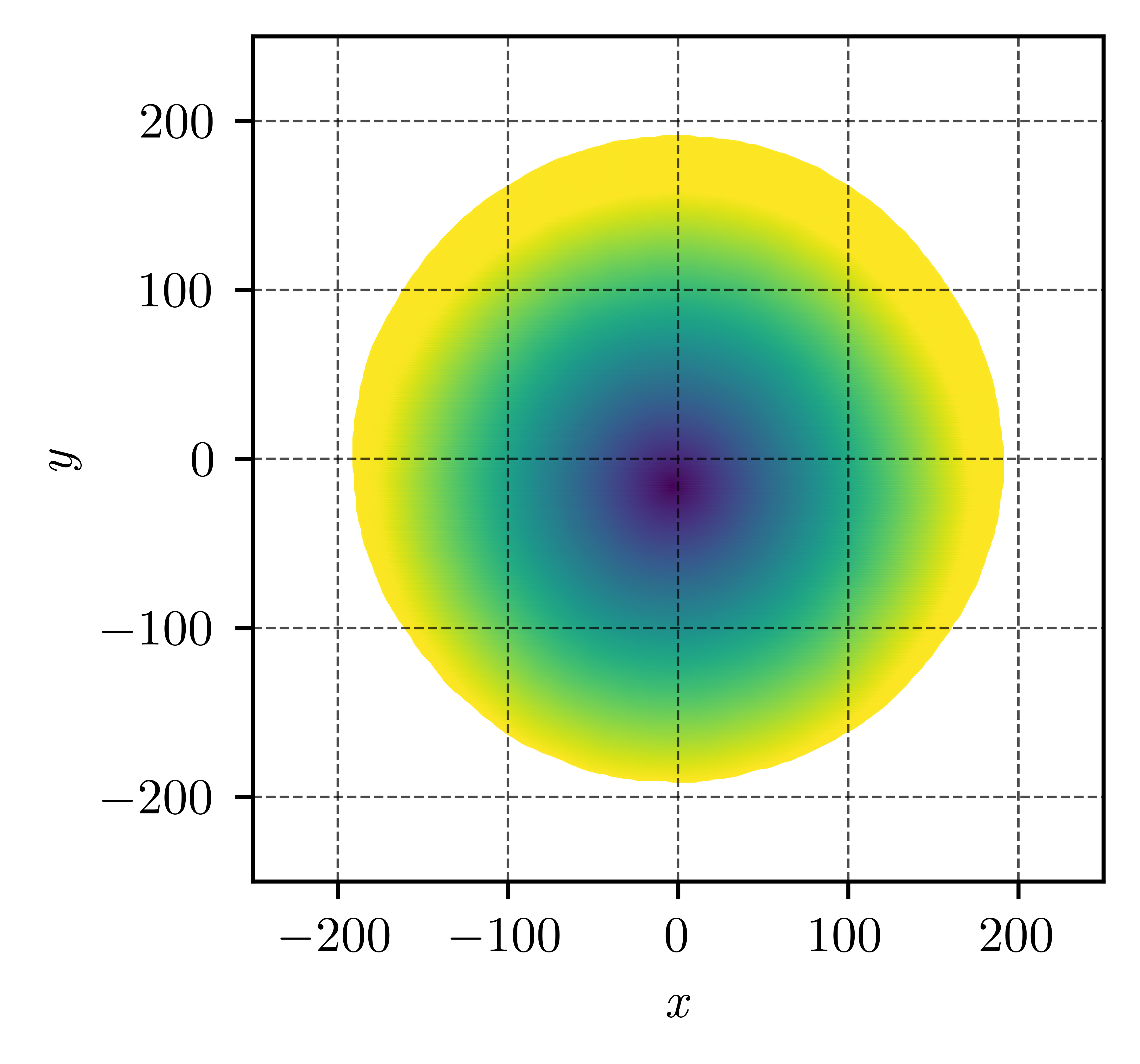}
        \caption{Level 75 sub-level set of an approximation of $-\ppShapeUnrA{\infty}$, where $\sA = \partial \cone = \{(x,y,0) : x,y\in\bbR\}$. Darker colors represent points closer to the origin.}\label{fig:FPP2D}
    \end{subfigure}

    \caption{A simulation of the ball of radius 75, $\{\xi:-\ppShapeUnr{\infty}(\xi)\le75\}$,  in Example \ref{ex:FPPforb} obtained by taking sub-level sets of the average of 100 samples of passage times from the origin. Both figures use the same data, so the pane on the right is the figure on the left viewed from below (with different coloring). We expect that $\ppShapeUnr{\infty}$ is continuous on $\cone$.}\label{fig:FPPforbshape}
\end{figure}

As explained above, it is enough to study the model from the perspective of the relative interior of each face $\sA$ of $\cone$. That there is a locally uniform version of the limit \eqref{ppShapeUnr} defining the shape function is the content of the following ``shape theorem". 

For $x \in \Rsemi$ let $\invSig_x$ be the $\sigma$-algebra generated by $A \in \Sig$ such that $\T_x^{-1} A = A$.
 
\begin{theorem}\label{Thm:ShapeUnr}
	Let $\beta \in (0,\infty]$.  Fix a face $\sA \in \faces$ of $\cone$ {\rm(}possibly $\cone$ itself\,{\rm)}.  
    Assume $\pote^+(\w,z) \in L^1(\bbP)$ for each $z \in \rangeA$, and $\pote^+(\w, z) \in \classL_{z, \rangeA}$ for each $z \in \rangeA \setminus \rangeA^{\id}$. For each $z\in\rangeA^{\id}$, assume that there exists a $\widehat{z}\in\rangeA\setminus\rangeA^{\id}$ such that $\pote^+(\w,z)\in\classL_{\widehat{z},\rangeA}$.  Assume that $\bbE\bigl[\sup_{n\geq 1} n^{-1} \bbE[\freeunr{\orig}{nx}{\beta} \given \sI_x]\bigr] < \infty$ for all $x \in \RsemiA \setminus \{\orig\}$.  Then $\bbP$-almost surely,
	$\ppShapeUnrA{\beta}$ is finite on $\sA$,
    $\ppShapeUnrA{\beta,\usc}$ is the unique continuous extension of $\ppShapeUnrA{\beta}$ from $\ri\sA$ to $\sA$,
	\begin{align}\label{shape-UB}
		\varlimsup_{\substack{|x|_1 \rightarrow \infty \\ x \in \RsemiA}} \frac{\freeunr{\orig}{x}{\beta} - \ppShapeUnrA{\beta, \usc}(x)}{|x|_1} \le 0,
	\end{align}
	and, simultaneously for all $\xi\in\sA$ with $\abs{\xi}_1=1$ and any sequence $v_n\in\RsemiA$ with $\abs{v_n}_1\to\infty$ and $v_n/\abs{v_n}_1\to\xi$, we have 
	\begin{align}\label{shape-unr}
		\ppShapeUnr{\beta}(\xi)\le
		\varliminf_{n\to\infty} \abs{v_n}_1^{-1}\freeunr{\orig}{v_n}{\beta}
		\le\varlimsup_{n\to\infty} \abs{v_n}_1^{-1}\freeunr{\orig}{v_n}{\beta}\le\ppShapeUnrA{\beta, \usc}(\xi).
	\end{align}
	If furthermore $\bigl(\Omega,\Sig^\pote_{\rangeA},\P,\{\T_z:z\in\rangeA\}\bigr)$ is ergodic, then $\ppShapeUnrA{\beta, \usc}$ is deterministic on $\sA$.  
\end{theorem}

In the statement of the previous result, the first three conditions concerning the integrability of $\pote^+$ and membership of potential steps in $\classL$ classes are minor regularity assumptions which are satisfied in the typical applications. The last condition, on the integrability of conditional expectations of $n^{-1}\freeunr{\orig}{nx}{\beta}$, is addressed in the following remark.

\begin{remark}\label{shape:finite}
	\cite[Theorem 2.8]{Jan-Nur-Ras-22} provides common cases which guarantee the  hypothesis \[
	\bbE\Bigl[\sup_{n\geq 1} n^{-1} \bbE[\freeunr{\orig}{nx}{\beta} \given \sI_x]\Bigr] < \infty\] holds for all $x \in \Rsemi \setminus \{\orig\}$. Particularly this hypothesis is satisfied under Condition \ref{VCondition} below.  
    See \cite[Lemmas 3.11 and 3.12]{Jan-Nur-Ras-22} for the proofs of these implications.
\end{remark}

Naturally, there is a version of the previous result in the special case of a restricted-length model, which we record below for use later in the paper.  Given a face $\sU'$ of $\Uset$, let $\ppShapeResU{\beta}$ be the function that equals $\ppShapeRes{\beta}$ on $\sU'$ and equals $-\infty$ outside of $\sU'$. Denote the upper semicontinuous regularizations of $\ppShapeResU{\beta}$ by $\ppShapeResU{\beta, \usc}$.  Then, under the conditions of Theorem \ref{Thm:ShapeRes}, the restriction of $\ppShapeResU{\beta, \usc}$ to  $\sU'$ is the unique continuous extension of $\ppShapeResU{\beta}$ from $\ri \sU'$ to all of $\sU'$. In particular, $\ppShapeResU{\beta, \usc}$ and  $\ppShapeUnrA{\beta, \usc}$ are concave and $\ppShapeUnrA{\beta, \usc}$ is positively $1$-homogeneous.

\begin{theorem}\label{Thm:ShapeRes}
	Let $\beta\in(0,\infty]$. Fix a face $\sU' \neq \{\orig\}$ of $\Uset$ {\rm(}possibly $\Uset$ itself\,{\rm)}.  Let $\sA$ be the cone generated by $\sU'$. 
	Assume $\rangeA\ne\rangeA^{\id}$.  
    Assume $\pote^+(\w,z) \in L^1(\bbP)$ for each $z \in \rangeA$, and $\pote^+(\w, z) \in \classL_{z, \rangeA}$ for each $z \in \rangeA \setminus \rangeA^{\id}$. For each $z\in\rangeA^{\id}$, assume that there exists a $\widehat{z}\in\rangeA\setminus\rangeA^{\id}$ such that $\pote^+(\w,z)\in\classL_{\widehat{z},\rangeA}$.
    Take $\beta\in(0,\infty]$.
	Assume that for all $k \in \ZZ_{>0}$ and $x \in \Dn{k}(\range\cap\Uset')$, $\bbE\bigl[\sup_{n\geq 1} n^{-1} \bbE[\freeres{\orig}{nx}{nk}{\beta} \given \sI_x]\bigr] <\infty$.   
    Then $\bbP$-almost surely, $\ppShapeResU{\beta}$ is finite on $\Uset'$, 
    $\ppShapeResU{\beta,\usc}$ is the unique continuous extension of $\ppShapeResU{\beta}$ from $\ri\sU'$ to $\sU'$,
	\begin{align}\label{shape-res-UB}
		\varlimsup_{n \rightarrow \infty} \max_{x \in n\sU' \cap \Dn{n}}\frac{\freeres{\orig}{x}{n}{\beta} - \ppShapeResU{\beta, \usc}(x)}{n} \le 0,
	\end{align}
	and, simultaneously for all $\xi\in\sU'$ and any sequence $v_n\in\RsemiA$ with $v_n/n\to\xi$ we have 
	\begin{align}\label{shape-res}
		\ppShapeResU{\beta}(\xi)\le
		\varliminf_{n\to\infty} n^{-1}\freeres{\orig}{v_n}{n}{\beta}
		\le\varlimsup_{n\to\infty} n^{-1}\freeres{\orig}{v_n}{n}{\beta}\le\ppShapeResU{\beta, \usc}(\xi).
	\end{align}
	If furthermore $\bigl(\bbP,\Sig^\pote_{\range'},\P\{\T_z:z\in\rangeA\}\bigr)$ is ergodic, then $\ppShapeResU{\beta, \usc}$ is deterministic on $\sU'$.  
\end{theorem}

\begin{remark}\label{rk:La_res-finite}
Condition \ref{VConditionRes} below implies 
 \[\bbE\Bigl[\sup_{n\geq 1} n^{-1} \bbE[\freeres{\orig}{nx}{nk}{\beta} \given \sI_x]\Bigr] <\infty\]
holds for all $k \in \ZZ_{>0}$ and $x \in \Dn{k}(\range\cap\Uset')$. 
This can be shown by transforming the restricted-length model into an unrestricted-length one, as described in Remark \ref{rk:resAsUnr1}, using \eqref{Esup-unr=res}, and following the argument in Remark \ref{shape:finite} within the context of the unrestricted-length version.
\end{remark}

	\section{The main result}\label{sec:main}
 	Recall that $\range$ denotes the set of admissible steps, its convex hull $\Uset$, and the cone $\cone = \{\sum_{z \in \range} b_z z : b_z \in \Rnonneg\}$. 
	 By \cite[Corollary A.2]{Ras-Sep-Yil-13}, $\orig \in \Uset$ is equivalent to the existence of loops, i.e.\ $\PathsPtPUnr{\orig}{\orig} \neq \{\orig\}$. By \cite[Corollary A.3]{Ras-Sep-Yil-13}, $\orig \in \ri \Uset$ is equivalent to the existence of admissible paths from $\orig$ to each $x \in \Rgroup$, i.e.\ $\Rsemi = \Rgroup$.  See Figure \ref{fig:Uset} for an illustration of possible cases of $\range$ and $\Uset$ and Figure \ref{fig:Shape} for illustrations of $\cone$. 
     Note that because this cone is finitely generated, it is either equal to $\bbR^d$ (which happens if $\orig \in \ri \Uset)$ or it is an intersection of finitely many half spaces whose boundary hyperplanes pass through the origin (see \cite[Theorem 19.1]{Roc-70}).
	 
\begin{figure}[h]
\begin{center}
\begin{tikzpicture}[>=latex,scale=0.6, z={(0.03,.55)}, y={(-.35,1.2)}]
\begin{scope}
\shade [ball color=red!90] (0,0,0) circle [radius=.1cm];
\draw (-0.1,-0.4,0) node {{\color{my-red}$\orig$}};
\draw[fill=cyan!80,line width=0.6pt,draw=cyan!50!black](3,0,0)--(0,3,0)--(0,0,3)--(3,0,0);
\shade [ball color=red!90] (3,0,0) circle [radius=.2cm];
\shade [ball color=red!90] (0,3,0) circle [radius=.2cm];
\shade [ball color=red!90] (0,0,3) circle [radius=.2cm];
\end{scope}
\begin{scope}[shift={(5.5,0)}]
\shade [ball color=red!90] (0,0,0) circle [radius=.1cm];
\draw (-0.1,-0.4,0) node {{\color{my-red}$\orig$}};
\draw[fill=cyan,line width=0.6pt,draw=cyan!50!black](3,0,0)--(0,0,3)--(2.5,2.5,2.5)--(3,0,0);
\draw[fill=cyan!80,line width=0.6pt,draw=cyan!50!black](0,3,0)--(0,0,3)--(2.5,2.5,2.5)--(0,3,0);
\draw[opacity=0.15](0,3,0)--(3,0,0);
\shade [ball color=red!90] (3,0,0) circle [radius=.2cm];
\shade [ball color=red!90] (0,3,0) circle [radius=.2cm];
\shade [ball color=red!90] (0,0,3) circle [radius=.2cm];
\shade [ball color=red!90] (2.5,2.5,2.5) circle [radius=.2cm];
\end{scope}
\begin{scope}[shift={(11,0)}]
\draw[fill=cyan,line width=0.6pt,draw=cyan!50!black](3,0,0)--(0,0,3)--(2.5,2.5,2.5)--(3,0,0);
\draw[fill=cyan!80,line width=0.6pt,draw=cyan!50!black](0,3,0)--(0,0,3)--(2.5,2.5,2.5)--(0,3,0);
\draw[opacity=0.15](0,3,0)--(3,0,0);
\shade [ball color=red!90] (0,3,0) circle [radius=.2cm];
\shade [ball color=red!90] (3,0,0) circle [radius=.2cm];
\shade [ball color=red!90] (0,0,3) circle [radius=.2cm];
\shade [ball color=red!90] (2.5,2.5,2.5) circle [radius=.2cm];
\draw (-.4,-.4,3) node {{\color{my-red}$\orig$}};
\end{scope}
\begin{scope}[shift={(16.5,0)}]
\draw[fill=cyan,line width=0.6pt,draw=cyan!50!black](3,0,0)--(0,0,3)--(2.5,2.5,2.5)--(3,0,0);
\draw[fill=cyan!80,line width=0.6pt,draw=cyan!50!black](0,3,0)--(0,0,3)--(2.5,2.5,2.5)--(0,3,0);
\draw[opacity=0.15](0,3,0)--(3,0,0);
\shade [ball color=red!90] (3,0,0) circle [radius=.2cm];
\shade [ball color=red!90] (0,3,0) circle [radius=.2cm];
\shade [ball color=red!90] (0,0,3) circle [radius=.2cm];
\shade [ball color=red!90] (2.5,2.5,2.5) circle [radius=.2cm];
\shade [opacity=0.65,ball color=red!90] (.95,.95,.95) circle [radius=.1cm];
\draw (1.2,0.75,1) node {{\transparent{0.65}\color{my-red}$\orig$}};
\end{scope}
\begin{scope}[shift={(22,0)}]
\draw[fill=cyan,line width=0.6pt,draw=cyan!50!black](3,0,0)--(0,0,3)--(2.5,2.5 ,2.5)--(3,0,0);
\draw[fill=cyan!80,line width=0.6pt,draw=cyan!50!black](0,3,0)--(0,0,3)--(2.5,2.5,2.5)--(0,3,0);
\draw[opacity=0.15](0,3,0)--(3,0,0);
\shade [ball color=red!90] (3,0,0) circle [radius=.2cm];
\shade [ball color=red!90] (0,3,0) circle [radius=.2cm];
\shade [ball color=red!90] (0,0,3) circle [radius=.2cm];
\shade [ball color=red!90] (2.5,2.5,2.5) circle [radius=.2cm];
\shade [opacity=0.65,ball color=red!90] (1.4,1.4,1.4) circle [radius=.1cm];
\draw (1.5,1.05,1.5) node {{\transparent{0.65}\color{my-red}$\orig$}};
\end{scope}
\end{tikzpicture}
\end{center}
\caption{\small The different panels depict various possible cases of $\range$ and $\Uset$. In each panel, the large balls represent $z\in\range$ that are extreme in $\Uset$, the small ball is $\orig$, and $\Uset$ is the lightly shaded figure.  Left to right: $\exists\uhat:\forall x\in\Uset,\ x\cdot\uhat=1$; $\orig\notin\Uset$; $\orig\in\ext\Uset$; 
$\orig\in\Uset\setminus\ri\Uset$; $\orig\in\ri\Uset$.}
\label{fig:Uset} 
\end{figure}

	%

%

Given a proper concave function $f:\R^d\to[-\infty,\infty)$, we will denote by $\partial f$ the superdifferential of $f$.  For $\xi\in\RR^d$,
\[\partial f(\xi)=\{v \in \RR^d : f(\zeta) \leq f(\xi)+v\cdot (\zeta - \xi)  \quad \forall \zeta \in \RR^d \}.\]  
By the argument in \cite[p.\ 15]{Roc-70}, $\partial f(\xi)$ is a closed convex set.
The image of $A\subset\R^d$ is denoted by $\partial f(A)$. If $f$ is positively $1$-homogeneous, then by the Fenchel-Young inequality, if $m\in\partial f(\xi)$ then $f(\xi)=m\cdot\xi$. See Lemma \ref{lem:concaveDualHomogeneous}.

\pgfdeclareradialshading{parabola}{\pgfpoint{-0.5cm}{-0.95cm}}%
{color(0cm)=(white);
color(0.15cm)=(cyan!20);
color(0.5cm)=(cyan!50);
color(0.9cm)=(cyan!90);
color(1cm)=(blue!30!cyan);
color(1.1cm)=(blue!50!cyan)
}

\begin{figure}[h]
\begin{center}
\begin{tikzpicture}[>=latex]
\tdplotsetmaincoords{70}{165}
\begin{scope}[rotate=30]
    \shade[ball color=cyan,opacity=0.8] (0,0) ellipse (2cm and 1.5cm);
    \draw[line width=0.8pt,draw=cyan!60!black] (-2,0) arc (180:360:2cm and 0.5cm);
    \draw[draw=cyan!60!black,opacity=0.4] (-2,0) arc (180:0:2cm and 0.5cm);
    \draw[line width=0.8pt,draw=cyan!60!black] (0,1.5) arc (90:270:0.5cm and 1.5cm);
    \draw[draw=cyan!60!black,opacity=0.4] (0,1.5) arc (90:-90:0.5cm and 1.5cm);
    \draw[line width=0.8pt,draw=cyan!60!black] (0,0) ellipse (2cm and 1.5cm);
    \shade [opacity=0.65,ball color=red!90] (-.8,0.1,0) circle [radius=.08cm];
    \draw (-.8,-.2,0) node {{\transparent{0.65}\color{my-red}$\orig$}};
\end{scope}
\begin{scope}[scale=2,shift={(2.3,0)},tdplot_main_coords,rotate=20]
\fill[cyan!75,line width=0.6pt,draw=cyan!60!black](1,0,0) -- (0,1,0) -- (0,0,1) -- cycle;
\fill[cyan!85,line width=0.6pt,draw=cyan!60!black](1,0,0) -- (0,1,0) -- (0,0,-1) -- cycle;
\fill[cyan!90,line width=0.6pt,draw=cyan!60!black](0,1,0) -- (-1,0,0) -- (0,0,1) -- cycle;
\fill[cyan,line width=0.6pt,draw=cyan!60!black](0,1,0) -- (-1,0,0) -- (0,0,-1) -- cycle;
\draw[opacity=0.15](-1,0,0)--(0,-1,0)--(1,0,0);
\draw[opacity=0.15](0,0,1)--(0,-1,0)--(0,0,-1);
\shade [opacity=0.65,ball color=red!90] (0,-0.3,.1) circle [radius=.04cm];
\draw (0,-.3,-.05) node {{\transparent{0.5}\color{my-red}$\orig$}};
\end{scope}
\begin{scope}[scale=0.37,shift={(20,-6)}]
\fill[cyan!80,line width=0.6pt,draw=cyan!50!black!70](0,0)--(7,1.5)--(9,4)--(6,5)--(4,5)--cycle;
\fill[cyan!60,line width=0.6pt,draw=cyan!50!black!70](4,5)--(6,5)--(4,7)--cycle;
\fill[cyan!50,line width=0.6pt,draw=cyan!50!black!70](4,7)--(6,5)--(9,4)--(6,8)--cycle;
\fill[cyan!50,line width=0.6pt,draw=cyan!50!black!70](4,7)--(6,8)--(1,7)--cycle;
\fill[cyan!70,line width=0.6pt,draw=cyan!50!black!70](0,0)--(4,5)--(4,7)--(1,7)--cycle;
\draw[opacity=0.1](6,8)--(7,1.5);
\fill[red,opacity=0.25,draw=red!20](0,0)--(15,3.21)--(6.8,8.5)--cycle;
\draw[red!50](0,0)--(15,3.21)--(6.8,8.5)--cycle;
\fill[red,opacity=0.2,draw=red!20](0,0)--(6.8,8.5)--(2,14)--cycle;
\draw[red!50](0,0)--(2,14)--(6.8,8.5)--cycle;
\fill[red,opacity=0.1,draw=red!10](15,3.21)--(2,14)--(6.8,8.5);
\draw[red!50](15,3.21)--(2,14)--(6.8,8.5)--cycle;
\shade [ball color=red!90] (0,0) circle [radius=.2];
\draw (0.15,-.65) node {{\color{my-red}$\orig$}};
\end{scope}
\begin{scope}[scale=0.3,shift={(-5,-25)}]
\fill[cyan!80,line width=0.6pt,draw=cyan!50!black](0,0)--(8,3)--(13.23,6.4)--(10.5,7.62)--(3,3)--cycle;
\fill[cyan!60,line width=0.6pt,draw=cyan!50!black](0,0)--(3,3)--(6.83,10.66)--(4.47,13.68)--(1,6)--cycle;
\fill[cyan!70,line width=0.6pt,draw=cyan!60!black!60](3,3)--(10.5,7.62)--(6.83,10.66)--cycle;
\fill[cyan!20,line width=0.6pt,draw=cyan!50!black](4.47,13.68)--(6.35,12.95)--(12.22,7.92)--(13.23,6.4)--(10.5,7.62)--(6.83,10.66)--cycle;
\draw[opacity=0.25](1,6)--(6.35,12.95);
\draw[opacity=0.25](8,3)--(12.22,7.92);
\fill[red,opacity=0.25,draw=red!20](0,0)--(8.5,8.5)--(15,5.625)--cycle;
\draw[red!50](0,0)--(8.5,8.5)--(15,5.625)--cycle;
\fill[red,opacity=0.2,draw=red!20](0,0)--(8.5,8.5)--(2.67,16)--cycle;
\draw[red!50](0,0)--(8.5,8.5)--(2.67,16)--cycle;
\fill[red,opacity=0.1,draw=red!10](8.5,8.5)--(2.67,16)--(15,5.625);
\draw[red!50](8.5,8.5)--(2.67,16)--(15,5.625)--cycle;
\shade [ball color=red!90] (0,0) circle [radius=.3];
\draw (0.5,-0.9) node {{\color{my-red}$\orig$}};
\end{scope}
\begin{scope}[scale=0.3,shift={(10.5,-25)}]
\shade[
shading=parabola,
line width=0.6pt,draw=cyan!50!black,opacity=0.8,scale=2,rotate=-41.5,shift={(0,1.25)}]{(-1.75,5.3) parabola bend (0,0) (1.75,5.3)}{--(1.75,5.3) arc [start angle=360, end angle=180, x radius = 1.75, y radius=.5]};
\fill[cyan!20,draw=cyan!50!black,scale=2,rotate=-41.5,shift={(0,1.25)}]{(1.75,5.25) arc [start angle=0, end angle=180, x radius = 1.75, y radius=.5]}{(1.75,5.25) arc [start angle=360, end angle=180, x radius = 1.75, y radius=.5]};
\draw[draw=cyan!60!black,opacity=0.4,scale=2,rotate=-41.5,shift={(0,1.25)}](0,0) parabola (-1,5.75);
\draw[line width=0.8pt,draw=cyan!60!black,scale=2,rotate=-41.5,shift={(0,1.25)}](0,0) parabola (1,4.8);
\draw[draw=cyan!60!black,opacity=0.4,scale=2,rotate=-41.5,shift={(0,1.25)}] (1.3,2.897) arc [start angle=0, end angle=180, x radius = 1.3, y radius=.5];
\draw[line width=0.8pt,draw=cyan!60!black,scale=2,rotate=-41.5,shift={(0,1.25)}] (1.3,2.897) arc [start angle=360, end angle=180, x radius = 1.3, y radius=.5];
\fill[red,opacity=0.25,draw=red!20](0,0)--(8.3,8.3)--(15,5.625)--cycle;
\draw[red!50](0,0)--(8.3,8.3)--(15,5.625)--cycle;
\fill[red,opacity=0.2,draw=red!20](0,0)--(8.3,8.3)--(2.67,16)--cycle;
\draw[red!50](0,0)--(8.3,8.3)--(2.67,16)--cycle;
\fill[red,opacity=0.1,draw=red!10](8.3,8.3)--(2.67,16)--(15,5.625);
\draw[red!50](8.3,8.3)--(2.67,16)--(15,5.625)--cycle;
\shade [ball color=red!90] (0,0) circle [radius=.3];
\draw (0.5,-0.9) node {{\color{my-red}$\orig$}};
\end{scope}
\begin{scope}[scale=.3,shift={(27,-25)}]
\fill[cyan!20,opacity=0.6,draw=cyan!40!black!40,scale=2,rotate=-41.5,shift={(-0.1,1.25)}]{(1.535,1.5) parabola bend (0,.95) (-1.6,1.5)}--(-4.2,5.63)--(3.85,5.83)--(1.535,1.5);
\shade[
shading=parabola,
line width=0.6pt,opacity=0.7,scale=2,rotate=-41.5,shift={(-0.1,1.25)}]{(-1.6,1.5) parabola bend (0,0) (1.535,1.5)}{(1.535,1.5) parabola bend (0.65,1) (0.28,1.5)}{(0.28,1.5) parabola bend (-0.5,0.9) (-1.6,1.5)};
\draw[line width=0.6pt,draw=cyan!30!black,opacity=0.5,scale=2,rotate=-41.5,shift={(-0.1,1.25)}](-1.6,1.5) parabola bend (0,0) (1.535,1.5);
\draw[draw=cyan!60!black,opacity=.3,scale=2,rotate=-41.5,shift={(-0.1,1.25)}] (1.075,.75) arc [start angle=5, end angle=185, x radius = 1.075, y radius=.15];
\draw[line width=0.6pt,draw=cyan!60!black,opacity=.5,scale=2,rotate=-41.5,shift={(-0.1,1.25)}] (1.075,.75) arc [start angle=355, end angle=185, x radius = 1.075, y radius=.35];
\fill[cyan,line width=0.6pt,draw=cyan!70!black!60,opacity=0.7,scale=2,rotate=-41.5,shift={(-0.1,1.25)}]{(1.535,1.5) parabola bend (0.65,1) (0.28,1.5)}--(0.46,4.75)--(3.85,5.83)--(1.535,1.5);
\fill[cyan!90,line width=0.6pt,draw=cyan!70!black!60,opacity=0.75,scale=2,rotate=-41.5,shift={(-0.1,1.25)}]{(0.28,1.5) parabola bend (-0.5,0.9) (-1.6,1.5)}--(-4.2,5.63)--(0.46,4.75)--(0.28,1.5);
\fill[red,opacity=0.25,draw=red!20](0,0)--(8.5,8.5)--(15,5.625)--cycle;
\draw[red!50](0,0)--(8.5,8.5)--(15,5.625)--cycle;
\fill[red,opacity=0.2,draw=red!20](0,0)--(8.5,8.5)--(2.67,16)--cycle;
\draw[red!50](0,0)--(8.5,8.5)--(2.67,16)--cycle;
\fill[red,opacity=0.1,draw=red!10](8.5,8.5)--(2.67,16)--(15,5.625);
\draw[red!50](8.5,8.5)--(2.67,16)--(15,5.625)--cycle;
\shade [ball color=red!90] (0,0) circle [radius=.3];
\draw (0.5,-0.9) node {{\color{my-red}$\orig$}};
\end{scope}
\end{tikzpicture}
\end{center}
\caption{\small The different panels depict various possible cases of the super-level sets $B_t=\{\xi\in\cone:\ppShapeUnr{\beta,\usc}(\xi)\ge t\}$, which are depicted darkly shaded. When $\orig\in\ri\Uset$, the cone $\cone$ is the whole of $\R^d$. When $\orig\notin\ri\Uset$, the cone $\cone$ is depicted lightly shaded. These cones are unbounded but truncated in the drawings and hence appear as pyramids. The faces of each of these cones are the cone itself, the two-dimensional faces that are truncated in the drawing and appear as triangles on the side of the pyramid, the lines that are the boundaries of the two-dimensional faces, and the point $\orig$. The super-level sets on the first row are all bounded. Left to right: $\orig\in\ri\Uset$, $t<0$, and $\ppShapeUnr{\beta,\usc}$ is differentiable; $\orig\in\ri\Uset$, $t<0$, and $B_t$ is polygonal; $\orig\notin\Uset$, $\ppShapeUnr{\beta,\usc}(\xi)<0$ $\forall\xi\ne\orig$, $t<0$, and $B_t$ is polygonal. The shapes on the second row are all unbounded and are truncated in the drawings. Left to right:  $\orig\notin\Uset$, $\ppShapeUnr{\beta,\usc}$ takes both $<0$ and $\ge0$ values, $t<0$, and $B_t$ is polygonal; $\orig\notin\Uset$, $\ppShapeUnr{\beta,\usc}$ takes both $>0$ and $<0$ values, $t>0$, and $\ppShapeUnr{\beta,\usc}$ is differentiable; $\orig\notin\Uset$, $\ppShapeUnr{\beta,\usc}(\xi)>0$ $\forall\xi\ne\orig$, $t>0$, and $\ppShapeUnr{\beta,\usc}$ is differentiable.}
\label{fig:Shape} 
\end{figure}

%
%
%
%
%

For each face $\sA \in \faces$ (possibly $\cone$ itself), let $\subspaceA$ be the linear subspace of $\RR^d$ generated by $\rangeA$. 
Since $\ppShapeUnrA{\beta,\usc}(\zeta) = -\infty$ when $\zeta\not\in\subspaceA$ we have
    \[\subspaceA\cap \superDiffUnrA{\beta,\usc}(\xi)
    =\bigl\{v\in\subspaceA:\ppShapeUnrA{\beta,\usc}(\zeta) \le \ppShapeUnrA{\beta,\usc}(\xi) + v\cdot(\zeta-\xi)\quad\forall\zeta\in\subspaceA\bigr\}.\]
The set on the right-hand side is the superdifferential at $\xi$ of the restriction of $\ppShapeUnrA{\beta,\usc}$ to $\subspaceA$. If $\ppShapeUnrA{\beta,\usc}(\xi)$ is finite for all $\xi\in\sA$, then \cite[Theorem 23.4]{Roc-70} says that this set is non-empty and bounded. It then has an extreme point, e.g., by the Krein–Milman theorem. Remark \ref{rmk:finiteLambda} explains how Conditions \ref{VCondition} and \ref{classLCondition} below guarantee the finiteness of $\ppShapeUnrA{\beta,\usc}$. 
See Figure \ref{fig:Shape} for an illustration of some possible cones and super-level sets of $\ppShapeUnr{\beta,\usc}$. Figure \ref{fig:FPPforbshape} already gives an example of the case where $\orig$ lies on the relative boundary of $\cone$.

%
\begin{remark}[FPP on $\Z^3$ with a forbidden step]
In the model described in Example \ref{ex:FPPforb} above, the cone has two faces $\cone = \{(x,y,z) : x,y\in\bbR, z\geq0\}$ and $\partial \cone =\{(x,y,0) : x,y\in\bbR\}$. If $\sA = \cone$, then $\subspaceA = \bbR^3$ and the elements of $\subspaceA\cap \superDiffUnrA{\infty,\usc}(\xi)$ correspond to $3$-dimensional hyperplanes ($2$-dimensional planes) which are tangent to $\ppShapeUnrA{\infty,\usc}$ at $\xi$. See Figure \ref{fig:FPP3D}. If $\sA = \partial \cone$, then $\subspaceA = \{(x,y,0) : x,y\in\bbR\}$ and the elements of $\subspaceA\cap \superDiffUnrA{\infty,\usc}(\xi)$ correspond to $2$-dimensional hyperplanes ($1$-dimensional lines) on $\{(x,y,0) : x,y\in\bbR\}$ which are tangent to $\ppShapeUnrA{\infty,\usc}$ at $\xi$.  See Figure \ref{fig:FPP2D}.
\end{remark}

To ensure that the unrestricted-length polymer measures in the next theorem are well defined, we need the following conditions.

     \begin{condition}[Unrestricted-length conditions]\label{VCondition} Given a face $\sA$ of $\cone$ {\rm(}possibly $\cone$ itself\,{\rm)}, assume that one of the following holds:
     \begin{enumerate}[label=\rm(\alph{*}), ref=\rm\alph{*}] \itemsep=2pt
     \item\label{condBddBelow0} The setting is undirected {\rm(}meaning $\orig \in\UsetA${\rm)} and the potential satisfies 
        \begin{align}\label{finiteZ}
	    \P\{\pote(\w,z)\ge0\}=1\quad\text{for all } z\in\rangeA
	    \end{align} 
     and $\pote^+(\w,z)\in L^1(\P)$ for each $z\in\rangeA$.
        \item\label{condBddBelowByC} The setting is directed {\rm(}meaning $\orig \not\in\UsetA${\rm)}, there exists $c \in \R$ such that
        \begin{align}
            \P\{\pote(\w,z)\ge c\}=1\quad\text{for all } z\in\rangeA,
        \end{align}
    and $\pote^+(\w,z)\in L^1(\P)$ for each $z\in\rangeA$. 
        
        \item\label{condProdMeasure}  The setting of Example \ref{ex:productEnv} holds,  $\orig \not\in \UsetA$, $\pote$ is local, $\P$ has a finite range of dependence $r_0$, and for some $q > d$, $\pote^+(\w,z)\in L^1(\P)$  and  $\pote^-(\w,z) \in L^q(\P)$ for all $z \in \rangeA$. 
     \end{enumerate}
     \end{condition}


When considering a restricted-length model, one can convert it to an unrestricted-length model as in Remark \ref{rk:resAsUnr1}. When viewing the model without doing this, parts \eqref{condBddBelowByC} and \eqref{condProdMeasure} of the above condition then  specialize to the condition below, which we present without going through unrestricted-length measures for the convenience of the reader.

     \begin{condition}[Restricted-length conditions]\label{VConditionRes} Given a face $\Uset'$ of $\Uset$ {\rm(}possibly $\Uset$ itself\,{\rm)}, assume that one of the following holds:
     \begin{enumerate}[label=\rm(\alph{*}), ref=\rm\alph{*}] \itemsep=2pt
        \item\label{condBddBelowByCRes} There exists $c \in \R$ such that
        \begin{align}
            \P\{\pote(\w,z)\ge c\}=1\quad\text{for all } z\in\range\cap\Uset'
        \end{align}
    and $\pote^+(\w,z)\in L^1(\P)$ for each $z\in\range\cap\Uset'$. 
        
        \item\label{condProdMeasureRes}  The setting of Example \ref{ex:productEnv} holds, $\orig \not\in \Uset'$, $\pote$ is local, $\P$ has a finite range of dependence $r_0$, and for some $q > d$, $\pote^+(\w,z)\in L^1(\P)$  and  $\pote^-(\w,z) \in L^q(\P)$ for all $z \in \range\cap\Uset'$. 
     \end{enumerate}
     \end{condition}

Before continuing, we comment on the strength and generality of these hypotheses.

\begin{remark} The moment hypotheses in Conditions \ref{VCondition}  and \ref{VConditionRes} hold in typical applications but are also often not quite sharp in particular examples. For example, in the standard FPP with i.i.d.\ edge weights, one can often weaken the moment hypothesis to one involving the minimum of the edge weights at a site. If $\orig\in\UsetA$, then \eqref{finiteZ} in Condition \ref{VCondition}\eqref{condBddBelow0} is necessary for the finiteness of $\prtunr{x}{y}{\beta}$. Otherwise loops can be repeated indefinitely, causing $\prtunr{x}{y}{\beta}(\w)=\infty$ for $x,y\in\ZZ^d$ with $y-x\in\RsemiA$, with positive probability. The standard FPP gives a frequently studied example of a model satisfying this condition. Condition \ref{VCondition}\eqref{condBddBelowByC} (and \ref{VConditionRes}\eqref{condBddBelowByCRes}) corresponds to a physically reasonable assumption that energies are bounded from below.  Among many other examples, this allows us to consider well-studied bounded potential models like uniformly elliptic RWRE models in ergodic environments, without needing independence. The last conditions allow us to cover examples which are unbounded below with a finite range of dependence. Natural examples of such models not covered by the first two cases include i.i.d.~Gaussian last-passage percolation and directed polymers as well as (not necessarily uniformly) elliptic RWREs in independent environments which satisfy our moment conditions.  
\end{remark}

Recall the set $\rangeA^{\id}$, defined in \eqref{Rid}. To apply the shape theorems for the limiting free energy we need the following mild technical condition. As explained after Definition \ref{def:cL}, this condition is a trade-off between moment and mixing assumptions and is typically valid in practice.

        \begin{condition}\label{classLCondition}
            Given a face $\sA$ of $\cone$ {\rm(}possibly $\cone$ itself\,{\rm)}, assume $\pote^+(\w,z) \in L^1(\bbP)$ for each $z \in \rangeA$, and $\pote^+(\w, z) \in \classL_{z, \rangeA}$ for each $z \in \rangeA \setminus \rangeA^{\id}$. For each $z\in\rangeA^{\id}$, assume that there exists a $\widehat{z}\in\rangeA\setminus\rangeA^{\id}$ such that $\pote^+(\w,z)\in\classL_{\widehat{z},\rangeA}$.
     \end{condition}

         \begin{remark}\label{rmk:finiteLambda}
    	By \cite[Theorem 2.8]{Jan-Nur-Ras-22}, if Conditions \ref{VCondition} and \ref{classLCondition} are satisfied, then we have, $\bbP$-almost surely, $\abs{\ppShapeUnrA{\beta}(\xi)}<\infty$ for all $\xi \in \ri \sA$, and $\abs{\ppShapeUnrA{\beta,\usc}(\xi)}<\infty$ for all $\xi \in \sA$. As mentioned earlier, the version of Condition \ref{classLCondition} assumed by \cite{Jan-Nur-Ras-22} is stronger than the one we assume here. However, their results still hold under our weaker condition with minor adjustments in the proofs. For detailed information on these adjustments, see the proof of Lemma \ref{lem:upperBoundCompact}.
        \end{remark}

A primary objective in this work is to construct polymer measures that are supported on the set of directed semi-infinite paths. To that end, we now define directedness.

\begin{definition}\label{def:directed}
For a set $A \subset \Uset$, the sequence $x_n$ is directed into $A$ if $|x_n|_1 \rightarrow \infty$ and all limit points of $x_n / |x_n|_1$ are contained in $A$.  
\end{definition}

The directedness of the polymers we construct requires the following extra assumption that avoids degenerate situations. Let $\Uset_\orig$ be the unique face of $\Uset$ that contains $\orig$ in its relative interior. Let $\range_{\orig}=\range\cap\Uset_\orig$. If $\orig\not\in\Uset$, then $\Uset_\orig=\range_\orig=\varnothing$. Let \[\sA_\orig=\Bigl\{\sum_{z \in \range_\orig} b_z z : b_z \in \Rnonneg\Bigr\}.\]

\begin{condition}\label{trans}
	Assume that 
	\begin{align}\label{V>0}
	\P\{\pote(\w,z)\ge0\}=1\quad\text{for all }z\in\range_{\orig}
	\end{align}
	and 
	if $\beta<\infty$, then assume also that either the reference walk $\rwP_{\orig}$ is transient or it is recurrent and \[\P\bigl\{\exists x\in\Rsemi_{\sA_\orig},\exists z\in\range_\orig:\pote(\T_x\w,z)>0\bigr\}=1.\] 
\end{condition}


\begin{remark}
Condition \eqref{V>0} states that the potential $\pote$ must be non-negative along any loops that may exist. This is implied by Condition \ref{VCondition}, which we already assume to ensure the well-definedness of the unrestricted-length polymer measures with which we work. By Lemma \ref{lm:trans.pos} below, Condition \ref{trans} is equivalent to the positive-temperature semi-infinite polymers we construct being transient, which is a necessary precondition for them to move in a preferred direction. 
\end{remark}

        For $\beta\in(0,\infty]$, a face $\sA\in\faces$ (including $\cone$), and $m\in \superDiffUnrA{\beta,\usc}(\ri\sA)$, define 
        \begin{equation}\label{def:facetUnr}
            \facetUnr{\m,\sA} = \{ \xi\in \sA : \abs{\xi}_1 = 1 \text{ and } \ppShapeUnrA{\beta, \usc}(\xi) = m \cdot \xi\}.
        \end{equation}
        This set is not empty because $m \in \superDiffUnrA{\beta,\usc}(\ri\sA)$ and $\ppShapeUnr{\beta,\usc}$ is positively $1$-homogeneous. 

	    Also, define 
        \begin{equation}\label{def:facetRes}
        \facetRes{\m,\sA} = \{\xi \in \UsetA : \ppShapeResA{\beta, \usc}(\xi) = m\cdot \xi\}.
        \end{equation}
        Theorem \ref{thm:CocycleMeasuresFace} implies, in particular, that this set is not empty (a direct convex analytic proof is also possible).


\begin{definition}
    For $x,y\in\Z^d$ with $y-x\in\Rgroup$ the Green's function is given by
	\begin{align}\label{g:def}
	\begin{split}
	\greens(x, y)
	&= \sum_{\ell=0}^\infty \prtres{x}{y}{\ell}{\beta} = \one_{\{y=x\}} + \sum_{\ell=1}^\infty \rwE_x\bigl[e^{-\beta \sum_{j=0}^{\ell-1} \pote(\T_{X_j} \w, X_{j+1}-X_j)} \one_{\{X_\ell=y\}}\bigr]\\ 
	&= \rwE_x\Bigl[\sum_{\ell=0}^\infty e^{-\beta \sum_{j=0}^{\ell-1} \pote(\T_{X_j} \w, X_{j+1}-X_j)} \one_{\{X_\ell=y\}}  \Bigr],
	\end{split}
	\end{align}
	where an empty sum is taken to be zero.
\end{definition}

Let $\DctbA$ be a countable (possibly empty or finite) set of pairs $(\beta, \m)$ so that for each pair, $\beta \in (0,\infty]$, 
$\ppShapeUnrA{\beta,\usc}$ is deterministic and finite on $\sA$, and there exists a $\xi \in(\ri \sA)\setminus\{\orig\}$ such that $\m \in \subspaceA \cap \superDiffUnrA{\beta,\usc}(\xi)$ and, if $\orig\in\UsetA$, $\ppShapeUnrA{\beta}(\xi) \neq 0$.  
An important case will be one where each $m$ is an extreme point,
\begin{equation}\label{mIsExtreme}
    \m \in \ext( \subspaceA \cap \superDiffUnrA{\beta,\usc}(\xi) ) \text{ for some } \xi \in(\ri \sA)\setminus\{\orig\}.
\end{equation}
Also, define $\Dctb = \{ (\sA, \beta, \m) : \sA \in \faces, (\beta,\m)\in\DctbA \}$  and $\Bctb = \{\beta \in (0,\infty] : \exists \m \in \RR^d, \sA \in \faces \text{ with } (\sA, \beta, \m) \in \Dctb\}$. 

\begin{remark}\label{rem:Dset}
For maximum utility, one would take $\DctbA$ to consist of a countable dense set of $\beta \in (0,\infty]$ which includes $\infty$ and then for each such $\beta$, an appropriately dense subset of $m \in \bigcup_{\xi \in \ri \sA\setminus{\{\orig\}}} \{\ext(\subspaceA \cap \superDiffUnrA{\beta,\usc}(\xi))\}$. One natural way to construct such a set is to take a dense set of $\beta$ as above and then a dense subset of directions $\xi \in \ri(\sA)\setminus{\{\orig\}}$ (for which $\ppShapeUnrA{\beta,\usc}(\xi)\neq 0$ if $\orig \in \UsetA$) which are directions of differentiability of $\ppShapeUnrA{\beta}$. For such $\xi$, $\superDiffUnrA{\beta}(\xi)$ consists of  a single point. Existence of such a set follows from \cite[Theorem 25.4]{Roc-70}. 
\end{remark}

\begin{remark}
It can happen that $(\beta,\m)$ can be in $\DctbA$ for multiple faces $\sA$ of $\cone$.
A simple, albeit degenerate, example is one where $d=3$, $\range=\{e_1,e_2,e_3\}$, and $\ppShapeUnr{\beta}(\xi)=\xi\cdot e_1$. Then for $\sA\in\{\cone,\cone(e_1,e_2),\cone(e_1,e_3),\cone(e_1)\}$,  and $\xi\in\ri\sA$, $\superDiffUnrA{\beta}(\xi)=e_1$. For all other faces $\sA$, $\ppShapeUnrA{\beta}\equiv0$ and $\superDiffUnrA{\beta}=\orig$.
\end{remark}

\begin{remark}[RWRE on $\bbZ$]\label{rk:RWREDA} 
We illustrate the conditions on the membership in $\DctbA$ in the case of the RWRE models from Figure \ref{fig:RWRE}, which were defined in Example \ref{ex:nRWRE}. In this model, $\cone=\R$ is the only face and $\subspaceC=\R$. 
Recall the description of $\ppShapeUnr{1}$ from Remark \ref{rk:RWRE_unrShape} and Figures \ref{fig:RWRE}. In the recurrent case, we have $\partial\ppShapeUnr{1}(t)=\{0\}$ for all $t\in\R$. In this case, there are no points that can belong to $\DctbC$ with $\beta=1$. In all the other cases, $\partial\ppShapeUnr{1}(t)=\{0\}$ for $t>0$, $\partial\ppShapeUnr{1}(t)=\{-\ppShapeUnr{1}(-1)\}$ for $t<0$, and $\partial\ppShapeUnr{1}(0)=[-\ppShapeUnr{1}(-1),0]$ (with the extreme slopes in the last case being $-\ppShapeUnr{1}(-1)$ and $0$). Thus, the only point that can belong to $\DctbC$ with $\beta=1$ is $(1,-\ppShapeUnr{1}(-1))$.

If, on the other hand, we consider the restricted-length version of the RWRE model and rewrite it as an unrestricted-length model, following Remark \ref{rk:resAsUnr1}, then the paths become directed and the restriction on the choice of $\xi$ in the definition of $\DctbCbar$ disappears. 
See Remark \ref{ex2:nRWRE} for further details.
\end{remark}


With the previous discussion in mind, the following is our main result. It produces for each linear facet of $\ppShapeUnrA{\beta, \usc}$ (which can degenerate into a singleton) a probability measure on semi-infinite paths that is consistent with both the restricted-length and the unrestricted-length point-to-point quenched polymer measures and under which the polymer is directed into the given facet. Appendix \ref{app:consistency}  discusses the Gibbs consistency properties of the restricted- and unrestricted-length measures, from which one can connect to the familiar setting of Dobrushin-Lanford-Ruelle equations.
Before stating the result, note that because $\cone$ is finitely generated, it has finitely many faces \cite[Theorem 19.1]{Roc-70}.

	\begin{theorem}\label{thm:CocycleMeasuresFace} 
	      For each face $\sA \in \faces$ {\rm(}possibly $\cone$ itself\,{\rm)} where  $\DctbA \neq \varnothing$, 
	    assume that $\pote$ satisfies Conditions \ref{VCondition} and \ref{classLCondition} on $\sA$ 
     and that \eqref{mIsExtreme} holds for each $(\beta,m) \in \DctbA$. 
	    For each $x\in\ZZ^d$, there exists a family $\{\Qinf{x}{\sA, \beta, \m, \w}:\sA \in\faces,(\beta, \m) \in \DctbA\}$ of random probability measures  on semi-infinite paths $x_{0:\infty} \in \PathsSemiInf{x}(\rangeA)$, such that the following is satisfied. There exists an event $\OmegaDir \subset \Omega$  with $\bbP(\OmegaDir) = 1$, on which the following properties hold: 
		\begin{enumerate}[label=\rm(\alph{*}), ref=\rm\alph{*}] \itemsep=2pt 
			\item\label{CocycleMeasuresFace.consPos} If $\beta < \infty$, then for all non-negative integers $j \leq k \leq n$, points $u, x, y \in \ZZ^d$ with $x-u$ and $y-x$ in $\RsemiA$, $x_{0:k} \in \PathsPtPKilled{u}{y}$ with $j = \min\{0 \leq i \leq k : x_i = x\}$, and $x_{k:n} \in \PathsNStep{y}{n-k}(\rangeA)$,  
			\begin{align*} &\Qinf{u}{\sA, \beta, \m, \w}(X_{0:\stoppt{y}+n-k} = x_{0:n} \given \stoppt{x} \leq \stoppt{y} <\infty, X_{0:\stoppt{x}} = x_{0:j}, X_{\stoppt{y}:\stoppt{y}+n-k}=x_{k:n}) \\
				&\qquad=\Qinf{u}{\sA, \beta, \m, \w}(X_{\stoppt{x}:\stoppt{y}} = x_{j:k} \given \stoppt{x} \leq \stoppt{y} <\infty) = \Qunr{x}{y}{\beta, \w}(X_{0:\stoppt{y}} = x_{j:k}). 
			\end{align*} 

            \item\label{CocycleMeasuresFace.Greens} If $\beta<\infty$, then for each distinct $x,y \in \ZZ^d$ such that $y-x \in \RsemiA$,
            \begin{align*}
                 \rwE^{\Qinf{x}{\sA,\beta,\m,\w}}\Bigl[ \sum_{n=0}^\infty \one_{\{X_n = y\}} \Big| \stoppt{y}<\infty \Bigr] = \frac{\greens(x,y)}{\prtunr{x}{y}{\beta,\w}}.
            \end{align*}

            \item\label{CocycleMeasuresFace.consPosRes} If $\beta < \infty$, then for all non-negative integers $j \leq k \leq n$, points $u, x, y \in \ZZ^d$  such that $x-u \in \DnA{j}$ and $y-x \in \DnA{k-j}$, $x_{0:j} \in \PathsPtPResKM{u}{x}{0}{j}$, $x_{j:k} \in \PathsPtPResKM{x}{y}{j}{k}$, and $x_{k:n} \in \PathsNStep{y}{n-k}(\rangeA)$, 
	\begin{align*}
		\Qinf{u}{\sA, \beta, \m, \w}(X_{0:n} = x_{0:n} \given X_{0:j} = x_{0:j}, X_{k:n} = x_{k:n}) &= \Qinf{u}{\sA, \beta, \m, \w}(X_{j:k} = x_{j:k} \given X_j = x, X_k = y) \\
		&= \Qres{x}{y}{k-j}{\beta, \w}(X_{0:k-j} = x_{j:k}).
	\end{align*} 		
			
			\item\label{CocycleMeasuresFace.consZero} If $\beta = \infty$, then for each $x \in \ZZ^d$, $\Qinf{x}{\sA, \infty, \m, \w}$ is supported on a set of semi-infinite geodesics rooted at $x$ and having increments in $\rangeA$.
			
			\item\label{CocycleMeasuresFace.trans} If either $\beta=\infty$, or $\beta<\infty$ while $\sA\ne\sA_\orig$, or $\beta<\infty$,  $\sA=\sA_\orig$, and Condition \ref{trans} holds, then for all $x\in\ZZ^d$, $\Qinf{x}{\sA, \beta, \m, \w}\{\abs{X_n}_1\to\infty\}=1$ and 
%
%
            \begin{align}\label{CocycleMeasuresFace.dir}
            \Qinf{x}{\sA, \beta, \m, \w}(X_{0:\infty} \text{ is directed into } \facetUnr{\m,\sA}) = 1.
            \end{align}
   
			\item\label{CocycleMeasuresFace.LLNlimits} For each $x \in \ZZ^d$, $\Qinf{x}{\sA, \beta, \m, \w}(\text{all limit points of } X_n/n \text{ are contained in } \facetRes{\m,\sA}) = 1$. 
		\end{enumerate}	
	\end{theorem}
    
	\begin{remark}[Previous work]\label{rem:Gibbsrel} This result extends several works in the literature, including \cite[Theorem 1.1]{Dam-Han-14} and the extension of that result to $\bbZ^d$ described in \cite[Section 5.4]{Auf-Dam-Han-17} (though with slightly more restrictive moment hypotheses), \cite[Theorem 3.2]{Jan-Ras-20-aop}, and \cite[Theorem 2.1 (i) and (ii)]{Geo-Ras-Sep-17-ptrf-2}. The proof of Theorem \ref{thm:CocycleMeasuresFace} relies primarily on Theorem \ref{thm:Cocycles}, which provides another construction of the queueing fixed points in \cite[Theorem 5.1]{Mai-Pra-03} 
    as well as the harmonic functions for RWREs considered in \cite[Theorem 4]{Yil-09-aop} and \cite[Theorem 3.1]{Yil-11-aop} (for directions where the quenched and averaged rate functions coincide), \cite[Lemma 1.6]{Yil-09-cpam} (the one-dimensional RWRE case), and  \cite[Example 5.6]{Geo-Ras-Sep-16} (the $L^2$ weak disorder regime). It can also be applied to exactly solvable models such as those in \cite{Bal-Ras-Sep-19, Geo-etal-15} to construct infinite-volume Gibbs measures, infinite geodesics, and generalized Busemann functions. In \cite{Yil-09-aop,Yil-11-aop,Geo-Ras-Sep-16}, stronger almost-sure limits are proved using martingales available in these settings. Similarly, almost-sure limits are obtained in \cite{Bal-Ras-Sep-19, Dam-Han-17,Geo-Ras-Sep-17-ptrf-2, Geo-etal-15, Jan-Ras-20-aop, New-95,Yil-09-cpam} by exploiting the geometric constraints on paths inherent to lower-dimensional cases.
        \end{remark}

	\begin{remark}
	    If $\facetUnr{\m,\sA}$ consists of a single point, e.g.\ when $\ppShapeUnrA{\beta,\usc}$ is strictly concave, then \eqref{CocycleMeasuresFace.dir}  says that the polymer has an asymptotic direction,  $\Qinf{x}{\sA, \beta, \m, \w}$-almost surely. Similarly, when $\facetRes{\m,\sA}$ is a singleton,  \eqref{CocycleMeasuresFace.LLNlimits} says that the polymer has an asymptotic velocity, $\Qinf{x}{\sA, \beta, \m, \w}$-almost surely.
	\end{remark}

    \begin{remark}\label{rmk:parametrizeDirections}
    A common way to parametrize the measures in the above theorem is  by directions $\xi \in(\ri \sA)\setminus\{\orig\}$, $\sA \in \faces$ instead of by vectors $\m$ in the superdifferential $\superDiffUnrA{\beta,\usc}(\ri \sA)$, $\sA\in\faces$. This can be done by applying Theorem \ref{thm:CocycleMeasuresFace} with $\m \in \ext(\subspaceA \cap \superDiffUnrA{\beta,\usc}(\xi))$.  Such an extreme point must exist, as explained earlier in this section. Then by homogeneity,
        $(\xi / \abs{\xi}_1) \cdot m = \ppShapeUnrA{\beta,\usc}(\xi / \abs{\xi}_1),$
        and $\xi / \abs{\xi}_1 \in \facetUnr{\m,\sA}$. As a result, $\bbP$-almost surely, under $\Qinf{x}{\sA,\beta,\m,\w}$, $X_{0:\infty}$ is almost surely directed into $\facetUnr{\m,\sA}$, which is a linear facet of $\ppShapeUnrA{\beta, \usc}$ that contains $\xi$. 
    \end{remark}

        \begin{remark}[FPP with a forbidden step] In Example \ref{ex:FPPforb}, applying the previous result to $\sA=\partial \cone = \{(x,y,0) : x,y\in\bbR\}$ with $\DctbA$ given by a countable dense set of directions of differentiability, we recover the construction of semi-infinite geodesics in planar Exponential FPP which can be obtained from the results in \cite{Dam-Han-14}. Because here we have $\orig \in \UsetA$, our conditions defining $\DctbA$ require us to know that such directions satisfy $\ppShapeUnrA{\infty}(\xi)\neq 0$. This follows from \cite[Theorems 3.1 and 6.1]{Kes-86-stflour}. Applying the result to $\sA = \cone = \{(x,y,z)\in \bbR^3 : z \geq 0\}$ with $\DctbA$ given by a countable dense set of interior directions of differentiability in $\cone$, we similarly obtain the existence of semi-infinite geodesics corresponding to a countable collection of faces of the limit shape defined by sub-level sets of $-\ppShapeUnrA{\infty}$ as illustrated in Figure \ref{fig:FPP3D}.  These have asymptotic directions contained in those faces by part \eqref{CocycleMeasuresFace.trans} of the theorem. Note that by a coupling with standard three-dimensional exponential FPP, it also follows from \cite[Theorems 3.1 and 6.1]{Kes-86-stflour}  that $\ppShapeUnrA{\infty}(\xi)\neq 0$ for $\xi \in \cone \setminus{\{\orig\}}$ here.
        \end{remark}

    \begin{remark}\label{rk:resAsUnr2} (Restricted-length models as unrestricted-length models)
        Remark \ref{rk:resAsUnr1} explained how one can rewrite a restricted-length model as an unrestricted-length one.
	We explain here how certain quantities for the restricted-length model transfer to ones for the unrestricted-length model.
    Consider a face $\Uset'$ of $\Uset$. Then the cone $\sA$ generated by $\Uset'$ is a face of the cone $\cone$ and $\rangeA=\range\cap\Uset'$. Continuing to use bars to denote the quantities in the unrestricted-length model,
	let $\overline\range'=\{\langle z,1\rangle:z\in\rangeA\}$
	and let $\overline\sA$ be the cone generated by this set. Then $\overline\range_{\overline\sA}=\overline\range'$, and
	$\overline\sA = \{\langle \zeta, t \rangle : \zeta\in\sA,t>0,\zeta/t \in \sU'\}\cup\{\langle\orig,0\rangle\}$. 
    Since, by definition,  
	$\overline \T_{\langle z,1\rangle}=\T_z$, for all $z\in\range$, we have $\overline\range_{\overline\sA}^{\id}=\{\langle z,1\rangle:z\in\rangeA^{\id}\}$.
	In particular, $\rangeA\ne\rangeA^{\id}$ implies $\overline\range_{\overline\sA}\ne\overline\range_{\overline\sA}^{\id}$ and $\overline\faces$ is the set of faces $\overline\sA$ such that $\sA\in\faces$. 
	
	If $\pote^+$ satisfies the conditions of Theorem \ref{Thm:ShapeRes}, then $\overline\pote^+$ satisfies the conditions of Theorem \ref{Thm:ShapeUnr}. Similarly, 
	Condition \ref{VConditionRes} transfers to Condition \ref{VCondition}. If $\pote^+$ satisfies Condition \ref{classLCondition}, then $\overline\pote^+$ also satisfies that condition. 

	Since for all $k\in\ZZ_{>0}$ and $x\in\Dn{k}$, $\overline{F}_{\langle \orig, 0\rangle, \langle x, k \rangle}^{\beta}=\freeres{\orig}{x}{k}{\beta}$ and $\overline\sI_{\langle x,k\rangle}=\sI_x$,
	\begin{align}\label{Esup-unr=res}
	\bbE\left[\sup_{n \geq 1} n^{-1} \bbE[\overline{F}_{\langle \orig, 0\rangle, n\langle x, k \rangle}^{\beta} \given \overline{\sI}_{\langle x, k \rangle}] \right] = \bbE\left[\sup_{n \geq 1} n^{-1} \bbE[\freeres{\orig}{nx}{nk}{\beta} \given \sI_x] \right].
	\end{align}
	Therefore, if the right-hand side is finite (as assumed, e.g., in Theorem \ref{Thm:ShapeRes}), then so is the left-hand side.
    In this case, Theorem \ref{Thm:ShapeUnr} implies that with $\bbP$-probability one,  $\overline\Lambda^{\beta}_{\overline\sA}$ is finite on $\overline\sA$. 
    Then, from \eqref{ppShapeRes} and \eqref{ppShapeUnr}, we get that, $\bbP$-almost surely, $\overline\Lambda^{\beta}_{\overline\sA}(\langle \zeta,t\rangle)=t\ppShapeResU{\beta}(\zeta/t)$ for all $\zeta\in\sA$ and $t>0$ with $\zeta/t\in\sU'$. 
    This implies that, with $\bbP$-probability one, $\ppShapeResU{\beta}$ is finite on $\sU'$, and as mentioned above Remark \ref{rk-cL}, this implies that, $\bbP$-almost surely, $\ppShapeResU{\beta,\usc}$ is the unique continuous extension of $\ppShapeResU{\beta}$ to $\sU'$. By Theorem \ref{Thm:ShapeUnr}, $\bbP$-almost surely, $\overline\Lambda^{\beta,\usc}_{\overline\sA}$ is the unique continuous extension of $\overline\Lambda^{\beta}_{\overline\sA}$ to $\overline\sA$. Therefore, with $\bbP$-probability one, $\overline\Lambda^{\beta,\usc}_{\overline\sA}(\langle \zeta,t\rangle)=t\ppShapeResU{\beta,\usc}(\zeta/t)$ for all $\zeta\in\sA$ and $t>0$ with $\zeta/t\in\sU'$.
    \end{remark}

    \begin{remark}[Extension to $|\range|=\infty$]
    It is natural to wonder to what extent our methods can be expected to generalize to the setting of an unbounded number of steps. Some extension in this direction is undoubtedly possible under some hypotheses on the reference walk and the potential. At this point, however, it is not clear to us what the correct conditions for such an extension are in general. The place to start any such program is the existence of the free energy and the shape theorem and then development of a variational representation of the free energy in terms of cocycles similar to \cite[Theorem 2.14]{Jan-Nur-Ras-22}. With these inputs in hand, our methods should mostly extend (as related methods have already done in the fully continuous setting of the KPZ equation and its associated continuum directed polymer \cite{Jan-Ras-Sep-22-1F1S-}), though one expects additional technical complications.
    \end{remark}

	The proof of Theorem \ref{thm:CocycleMeasuresFace} is given in Section \ref{sec:MainThmProofs}. Sections \ref{sec:coc} and \ref{sec:semiinf} develop a more general theory from which Theorem \ref{thm:CocycleMeasuresFace} follows.

 \section{Recovering cocycles}\label{sec:coc}
In this section, we construct the generalized Busemann functions mentioned in the introduction. We then use them to construct the semi-infinite polymer measures in Theorem \ref{thm:CocycleMeasuresFace}. 

	
	\begin{definition}\label{def:cocycle}
		Given a face $\sA \in \faces$ {\rm(}possibly $\cone$ itself\,{\rm)}, consider a Borel-measurable function 
		\[B : \{(x,y) : x, y \in \ZZ^d, y-x \in \RgroupA\} \times \Omega \rightarrow \RR.\] 
		\begin{enumerate}[label=\rm(\alph{*}), ref=\rm\alph{*}] \itemsep=2pt
		\item $B$ is said to be a cocycle,  if for $\bbP$-almost every $\omega$, $B(v, x, \omega) + B(x, y, \omega) = B(v, y, \omega)$ for all $v,x,y\in\ZZ^d$ with $x-v, y-x \in \RgroupA$.
		\item $B$ is $\T$-covariant if for $\bbP$-almost every $\omega$, $B(x+v, y+v, \omega) = B(x, y, \T_{v} \omega)$ for all $v,x,y\in\ZZ^d$ with $y-x \in \RgroupA$.  
		\item For $\beta\in(0,\infty]$, we say that $B$ $\beta$-recovers, 
		if for $\bbP$-almost every $\omega$, we have for all $x\in\ZZ^d$,
			\begin{align}
				&\sum_{z \in \rangeA} p(z) \exp \left\{ -\beta B(x, x+z, \w) - \beta \pote(\T_x \w, z) \right\} = 1, \text{ for } \beta < \infty \qquad \text{ and }\label{eqn:rec-beta} \\
				&\min_{z \in \rangeA} \left\{ B(x, x+z, \w) + \pote(\T_x \w, z) \right\} = 0, \text{ for } \beta = \infty.\label{eqn:rec-inf}
			\end{align}
		\item A cocycle $B$ on $\sA$ is said to be $L^1(\bbP)$ if $\bbE[|B(\orig, z)|] < \infty$ for all $z \in \rangeA$.  If also $\bbE[B(\orig, z)] = 0$ for all $z \in \rangeA$, then $B$ is centered.  
		\item An $L^1(\P)$ $\T$-covariant $\beta$-recovering cocyle is called a generalized Busemann function {\rm(}in inverse temperature $\beta${\rm)}.
		\end{enumerate}
	\end{definition}

	
	
	For a face $\sA \in \faces$ (possibly $\cone$ itself), recall that $\subspaceA$ is the linear span of $\rangeA$. Let $\invSig_\sA$ be the $\sigma$-algebra generated by the events that are invariant under $\T_z$ for all $z\in\RgroupA$. For a $\T$-covariant $L^1(\bbP)$ cocycle $B$ on $\sA$, the invariance of $\bbP$ under the shifts $\T_z$ implies that $\bbE[B(x,y) \given \invSig_\sA] = \bbE[B(\orig, y-x) \given \invSig_\sA]$ for all $x,y \in \RgroupA$. Consequently, there exists a unique (possibly random) vector $\mVec(B) \in \subspaceA$ such that 
    \begin{align}
	    \bbE[B(x, y) \given \invSig_\sA] = \mVec(B) \cdot (y-x). \label{meanVec}
	\end{align}  
	

	\begin{remark}[Space-time cocycles]\label{rmk:spaceTimeCocycle}
        If we write a restricted-length model as an unrestricted-length model, as in Remark \ref{rk:resAsUnr1}, then a cocycle in the unrestricted-length version corresponds to a \emph{space-time cocycle} in the restricted-length version. Precisely, let $\stRgroupA$ be the group generated by $\{\langle z,1\rangle:z\in\rangeA\}$. Lemma \ref{lm:stRgroupA} gives a characterization of this group. Consider a Borel-measurable function 
        \[B : \{(x,j,y,k)\in\ZZ^d\times\ZZ\times\ZZ^d\times\ZZ, \langle y-x,k-j\rangle\in\stRgroupA\} \times \Omega \longrightarrow \RR.\]
        \begin{enumerate}[label=\rm(\alph{*}), ref=\rm\alph{*}] \itemsep=2pt
        \item\label{rmk:spaceTimeCocycle.coc}  $B$ satisfies the space-time cocycle property if
        for $\bbP$-almost every $\w$, all $x,y,z\in\ZZ^d$, and all $j,k,\ell\in\ZZ$ with $\langle y-x,k-j\rangle,\langle z-y,\ell-k\rangle\in\stRgroupA$, 
        \begin{align}\label{st-coc}
        B(x, j, y, k, \omega) + B(y, k, z, \ell, \omega) = B(x, j, z, \ell, \omega).
        \end{align}
        \item $B$ is $\T$-covariant if for $\bbP$-almost every $\w$, $B(x+v, j+\ell, y+v, k+\ell, \w) = B(x, j, y, k, \T_v \w)$ for all $j, k, \ell \in \ZZ$ and $x,y,v\in\ZZ^d$ 
        such that $\langle y-x,k-j\rangle\in\stRgroupA$.
        In particular, $B(x,j+\ell,y,k+\ell,\w)=B(x,j,y,k,\w)$, i.e., the cocycle depends on the time coordinates only through the size of the time increment. 
        Consequently, in the case of a space-time directed polymer, described in Example \ref{examples}\eqref{itm:directedPolymers}, one 
        can drop the time coordinates and write $B(x,y)$, since $(y-x)\cdot\uhat$ determines the time increment.
        \item\label{rmk:spaceTimeCocycle.rec} For $\beta \in (0,\infty]$, the $\beta$-recovery property is that for $\bbP$-almost every $\w$, we have for all $x \in \ZZ^d$ and $j \in \ZZ$,
        \begin{align}
				&\sum_{z \in \rangeA} p(z) \exp \left\{ -\beta B(x, j, x+z, j+1, \w) - \beta \pote(\T_x \w, z) \right\} = 1, \text{ if } \beta < \infty,\quad\text{and}\label{st-recpos}\\
				&\min_{z \in \rangeA} \left\{ B(x, j, x+z, j+1, \w) + \pote(\T_x \w, z) \right\} = 0, \text{ if } \beta = \infty.\label{st-rec0}
	\end{align}
        \item A space-time cocycle is said to be $L^1(\bbP)$ if $\bbE[|B(\orig,0, z,1)|] < \infty$ for all $z \in \rangeA$.  If in addition $\bbE[B(\orig,0, z, 1)] = 0$ for all $z \in \rangeA$, then $B$ is centered.  
        \item For a shift-covariant $L^1(\bbP)$ space-time cocycle $B$, the random vector $\mVec(B) \in \subspaceA \times \RR$ is the unique vector such that $\E[B(x, j, y, k) \given \invSig_\sA] = \mVec(B) \cdot \langle y-x, k-j \rangle$ for all $x,y\in\ZZ^d$ and $j,k\in\ZZ$ such that $\langle x,j\rangle,\langle y,k\rangle\in \stRgroupA$.\qedhere 
   \end{enumerate}
    \end{remark}
    We next note a property of space-time cocycles which explains their relationship to the cocycles one sees when working with unrestricted-length models and which will allow us to connect our results to some of the previous literature. In words, it says that a space-time cocycle can be decomposed into a cocycle in the sense of Definition \ref{def:cocycle} plus a shift-invariant linear function of the time increment.
    
    \begin{lemma}\label{lem:ResCoc}
    Fix a face $\sA \in \faces$ {\rm(}possibly $\cone$ itself\,{\rm)}.
    Let $\B{}$ be a space-time $T$-covariant cocylce as in Remark \ref{rmk:spaceTimeCocycle}. 
    Then there exists an $\invSig_\sA$-measurable random variable $c$ and a cocycle $\overline B$ {\rm(}as in Definition \ref{def:cocycle}{\rm)} such that $\P$-almost surely, 
    $B(x,j,y,k)=\overline B(x,y)+c(k-j)$, for all $(x,j,y,k)\in\ZZ^d\times\ZZ\times\ZZ^d\times\ZZ$ with 
    $\langle y-x,k-j\rangle\in\stRgroupA$.
\end{lemma}

\begin{proof}
Let $A$ denote the set of $m\in\Z$ such that $\langle \orig,m\rangle\in\stRgroupA$. Note that $A$ is an additive subgroup of $\Z$. Fix $m\in A$.
Take any $z\in\rangeA$. Then $\langle z,1\rangle\in\stRgroupA$ and $\langle z,m+1\rangle\in\stRgroupA$
and so $\P$-almost surely,
\begin{align*}
\B{}(\orig,0,\orig,m,\w)+B(\orig,0,z,1,\w)&=\B{}(\orig,0,z,m+1,\w)=\B{}(\orig,0,z,1,\w)+B(z,0,z,m,\w)\\
&=\B{}(\orig,0,z,1,\w)+B(\orig,0,\orig,m,T_z\w).
\end{align*}
This implies that $\P$-almost surely, $\B{}(\orig,0,\orig,m,\w)=\B{}(\orig,0,\orig,m,T_z\w)$ for all $z\in\rangeA$. 
Consequently, for any $m\in A$, $\B{}(\orig,0,\orig,m,\w)$ is $\invSig_\sA$-measurable. Call $c_m=\E[\B{}(\orig,0,\orig,m)\,|\,\invSig_\sA]$, which satisfies $c_m=\B{}(\orig,0,\orig,m,\w)$, $\P$-almost surely.

If $m,n\in A$, then $n+m\in A$ and the cocycle property of $\B{}$ implies $c_{m+n}=c_m+c_n$. This implies that there exists an  $\invSig_\sA$-measurable random variable $c$ such that $c_m=cm$ for all $m\in A$. Then we have $\P$-almost surely, for any $m\in A$, $\B{}(\orig,0,\orig,m)=cm$.


By the covariance of $B$, we have
\[\B{}(x,k,x,j,\w)=\B{}(\orig,0,\orig,j-k,T_x\w)=c(j-k),\]
for $\P$-almost every $\w$ and any $x\in\ZZ^d$ and $(j,k)\in\ZZ\times\ZZ$ with $k-j\in A$.

Now, take $(x,y,j,k,m,n)\in\ZZ^d\times\ZZ^d\times\ZZ\times\ZZ\times\ZZ\times\ZZ$ with 
$\langle y-x,k-j\rangle\in\stRgroupA$ and $\langle y-x,n-m\rangle\in\stRgroupA$. 
Then
\[B(x,j,y,k)-B(x,m,y,n)=B(x,0,y,k-j)-B(x,0,y,n-m)=B(y,n-m,y,k-j)=c\bigl((k-j)-(n-m)\bigr).\]
Thus, if $x,y\in\ZZ^d\times\ZZ^d$ are such that $y-x\in\RgroupA$, then
\[\overline B(x,y)=B(x,j,y,k)-c(k-j)\quad\text{for any }j,k\in\Z\times\Z\text{ such that }\langle y-x,k-j\rangle\in\stRgroupA\]
is well defined and we have
\[B(x,j,y,k)=\overline B(x,y)+c(k-j)\]
as desired. 
\end{proof}

We document next a property of the random vector $\mVec(B)$, which will be of use to us in what follows.

\begin{lemma}\label{lem:E[m(B)]}
    Fix a face $\sA\in\faces$ {\rm(}possibly $\cone$ itself\,{\rm)}. Let $\B{}$ be an $L^1(\bbP)$ $\T$-covariant $\beta$-recovering cocycle on $\sA$.   
    If $\bbE[\mVec(B)] \in \superDiffUnrA{\beta, \usc}(\xi)$ for some $\xi \in \ri \sA$, then $\mVec(B) \in \superDiffUnrA{\beta, \usc}(\xi)$, $\bbP$-almost surely.  If additionally $\bbE[\mVec(B)] \in \ext\superDiffUnrA{\beta, \usc}(\xi)$, then $\mVec(B) = \bbE[\mVec(B)]$, $\bbP$-almost surely. 
\end{lemma}
\begin{proof}
    Let $\xi \in \ri \sA$ be such that $\bbE[\mVec(B)] \in \superDiffUnrA{\beta, \usc}(\xi)$.  By Lemma \ref{lem:concaveDualHomogeneous}, this implies that $\ppShapeUnr{\beta}(\xi) = \ppShapeUnrA{\beta,\usc}(\xi) = \bbE[\mVec(B)]\cdot\xi$. By the variational formula \cite[Theorem 2.14]{Jan-Nur-Ras-22}, we have $\bbP$-almost surely, $\ppShapeUnr{\beta}(\zeta) \leq \mVec(B)\cdot\zeta$ for all $\zeta \in \sA$.  Thus, it must be that, $\bbP$-almost surely, 
    \[ 
    0 = \ppShapeUnr{\beta}(\xi) - \mVec(B)\cdot \xi \leq \sup_{\zeta \in \sA}(\ppShapeUnr{\beta}(\zeta) - \mVec(B)\cdot \zeta) \leq 0.
    \] 
    Then applying \cite[Theorem 4.21]{Ras-Sep-15-ldp}, we obtain that $\mVec(B) \in \superDiffUnrA{\beta, \usc}(\xi)$.  If additionally $\bbE[\mVec(B)]$ is an extreme point of $\superDiffUnrA{\beta, \usc}(\xi)$, then $\mVec(B) = \bbE[\mVec(B)]$ $\bbP$-almost surely by the definition of an extreme point.  
\end{proof}

    For a face $\sA$ of $\cone$ and a non-empty set $I \subset \ZZ^d$, define the set of points reachable from some point of $I$ taking steps in $\rangeA$,
    \begin{equation}
    I_{\sA}^{\ge} = \{y \in \ZZ^d : y-x \in \RsemiA \text{ for some } x \in I \}. \label{def:Igeq}
    \end{equation}
	Also, define the set of points unreachable from all points in $I$ with steps in $\rangeA$,
	\begin{equation}
	I_{\sA}^{<} = \ZZ^d \setminus I_{\sA}^{\ge} = \{y \in \ZZ^d : y-x \not\in \RsemiA \text{ for all } x \in I \}. \label{def:Ileq}
	\end{equation} 
	If the model is undirected on $\sA$, i.e.\ $\orig\in\ri\UsetA$, then $\RsemiA = \RgroupA$ so that $I_{\sA}^{\ge} = \bigcup_{x \in I} (x + \RgroupA)$. 

    Recall the countable (possibly finite or empty) set of inverse temperatures and extremal supergradients $\DctbA$ defined in the paragraph containing \eqref{mIsExtreme}.
	For each $\sA \in \faces$, let $\sE_\sA = \{ (\beta, m, x, z) : (\beta,m) \in \DctbA, x \in \ZZ^d, z \in \rangeA \}$ and let $\Omegahat = \Omega \times \prod_{\sA \in \faces} \RR^{\sE_\sA}$ be equipped with the product topology and its Borel $\sigma$-algebra, $\Sighat$. 
	For $\omegahat \in \Omegahat$, let $\wProj$ be the projection to its $\Omega$ coordinate $\w$. We will sometimes write this as $\w(\omegahat)$. For  each $\sA \in \faces$ and $(\beta,\m,x,z)\in \sE_{\sA}$ let $\B{\sA,\beta, \m}(x, x+z, \omegahat)$ be the $(\sA,\beta, \m,  x, z)$-coordinate of $\omegahat$.  
Let $\That=\bigl\{\That_{v}: v \in \ZZ^d \bigr\}$ be the $\Sighat$-measurable group of transformations 
that map \[
	\left(\omega,\left\{b_{\sA, \beta, \m, x, z } : \sA \in \faces, (\beta, \m, x, z) \in \sE_{\sA} \right\}\right)
	\] to \[
	\left(\T_v \omega, \left\{b_{\sA, \beta, \m, x+v, z} : \sA \in \faces, (\beta, \m, x, z) \in \sE_{\sA}\right\}\right).
	\] 
As such, we have 	
$\pi_\Omega(\That_v\omegahat)=T_v\wProj$ and for all $\sA \in \faces$ and $(\beta,\m,x,z)\in \sE_\sA$, $\B{\sA,\beta, \m}(x, x+z,\That_v\omegahat)=\B{\sA,\beta,\m}(x+v,x+v+z,\omegahat)$.   Note that $\Omegahat$ satisfies the hypotheses on $\Omega$ in Section \ref{sec:rwrp}.


	\begin{theorem}\label{thm:Cocycles}
  For each face $\sA \in \faces$ {\rm(}possibly $\cone$ itself\,{\rm)} where $\DctbA \neq \varnothing$,   
	    assume that $\pote$ satisfies Conditions \ref{VCondition} and \ref{classLCondition} on $\sA$.
     There exists a $\That$-invariant probability measure $\bbPhat$ on $(\Omegahat, \Sighat)$ and a real-valued measurable function $\B{\sA, \beta, \m}(x, y, \omegahat)$ of $\{ (\sA, \beta, \m, x, y, \omegahat) : \sA \in \faces, (\beta, m) \in \DctbA, x \in \ZZ^d, y \in \ZZ^d, y-x \in \RgroupA, \omegahat \in \Omegahat \}$ such that the following hold. 
		\begin{enumerate} [label=\rm(\alph{*}), ref=\rm\alph{*}] \itemsep=2pt  
		\item\label{Cocycles.a} For any event $A \in \Sig$, $\bbPhat(\wProj \in A) = \bbP(A)$.
		
		\item\label{Cocycles.c} There exists a $\That$-invariant event $\OmegaCoc$ with $\bbPhat(\OmegaCoc) = 1$ such that for each face $\sA\in \faces$, $\omegahat \in \OmegaCoc$, $u, v, x,y \in \ZZ^d$ such that $x-u \in \RgroupA$, $y-x \in \RgroupA$, and $(\beta,\m)\in\DctbA$, 
		\begin{align}
			&\B{\sA, \beta, \m}(x+v, y+v, \omegahat) = \B{\sA, \beta, \m}(x, y, \That_v \omegahat)\label{Bcov} \\
			&\B{\sA, \beta, \m}(u, x, \omegahat) + \B{\sA, \beta, \m}(x, y, \omegahat) = \B{\sA, \beta, \m}(u, y, \omegahat)\label{Bcoc} \\
			&\sum_{z \in \rangeA} p(z) \exp \Bigl\{ -\beta \B{\sA, \beta, \m}(x, x+z, \omegahat) - \beta \pote(\omega(\That_x \omegahat), z) \Bigr\} = 1, \text{ if } \beta < \infty \label{Brecbeta}\\
			&\min_{z \in \rangeA} \Bigl\{ \B{\sA, \beta, \m}(x, x+z, \omegahat) + \pote(\omega(\That_x \omegahat), z) \Bigr\} = 0, \text{ if } \beta = \infty.\label{Brecinf}
		\end{align}
		
		\item\label{Cocycles.b}  For any face $\sA\in \faces$, $(\beta, \m)\in\DctbA$, and $x,y\in\ZZ^d$ such that $y-x \in \RgroupA$, $\B{\sA, \beta, \m}(x, y)$ is integrable under $\bbPhat$, 
		\begin{equation}
			\bbEhat[\B{\sA, \beta, \m}(x, y)] = \m \cdot (y-x), 
		\end{equation}
        and $\bbEhat[\mVec(\B{\sA, \beta,\m})] = m$.
        
        \item\label{Cocycles.d} Fix a face $\sA\in\faces$. If the random variables $\{(\pote(T_v\w,z))_{z\in\rangeA}:v\in\ZZ^d\}$ are independent under $\bbP$, then  for any non-empty set $I \subset \ZZ^d$, the variables 
		\begin{equation}\label{Ivars}
		\bigl\{ \B{\sA, \beta, \m}(x, y, \omegahat) : x \in I, y-x \in \RsemiA, (\beta, m) \in \DctbA \bigr\}
		\end{equation}
        are independent, under $\bbPhat$, of the variables $\{\pote(T_v\w,z):z\in\rangeA, v \in I_{\sA}^{<}\}$. 		
		\end{enumerate}
	\end{theorem}


    \begin{remark}[Restricted-length space-time cocycles]\label{rmk:obtainingSTCocycles}
    The above can be applied to  restricted-length polymers as follows.
  %
    Let $\Uset'$ be a face of $\Uset$ with $\rangeA \neq \rangeA^{\id}$, where $\sA$ is the cone generated by $\Uset'$. Let $\beta \in (0,\infty]$. 
     Write the restricted-length model in terms of an unrestricted-length model as described in Remark \ref{rk:resAsUnr1}. 
    If $\pote$ satisfies Conditions \ref{VConditionRes} and \ref{classLCondition}, it satisfies the conditions of Theorem \ref{Thm:ShapeRes}, and $\overline\pote$ satisfies Conditions \ref{VCondition} and \ref{classLCondition} on $\overline\sA$, and hence it satisfies the conditions of Theorem \ref{Thm:ShapeUnr}. As explained in Remark \ref{rk:resAsUnr2}, with $\bbP$-probability one, $\ppShapeResU{\beta,\usc}$ is finite on $\Uset'$, $\overline\Lambda^{\beta,\usc}_{\overline\sA}$ is finite on $\overline\sA$, and $\overline\Lambda^{\beta,\usc}_{\overline\sA}(\langle \zeta,t\rangle)=t\ppShapeResU{\beta,\usc}(\zeta/t)$ for all $\zeta\in\sA$ and $t>0$ with $\zeta/t\in\sU'$. Therefore, if $\ppShapeResU{\beta,\usc}$ is $\bbP$-almost surely deterministic on $\Uset'$, then $\overline\Lambda^{\beta,\usc}_{\overline\sA}$ 
    is $\bbP$-almost surely deterministic on $\overline\sA$. Observe next that if $\xi\in\ri\Uset'$, then $\overline\xi=\langle\xi,1\rangle\in(\ri\overline\sA)\setminus\{\langle\orig,0\rangle\}$. 
    Take $\overline\m\in\overline W_{\overline\sA}\cap\partial\overline\Lambda^{\beta,\usc}_{\overline\sA}(\langle\xi,1\rangle)$ for some $\xi\in\ri\Uset$. (This set is non-empty, bounded, and has at least one extreme point, as explained at the beginning of Section \ref{sec:main}.)
  Now all the hypotheses of Theorem \ref{thm:Cocycles} are satisfied for the unrestricted-length model, and the theorem produces the cocycles $\overline{B}^{\overline \sA, \beta, \overline\m}(\langle x,k \rangle, \langle y, \ell \rangle, \omegahat)$, $\langle y-x,k-\ell\rangle\in\overline\sG_{\overline\sA}$.  
  Then, for the original restricted-length model, we define the space-time cocycles (see Remark \ref{rmk:spaceTimeCocycle}) by setting
    	\[
    	\B{\Uset', \beta, \overline\m}(x, k, y, \ell, \omegahat) = \overline{B}^{\overline \sA, \beta, \overline\m}(\langle x,k \rangle, \langle y, \ell \rangle, \omegahat).
    	\]
    In particular, if we write $\overline\m=\langle\m,c\rangle$, then 
        \[\bbEhat[\B{\Uset', \beta, \overline\m}(x,k ,y,\ell)] = 
    	\bbEhat[\overline{B}^{\overline \sA, \beta, \overline\m}(\langle x,k \rangle, \langle y, \ell \rangle)]
    	=\overline\m \cdot \langle y-x,\ell-k\rangle=\m\cdot(y-x)+c(\ell-k).\]
    A direct computation given in Appendix \ref{app:conv} shows that
    \begin{align}\label{comp:c}
    c=\ppShapeResU{\beta,\usc}(\xi) - \m\cdot\xi;
    \end{align}
        moreover, $\m$ is in the superdifferential at $\xi$ of the concave function that is equal to $\Lambda^{\beta,\usc}_{\sA}$ on $\Uset'$ and is set to $-\infty$ outside $\Uset'$.  When $\invSig_\sA$ is trivial, this is also the value of $c$ from Lemma \ref{lem:ResCoc}.
    \end{remark}
    
    \begin{remark}[RWRE on $\bbZ$]\label{ex2:nRWRE} Consider the models in Figure \ref{fig:RWRE}. As noted in Remark \ref{rk:RWREDA}, in the recurrent case, there are no points in $\DctbC$ with $\beta=1$ that can serve as input for Theorem \ref{thm:Cocycles}. In contrast, in the transient cases, there exists a unique point $(1,\m)$, where $\m=-\ppShapeUnr{1}(-1)$, that can belong to $\DctbC$. Using this as input in Theorem \ref{thm:Cocycles} yields a cocycle $\B{\cone,1,\m}$, which we will denote more concisely as ${\overline B}^{\m}$.

    Next, consider the restricted-length version of the RWRE, reformulated as an unrestricted-length model per Remark \ref{rmk:obtainingSTCocycles}. Here, the set of admissible steps becomes $\overline\range=\{\langle1,1\rangle,\langle-1,1\rangle\}$, and $\langle0,0\rangle\not\in\overline\Uset$. Consequently, if we take the full cone $\overline\cone\subset\R^2$ generated by $\overline\range$, there are no restrictions on the choice of $\xi$ in defining $\DctbCbar$.

    Remark \ref{rmk:obtainingSTCocycles} explains how the superdifferential of the limiting free energy for this unrestricted-length model connects back to the superdifferential of 
$\ppShapeRes{1}$, which is either the single point 
$\frac{d}{dt}\ppShapeRes1(t)$ for $t\in(-1,1)\setminus\{0\}$, or an interval $[0,\frac{d}{dt}\ppShapeRes{1}(0-)]$ when $t=0$. See Figure \ref{fig:RWRE}.

Using this setup, Theorem \ref{thm:Cocycles} produces space-time cocycles (as described in Remark \ref{rmk:obtainingSTCocycles}) for a countable collection of $\m\in\partial\ppShapeRes{1}(t)$ with $t\in(-1,1)$. By Lemma \ref{lem:ResCoc}, these cocycles take the form
\[\B{\m}(x,j,y,k)=\overline{B}^{\m}(x,y)+c(k-j),\]
where 
$\overline{B}^{\m}$ is a cocycle and $c$ is a deterministic constant.

The corresponding Gibbs measure, constructed via Theorem \ref{thm:CocycleMeasuresFace}, is the Doob 
$h$-transform of the RWRE, with the one-step transition probability kernel
\[\pi^\m_z=\pi_z e^{\overline{B}^\m(\orig,z)+c}.\] 
These measures were previously constructed and studied in \cite{Yil-09-cpam} as the minimizers of the level-3 to level-1 contraction principle for quenched large deviations in a more general one-dimensional RWRE model. See, in particular, Lemma 1.6, Definition 5.6, and Theorem 5.17 of that paper. Our constant $c$ corresponds to the parameter $r$ in their notation. 
    
    $\B{\m}$ is a genuine cocycle if and only if  $c=0$, in which case $\B{\m}=\overline{B}^\m$. In this case,  (5.15) and (5.17) in \cite{Yil-09-cpam} imply that this cocycle coincides with the one obtained in the first paragraph of this remark, when applying Theorem \ref{thm:Cocycles} to the unrestricted-length RWRE. Furthermore, by \eqref{comp:c}, $c=0$ implies $\ppShapeRes{1}(t)=\m t$ and, by the observation following \eqref{comp:c}, $\m$ is also in the superdifferential of $\ppShapeUnr{1}$. This forces $\m\in\{\ppShapeUnr{1}(-1),0\}$. Thus, in the transient nestling cases, the points $t\in(-1,0)$ for which $c=0$ correspond precisely to the points in the interval left of $0$ where $\ppShapeRes{1}$ is linear (see Figures \ref{fig:RWREtv0} and \ref{fig:RWREmnest}). In the non-nestling case, there exists a unique such $t$, which is the sole point in $(-1,0)$ where $\ppShapeRes{1}$ and $\ppShapeUnr{1}$ intersect. See Figure \ref{fig:RWREnnest}.
    
    We conclude this remark with a high-level explanation of the significance of the unique intersection point mentioned at the end of the previous paragraph. Let $t_0\in(-1,0)$ denote this point. By the maximum entropy principle (see \cite[Section 5.3]{Ras-Sep-15-ldp} for a similar application), conditioning the RWRE on $\{\tau_{-n}<\infty\}$ and taking $n\to\infty$ causes the path measure to converge almost surely to the Doob $h$-transformed RWRE tilted by the cocycle $\overline{B}^{\m}$, with $\m=-\ppShapeUnr{1}(-1)$. But since $\overline{B}^{\m}=\B{\m}$, 
    Lemma 1.6 and Theorem 5.17 in \cite{Yil-09-cpam} along with another application of the maximum entropy principle imply that this tilted RWRE corresponds exactly to conditioning the original RWRE on having the asymptotic velocity $t_0$. Succinctly, if we condition the walk on the (zero probability) event $\{X_n \to -\infty\}$, it will proceed ballistically with velocity $t_0$.  
    \end{remark}

    \begin{remark}
	Our construction shows that the claim in Theorem \ref{thm:Cocycles}\eqref{Cocycles.d} continues to hold if independence is replaced by mixing. That is, if  $\sigma\{(\pote(T_v\w,z))_{z\in\rangeA}:v\in A\}$ and $\sigma\{(\pote(T_v\w,z))_{z\in\rangeA}:v\in A'\}$ mix under $\bbP$ at rate $r_k$, when $A,A'\subset\ZZ^d$ are separated by distance $k$, then $\sigma\bigl\{ \B{\sA, \beta, \m}(x, y, \omegahat) : x \in I, y-x \in \RsemiA, (\beta, m) \in \DctbA \bigr\}$ and $\sigma\{\pote(T_v\w,z):z\in\rangeA, v \in I'\}$ mix at the same rate $r_k$ if $I\subset\ZZ^d$ and $I'\subset I_{\sA}^{<}$ are separated by distance $k$. See \cite{Bra-05} for a definition and basic properties of the strong mixing rate between $\sigma$-algebras.
	\end{remark}

    We prove the theorem towards the end of the section. 
    For a high-level summary of the ideas, see Section \ref{sec:MRW}  in the introduction.
    We begin by defining a few objects and proving a few bounds. Assume the hypotheses of the above theorem for the rest of this section.
    
    For a face $\sA\in \faces$ and a pair $(\beta, \m) \in \DctbA$, $\m \in \subspaceA \cap \superDiffUnrA{\beta,\usc}(\xi)$ for some $\xi \in \ri \sA \setminus \{\orig\}$.  
    Define $\uhat, h \in \RR^d$ 
    in the following way:
    
    If either Condition \ref{VCondition}\eqref{condBddBelowByC} or \ref{VCondition}\eqref{condProdMeasure} holds, then $\orig \not\in\UsetA$.  \cite[Lemma A.1]{Jan-Nur-Ras-22} implies that there exists a vector $u$ such that $u \cdot z > 0$ for all $z \in \rangeA$.  Then $u \cdot \xi > 0$.  Define $\uhat = u / (u \cdot \xi)$ and $h = \uhat - m$. 


    If instead Condition \ref{VCondition}\eqref{condBddBelow0} holds, then $\orig \in \UsetA$ and by the definition of $\DctbA$ (above \eqref{mIsExtreme}), we have that $\ppShapeUnr{\beta}(\xi) \neq 0$ and hence $\ppShapeUnr{\beta}(\xi) < 0$.
    Also by Remark \ref{rmk:finiteLambda}, $\ppShapeUnr{\beta}(\xi) > -\infty$.  Define $\uhat = \frac{\m}{\ppShapeUnr{\beta}(\xi)}$ and $h = \frac{1-\ppShapeUnr{\beta}(\xi)}{\ppShapeUnr{\beta}(\xi)} \m$.  
    Note that $\uhat \cdot \xi = \frac{\m}{\ppShapeUnr{\beta}(\xi)} \cdot \xi = 1$, since $\m \cdot \xi = \ppShapeUnrA{\beta,\usc}(\xi) = \ppShapeUnr{\beta}(\xi)$ by Lemma \ref{lem:concaveDualHomogeneous}.  

   With the above definitions, we can now consider that throughout the proof, each $\sA\in \faces$ and  $(\beta, \m) \in \DctbA$ are accompanied by $\xi$, $\uhat$, and $h$ that satisfy these properties:
    \begin{equation}\label{eq:uhatHConditions}
    \begin{aligned}
        &\m = \uhat - h, \\
        &\uhat \cdot \xi = 1, \\
        &\uhat \cdot z > 0 \text{ for all } z \in \rangeA, \text{ if either Condition \ref{VCondition}\eqref{condBddBelowByC} or \ref{VCondition}\eqref{condProdMeasure} holds, and} \\
        &h = (1-\ppShapeUnr{\beta}(\xi)) \uhat, \text{ if Condition \ref{VCondition}\eqref{condBddBelow0} holds.}
    \end{aligned}
    \end{equation}
 
    Let 
    \begin{align}\label{R0}
    \maxz_0 = \sum_{z \in \range} \abs{z \cdot \uhat}
    \end{align}
    and take $R>\maxz_0$.  
    For $x \in \ZZ^d$ and $t \in \RR$, we introduce the following \textit{slab at $\uhat$-level t of width R},
    \[\Lt{\sA,x,t} = \Lt{\sA,x,t}(\maxz, \uhat) = \{v \in \ZZ^d : v-x \in \RgroupA \text{ and } t \leq v \cdot \uhat <  t + \maxz\}.\] 
    Note that if $y \in \ZZ^d$ with $y-x \in \RgroupA$, then $\Lt{\sA,x,t} = \Lt{\sA,y,t}$. 
    We abbreviate $\Lt{\sA,\orig,t} = \Lt{\sA,t}$.  For $\beta < \infty$, $x \in \ZZ^d$, $t>x \cdot \uhat$, and $\e \geq 0$ define the (regularized) tilted \emph{point-to-level partition function} 
	\[
    \prtlevn{x}{t}{\sA, \beta, \e}(\w)= \sum_{v\in\Lt{\sA, x,t}} \prtunr{x}{v}{\beta}(\w) e^{\beta h \cdot (v-x) -\e\abs{v}_1}.
	\]
    Above, the regularization by introducing $\e|v|_1$ in the exponent is there solely to ensure finiteness.

    Recall the convention that $\prtunr{x}{v}{\beta} = 0$ if there are no admissible paths from $x$ to $v$. Hence, the above sum is only over sites $v$ that are accessible from $x$, i.e., $v-x\in\RsemiA$. Since $\xi \cdot \uhat = 1$ there exists a $z\in\rangeA$ such that $z\cdot\uhat>0$. Since $z\cdot\uhat<\maxz$ and $x\cdot \uhat < t$, there is at least one $v \in \Lt{\sA, x,t}$ which is reachable from $x$.  
    The corresponding free energy is
	\[
	\freelevn{x}{t}{\sA, \beta, \e}(\omega) = \frac{1}{\beta} \log  \prtlevn{x}{t}{\sA, \beta, \e}(\w).
	\] 
    At zero temperature, define 
    \[
    \freelevn{x}{t}{\sA, \infty, 0}(\w) = \sup_{v\in\Lt{\sA, x,t}} (\freeunr{x}{v}{\infty}(\w) +  h\cdot (v-x)).
    \]
    Recall the convention that $\freeunr{x}{v}{\infty} = -\infty$ if there are no admissible paths from $x$ to $v$, meaning that $v$ is not considered in the supremum.  
 
	Condition \ref{VCondition} guarantees that $\freelevn{x}{t}{\sA, \beta,\e}(\w)$ is finite and integrable, which we now show. 
    \begin{lemma}\label{lem:FneFinite}
     If $\beta = \infty$ or Condition \ref{VCondition}\eqref{condBddBelowByC} or \eqref{condProdMeasure} is satisfied, take $\e = 0$.  If $\beta < \infty$ and Condition \ref{VCondition}\eqref{condBddBelow0} holds, take $\e > 0$.  Then, for each $\sA \in \faces$, $x \in \ZZ^d$, $\beta \in (0,\infty]$, and $t>x \cdot \uhat$, 
    \[
    \bbE[|\freelevn{x}{t}{\sA, \beta, \e}|] < \infty.
    \] 
    \end{lemma}
    \begin{proof}     
    Let $\sA, x, \beta, \e,$ and $t$ be as in the statement.  Choose $y$ such that $y \in \Lt{\sA, x,t}$ and $y-x \in \RsemiA$.  The hypothesis on $t$ and the definition of $\maxz$ ensure that such a $y$ exists. For such $y$, let $x_{0:n}$ be an admissible path with $x_0=x$ and $x_n=y$.  Then for $\beta \in (0,\infty)$, we have
\begin{align*}
\freelevn{x}{t}{\sA, \beta, \e} &= \frac{1}{\beta} \log \sum_{v \in \Lt{\sA, x,t}} \prtunr{x}{v}{\beta} e^{\beta h \cdot (v-x) -\e \abs{v}_1} \geq \frac{1}{\beta} \log \Bigl(\prtunr{x}{y}{\beta} e^{\beta h \cdot (y-x) -\e \abs{y}_1}\Bigr), \text{ and } \\
\freelevn{x}{t}{\sA, \infty, 0} &\geq \freeunr{x}{y}{\infty} + h \cdot (y-x).
\end{align*}
Recall that $\freeunr{x}{y}{\infty}\ge\frac{1}{\beta} \log \prtunr{x}{y}{\beta}$ and then bound
\begin{align*}
\E\Bigl[\frac{1}{\beta} \log \prtunr{x}{y}{\beta}\Bigr] \geq \sum_{i=0}^{n-1}\left[\log p(x_{i+1}-x_i)-\E\left[\pote(T_{x_i}\w, x_{i+1}-x_{i})\right] \right]>-\infty.
\end{align*}

For the upper bound, first consider the case where Condition \ref{VCondition}\eqref{condBddBelow0} is satisfied so that $\pote(\w, z) \geq 0$ for all $z \in \rangeA$.  
Then for all $v$ such that $v-x\in \RsemiA$, $\prtunr{x}{v}{\beta} = \rwE_x\Bigl[ e^{-\beta \sum_{i=0}^{\stoppt{v}-1} \pote(\T_{X_i} \w, X_{i+1}-X_i)} \one_{\{\stoppt{v} < \infty\}} \Bigr] \leq 1$. Recall that $\ppShapeUnr{\beta}(\xi) < 0$ and $h = (1-\ppShapeUnr{\beta}(\xi)) \uhat$.  Then, with the convention that $C$ is a chameleon constant that may change line-to-line,
\begin{align}
\freelevn{x}{t}{\sA, \beta, \e} &= \frac{1}{\beta} \log \sum_{v \in \Lt{\sA, x,t}} \prtunr{x}{v}{\beta} e^{\beta h \cdot (v-x) -\e \abs{v}_1}
\leq \frac{1}{\beta} \log \sum_{v \in \Lt{\sA, x,t}} e^{\beta (1-\ppShapeUnr{\beta}(\xi))\uhat \cdot v - \beta h \cdot x -\e \abs{v}_1} \nonumber \\ 
&\leq \frac{1}{\beta} \log \sum_{v \in \Lt{\sA, x,t}} e^{\beta (1-\ppShapeUnr{\beta}(\xi))(t+\maxz) - \beta h \cdot x -\e \abs{v}_1}
= (1-\ppShapeUnr{\beta}(\xi))(t+\maxz) - h\cdot x + \frac{1}{\beta} \log \sum_{v \in \ZZ^d} e^{-\e \abs{v}_1} \nonumber \\
&\leq (1-\ppShapeUnr{\beta}(\xi))(t+\maxz) - h\cdot x + \frac{1}{\beta} \log\bigg( C \sum_{k=0}^\infty k^d e^{-\e k} \bigg)\nonumber \\
&\leq (1-\ppShapeUnr{\beta}(\xi))(t+\maxz) - h\cdot x + \frac{1}{\beta} \log\bigg( C\sum_{k=0}^\infty e^{-\e k/2}\bigg)\nonumber \\
&\le (1-\ppShapeUnr{\beta}(\xi))(t+\maxz) - h\cdot x + \beta^{-1}\log (C\e^{-1}), \label{eqn:L1UpperBoundCase1}
\end{align}
where the last bound applies for $\e$ sufficiently small.
At zero temperature, the equivalent bound is
\begin{align}
\freelevn{x}{t}{\sA, \infty, 0} &= \sup_{v \in \Lt{\sA, x,t}} (\freeunr{x}{v}{\infty}(\w)+h\cdot(v-x))  \leq \sup_{v \in \Lt{\sA, x,t}} ((1-\ppShapeUnr{\beta}(\xi))v\cdot\uhat - h\cdot x) \nonumber \\
&\leq (1-\ppShapeUnr{\beta}(\xi))(t+\maxz) - h\cdot x. \label{eqn:L1UpperBoundCase1Zero}
\end{align}

Next, assume that either Condition \ref{VCondition}\eqref{condBddBelowByC} or \eqref{condProdMeasure} holds. Then $\orig\not\in\UsetA$ and there exists a $\delta>0$ such that for all $z \in \rangeA$, $\uhat \cdot z \geq \delta > 0$.  For $v \in L_{\sA, x,t}$, all paths from $x$ to $v$ have length at least $(t-x\cdot\uhat) \maxz^{-1}$ since $\maxz > \sum_{z \in \rangeA} \abs{z \cdot \uhat}$.  Similarly all paths have length at most $(t + \maxz-x\cdot\uhat)\delta^{-1}$  since $z \cdot \uhat \geq \delta$ for all $z \in \rangeA$.  
Let $\ell = \ell(t) = \lfloor (t + \maxz-x\cdot\uhat)\delta^{-1}\rfloor$ be the maximum possible length.  Let $\tilde{L}_{\sA, x,t}$ be the finite set of $v \in \Lt{\sA, x,t}$ which are reachable from $x$.  The cardinality of $\tilde{L}_{\sA, x,t}$ is bounded above by $C\ell^d$ for some fixed constant $C > 0$ depending only on the dimension $d$ and $\rangeA$. 

Suppose Condition \ref{VCondition}\eqref{condBddBelowByC} holds
so that $\pote(\w,z) \geq c$ for some $c \in \R$ and all $z \in \rangeA$.
 Then,
\begin{align}
    \freelevn{x}{t}{\sA, \beta, 0} &= \frac{1}{\beta} \log \sum_{v \in \Lt{\sA, x,t}} \prtunr{x}{v}{\beta} e^{\beta h \cdot (v-x)}
    = \frac{1}{\beta} \log \sum_{v \in \Lt{\sA, x,t}} \rwE_x\Bigl[ e^{-\beta \sum_{i=0}^{\stoppt{v}-1} \pote(\T_{X_i}\w, X_{i+1}-X_i)} \one_{\{\stoppt{v}<\infty\}} \Bigr] e^{\beta h \cdot (v-x) } \nonumber\\
    &\leq \frac{1}{\beta} \log \sum_{v \in \Lt{\sA, x,t}} \rwE_x\Bigl[ e^{\beta \ell\abs{c} } \one_{\{\stoppt{v}<\infty\}} \Bigr] e^{\beta h \cdot (v-x) }
    \leq \frac{1}{\beta} \log \sum_{v \in \tilde{L}_{\sA, x,t}} e^{\beta \ell\abs{c} + \beta h \cdot (v-x) } \nonumber\\
    &\leq \ell\abs{c}  + \frac{1}{\beta} \log C\ell^d  + \max_{v \in \tilde{L}_{\sA, x,t}} h \cdot(v-x)\nonumber\\
    &\leq \ell \abs{c}  + \frac{1}{\beta}  \log C\ell^d + \ell d \max_{1\le i \le d} |h_i| \cdot \max_{\substack{z \in \rangeA \\ 1 \le i \le d }} |z_i|.  \label{eqn:L1upperBoundBddBelowC}
\end{align}
The last line follows since $v$ is at most $\ell$ steps from $x$.  At zero temperature, the equivalent bound is
\begin{align}
    \freelevn{x}{t}{\sA, \infty, 0} &= \sup_{v \in \Lt{\sA, x,t}} (\freeunr{x}{v}{\infty} + h\cdot(v-x)) \leq \sup_{v \in \tilde{L}_{\sA, x,t}}  (\ell \abs{c} + h\cdot(v-x)) \nonumber\\
    &\leq \ell\abs{c} + \ell d \max_{1\le i \le d} |h_i| \cdot \max_{\substack{z \in \rangeA \\ 1 \le i \le d }} |z_i|.\label{eqn:L1upperBoundBddBelowCZero}
\end{align}

Lastly, suppose Condition \ref{VCondition}\eqref{condProdMeasure} is satisfied.
Then in the positive temperature case,
\begin{align}
\freelevn{x}{t}{\sA, \beta, 0} &= \frac{1}{\beta} \log \sum_{v \in \Lt{\sA,x, t}} \prtunr{x}{v}{\beta} e^{\beta h \cdot v} \leq \frac{1}{\beta} \log \bigl(C \ell^{d} \max_{v \in \tilde{L}_{\sA, x,t}} \prtunr{x}{v}{\beta} e^{\beta h \cdot v})  
\leq  \max_{v \in \tilde{L}_{\sA, x,t}} \frac{1}{\beta} \log e^{\beta \freeunr{x}{v}{\infty}+\beta h \cdot v} + \frac{1}{\beta}\log C\ell^{d} \nonumber\\
&\leq \max_{(t-x\cdot\uhat) \maxz^{-1} \leq k \leq (t + \maxz-x\cdot\uhat)\delta^{-1} } \max_{v \in \tilde{L}_{\sA,x, t}} (\freeres{x}{v}{k}{\infty} +h\cdot v) + \frac{1}{\beta}\log C\ell^{d}. \label{eqn:eqn:L1upperBoundIID}
\end{align}
Similarly, in the zero temperature case,
\begin{align}
\freelevn{x}{t}{\sA, \infty, 0} 
\leq \max_{(t-x\cdot\uhat) \maxz^{-1} \leq k \leq (t + \maxz-x\cdot\uhat)\delta^{-1} } \max_{v \in \tilde{L}_{\sA,x, t}} (\freeres{x}{v}{k}{\infty} +h\cdot v).\label{eqn:eqn:L1upperBoundIID2}
\end{align}
The bound now follows since $\bbE[\freeres{x}{v}{k}{\infty}] < \infty$ for each of the finitely many $k$ and $v$ because $\pote^-(\w,z) \in L^1(\P)$ for each $z \in \rangeA$.
\end{proof}

For each $t >  \max\{x \cdot \uhat,(x+y) \cdot \uhat\}$ and $\e \geq 0$, the free energy satisfies an approximate shift-covariance property, which is exact if $\e = 0$.
 For $\beta \in (0,\infty)$, start by writing
    \begin{align}
        \freelevn{x+y}{t}{\sA, \beta, \e}(\w) &= \frac{1}{\beta} \log \sum_{v \in \Lt{\sA, x+y,t}} \prtunr{x+y}{v}{\beta}(\w) e^{\beta h \cdot (v-x-y) -\e \abs{v}_1} 
        = \frac{1}{\beta} \log \sum_{v \in \Lt{\sA, x+y,t}} \prtunr{x}{v-y}{\beta}(\T_y \w) e^{\beta h \cdot (v-x-y)-\e \abs{v}_1} \nonumber \\
        &= \frac{1}{\beta} \log \sum_{v \in \Lt{\sA, x,t-y\cdot \uhat}} \prtunr{x}{v}{\beta}(\T_y \w) e^{\beta h \cdot (v-x) -\e \abs{v+y}_1}. \label{eqn:approxShiftCovDeriv}
        \end{align}
    Using the bounds $-\varepsilon \abs{v}_1 - \varepsilon \abs{y}_1 \leq -\varepsilon \abs{v+y}_1 \leq -\varepsilon \abs{v}_1 + \varepsilon \abs{y}_1$, we see that
    \begin{equation}\label{eqn:approxShiftCov}
    \Bigl|\freelevn{x+y}{t}{\sA, \beta, \e}(\w) - \freelevn{x}{t-(y\cdot\uhat)}{\sA, \beta, \e}(\T_y \w)\Bigr| \leq \frac{ \e \abs{y}_1}{\beta}. 
    \end{equation}
    If $\beta=\infty$ and $\e=0$, this shift-covariance is exact:
    \begin{align}
    \freelevn{x+y}{t}{\sA, \infty, 0}(\w) &= \sup_{v \in \Lt{\sA, x+y,t}} (\freeunr{x}{v-y}{\infty}(\T_y\w) + h\cdot(v-x-y)) \nonumber \\
    &= \sup_{v \in \Lt{\sA, x,t-(y\cdot\uhat)}} (\freeunr{x}{v}{\infty}(\T_y\w) + h\cdot(v-x)) = \freelevn{x}{t-(y\cdot\uhat)}{\sA, \infty, 0}(\T_y \w).  \label{eqn:approxShiftCovZeroTemp}
    \end{align}

	For $\sA \in \faces$, $(\beta,\m) \in \DctbA, \e \geq 0, t\in\R, x \in \ZZ^d,$  and  $z \in \rangeA$,  define 
	\begin{align}
		\Bn{t,\e}{\sA, \beta, \m}(x, x+z) &= \freelevn{x}{t}{\sA, \beta, \e}-\freelevn{x+z}{t}{\sA, \beta, \e} - h \cdot z,
	\end{align}
	if $t>\max\bigl\{x \cdot \uhat,(x+z) \cdot \uhat : z \in \rangeA \bigr\}$ 
    and 
	$\Bn{t,\e}{\sA, \beta, \m}(x, x+z) = 0$ otherwise.  
	
	\begin{lemma}
	Let $\sA \in \faces$, $(\beta, \m) \in \DctbA$, and $\e \geq 0$. Then $\P$-almost surely for all $x\in\ZZ^d$ and $t > \max\bigl\{x \cdot \uhat,(x+z) \cdot \uhat : z \in \rangeA \bigr\}$,
	\begin{align}
		&\sum_{z \in \rangeA} p(z) \exp \left\{ -\beta \Bn{t, \e}{\sA, \beta, \m}(x, x+z) - \beta \pote(\T_x \omega, z) \right\} = 1\quad\text{if } \beta < \infty \text{ and}\label{unrPosPreLimRecovery}\\
		&\min_{z \in \rangeA} \left\{ \Bn{t, \e}{\sA, \infty, \m} (x, x+z) + \pote(\T_x \omega, z) \right\} = 0\quad\text{if } \beta = \infty.\label{unrZeroPreLimRecovery}
	\end{align}
	\end{lemma}
	\begin{proof}
            The desired equations essentially come from the one-step decomposition of $\prtlevn{x}{t}{\sA, \beta, \e}$. 
		Take $\sA,\beta, \m, \e, x,$ and $t$ as in the statement. Recall that $\Lt{\sA, x,t} = \Lt{\sA, x+z,t}$ if $z \in \rangeA$.  Then 
	\begin{align*}
		&\sum_{z \in \rangeA} p(z) \exp \bigl\{ -\beta \Bn{t, \e}{\sA, \beta, \m}(x, x+z) - \beta \pote(\T_x \omega, z) \bigr\} \\
		&= \sum_{z \in \rangeA} p(z) \exp \bigl\{ -\beta( \freelevn{x}{t}{\sA, \beta, \e}-\freelevn{x+z}{t}{\sA, \beta, \e} -  h\cdot z ) - \beta \pote(\T_x \omega, z)\bigr\} \\
		&= \sum_{z \in \rangeA} p(z) \exp \bigl\{  \log \prtlevn{x+z}{t}{\sA, \beta, \e} -\log \prtlevn{x}{t}{\sA, \beta, \e}  + \beta h \cdot z - \beta \pote(\T_x \omega, z) \bigr\} \\
		&= ( \prtlevn{x}{t}{\sA, \beta, \e} )^{-1} \sum_{z \in \rangeA} p(z) e^{ \beta h \cdot z - \beta \pote(\T_x \omega, z) } \prtlevn{x+z}{t}{\sA, \beta, \e} \\
		&= (\prtlevn{x}{t}{\sA, \beta, \e})^{-1} \sum_{z \in \rangeA} p(z) e^{\beta h \cdot z - \beta \pote(\T_x \omega, z)} \\
		&\qquad\qquad\qquad\qquad\qquad
		\times\sum_{v \in \Lt{\sA, x+z,t}} e^{\beta h \cdot (v-x-z) -\e\abs{v}_1} \rwE_{x+z} \Bigl[ \exp \Bigl\{ -\beta \sum_{i=0}^{\stoppt{v}-1} \pote(\T_{X_i} \omega, X_{i+1}-X_i) \Bigr\} \one_{\{\stoppt{v}<\infty\}} \Bigr]   \\
		&= (\prtlevn{x}{t}{\sA, \beta, \e})^{-1} \sum_{v \in \Lt{\sA, x,t}}  e^{\beta h \cdot (v-x) -\e \abs{v}_1} \sum_{z \in \rangeA} p(z)  \rwE_{x} \Bigl[ \exp \Bigl\{ -\beta \sum_{i=0}^{\stoppt{v}-1} \pote(\T_{X_i} \omega, X_{i+1}-X_i) \Bigr\} \one_{\{\stoppt{v}<\infty\}} \Given X_1 = x+z \Bigr]   \\
		&= (\prtlevn{x}{t}{\sA, \beta, \e})^{-1}  \sum_{v \in \Lt{\sA, x,t}}  e^{\beta h \cdot (v-x) -\e \abs{v}_1} \rwE_{x} \Bigl[ \exp \Bigl\{ -\beta \sum_{i=0}^{\stoppt{v}-1} \pote(\T_{X_i} \omega, X_{i+1}-X_i)  \Bigr\} \one_{\{\stoppt{v}<\infty\}} \Bigr]   \\
		&= (\prtlevn{x}{t}{\sA, \beta, \e})^{-1}  \prtlevn{x}{t}{\sA, \beta, \e} = 1.
	\end{align*}
	The fifth equality above comes from the Markov property. The recovery property \eqref{unrPosPreLimRecovery} is proved.  For the zero temperature case, write 
	\begin{align*}
		&\min_{z \in \rangeA} \Bigl\{ \Bn{t, \e}{\sA, \infty, \m}(x, x+z) + \pote(\T_x \omega, z) \Bigr\} \\
        &= \min_{z \in \rangeA} \Bigl\{ \freelevn{x}{t}{\sA, \infty, 0}  - \freelevn{x+z}{t}{\sA, \infty, 0} - h \cdot z + \pote(\T_x \omega, z) \Bigr\} \\
        &= \min_{z \in \rangeA} \Bigl\{ \freelevn{x}{t}{\sA, \infty, 0} - \sup_{v \in \Lt{\sA, x+z,t}} \sup_{k \geq 1} \sup_{\substack{x_{0:k} \in \PathsPtPKilledRes{x+z}{v}{k} }} \Bigl\{ -\sum_{i=0}^{k - 1} \pote(T_{x_i} \omega, x_{i+1}-x_i) + h\cdot (v-x-z) \Bigr\}  - h \cdot z + \pote(\T_x \omega, z) \Bigr\}    \\
		  &= \freelevn{x}{t}{\sA, \infty, 0} -  \max_{z \in \rangeA} \sup_{v \in \Lt{\sA, x,t}} \sup_{k \geq 2}\sup_{\substack{x_{0:k} \in \PathsPtPKilledRes{x}{v}{k} \\ x_1 = x+z }} \Bigl\{ -\sum_{i=0}^{k - 1} \pote(T_{x_i} \omega, x_{i+1}-x_i) + h \cdot (v-x) \Bigr\}     \\
		&= \freelevn{x}{t}{\sA, \infty, 0} -  \sup_{v \in \Lt{\sA, x,t}} \sup_{k \geq 1} \sup_{\substack{x_{0:k} \in \PathsPtPKilledRes{x}{v}{k} }} \Bigl\{ -\sum_{i=0}^{k - 1} \pote(T_{x_i} \omega, x_{i+1}-x_i)  + h \cdot (v-x) \Bigr\}    \\
		&= \freelevn{x}{t}{\sA, \infty, 0} - \freelevn{x}{t}{\sA, \infty, 0} = 0.
	\end{align*}
	The recovery property \eqref{unrZeroPreLimRecovery} is proved.
	\end{proof}

The next lemma gives some information about where the minimum in \eqref{unrZeroPreLimRecovery} is attained. 

\begin{lemma}\label{lm:rec-gen-t}
Let $\sA \in \faces$ and let $\m$ be such that $(\infty, \m) \in \DctbA$. 
Then $\P$-almost surely, for any finite set $A\subset\ZZ^d$ and any $t>\max\{x\cdot\uhat,(x+z)\cdot\uhat:x\in A,z\in\rangeA\}$,
there exists $y\in A$ and $z\in\rangeA$ such that $y+z\not\in A$ and
$\pote(\T_y\w,z)+\Bn{t,0}{\sA, \infty, \m}(y,y+z)=0$.
\end{lemma}

\begin{proof}
    Take $A$ and $t$ as in the statement, $x\in A$ and $\e>0$. Let $v\in\Lt{\sA, x,t}$ be such that 
    \[\freelevn{x}{t}{\sA, \infty, 0}(\w) - \e \le  \freeunr{x}{v}{\infty}(\w) +  h\cdot (v-x)\le \freelevn{x}{t}{\sA, \infty, 0}(\w).\]
    Take a path $x_{0:n}=x_{0:n}(\e)\in\PathsPtPRes{x}{v}{n}$ such that
        \[\freeunr{x}{v}{\infty}(\w)-\e\le - \sum_{i=0}^{n-1} \pote(\T_{x_i}\w, x_{i+1}-x_i)\le \freeunr{x}{v}{\infty}(\w).\]
    Note $v\not\in A$ by the definition of $t$. Let $k=k(\e)=\min\{i\in[0,n-1]:x_{i+1}\not\in A\}$, $y=y(\e)=x_k$, and $z=z(\e)=x_{k+1}-x_k$. Since 
    $\freeunr{x}{y}{\infty}(\w)+\freeunr{y}{v}{\infty}(\w)\le \freeunr{x}{v}{\infty}(\w)$ for all $v\in\Lt{\sA, x,t}$, we have
    $\freeunr{x}{y}{\infty}(\w)+h\cdot(y-x)+\freelevn{y}{t}{\sA, \infty, 0}(\w)\le \freelevn{x}{t}{\sA, \infty, 0}(\w)$. Thus,
        \begin{align*}
        \freeunr{x}{y}{\infty}(\w)+h\cdot(y-x)+\freelevn{y}{t}{\sA, \infty, 0}(\w)-2\e&\le \freelevn{x}{t}{\sA, \infty, 0}(\w)-2\e\le - \sum_{i=0}^{n-1} \pote(\T_{x_i}\w, x_{i+1}-x_i)+h\cdot(v-x)\\
        &\le \freeunr{x}{y}{\infty}(\w)+h\cdot(y-x)- \sum_{i=k}^{n-1} \pote(\T_{x_i}\w,x_{i+1}-x_i)+h\cdot(v-y),
        \end{align*}
        which implies
        \[\freelevn{y}{t}{\sA, \infty, 0}(\w)-2\e\le -\sum_{i=k}^{n-1} \pote(\T_{x_i}\w,x_{i+1}-x_i)+h\cdot(v-y).\]
    Thus,
    \begin{align*}
    \pote(\T_y\w,z)+\Bn{t,0}{\sA, \infty, \m}(y, y+z) 
    &= \pote(\T_y\w,z)+\freelevn{y}{t}{\sA, \infty, 0}-\freelevn{y+z}{t}{\sA, \infty, 0} - h \cdot z\\
    &\le \pote(\T_y\w,z)+\freelevn{y}{t}{\sA, \infty, 0}-\freeunr{y+z}{v}{\infty}(\w)  -h\cdot(v-y-z)- h \cdot z\\
    &\le \freelevn{y}{t}{\sA, \infty, 0}+\sum_{i=k}^{n-1} \pote(\T_{x_i}\w, x_{i+1}-x_i)  - h \cdot(v-y)\le2\e.
    \end{align*}
Since $A$ and $\rangeA$ are finite, we can find a subsequence $\e_j\to0$ and $y\in A$, $z\in\rangeA$, such that $y(\e_j)=y\in A$ and $z(\e_j)=z\in\rangeA$ for all $j$. Applying the above with $\e=\e_j$ and taking $j\to\infty$ shows that $\pote(\T_y\w,z)+\Bn{t,0}{\sA, \infty, \m}(y, y+z)\le0$. The claim follows from this and \eqref{unrZeroPreLimRecovery}. 
\end{proof}

   Next, we develop bounds on the limiting free energy.
    Recall the definition of $\maxz_0$ in \eqref{R0}. If necessary, enlarge the slab width to satisfy $R>2\maxz_0$.
    
    \begin{lemma}\label{lm:EF-bound}
    	If $\beta = \infty$ or Condition \ref{VCondition}\eqref{condBddBelowByC} or \eqref{condProdMeasure} is satisfied, take $\e = 0$.  If $\beta < \infty$ and Condition \ref{VCondition}\eqref{condBddBelow0} holds, take $\e > 0$.    Then, for any $a<b$,
    \begin{align}\label{p2l-EF}
    1 - \frac{\abs{\xi}_1 \e}{\beta}\le \varliminf_{n \rightarrow \infty} \frac{1}{n} \bbE[\inf_{a\le s\le b}\freelevn{\orig}{n+s}{\sA, \beta, \e}] \le \varlimsup_{n \rightarrow \infty} \frac{1}{n} \bbE[\sup_{a\le s\le b}\freelevn{\orig}{n+s}{\sA, \beta, \e}]\le1
    \end{align}
	and, $\bbP$-almost surely,
    \begin{align}\label{p2l-F}
    1 - \frac{\abs{\xi}_1 \e}{\beta}\le\varliminf_{n \rightarrow \infty} \frac{1}{n}\inf_{a\le s\le b} \freelevn{\orig}{n+s}{\sA, \beta, \e}
    \le\varlimsup_{n \rightarrow \infty} \frac{1}{n} \sup_{a\le s\le b}\freelevn{\orig}{n+s}{\sA, \beta, \e}\le1.
    \end{align}
    \end{lemma}
    
    \begin{proof}
        It will be convenient in this and the following paragraph to modify our notation to include the width of the slab, writing $\freelevn{\orig}{n+a,\maxz}{\sA, \beta, \e}$ when this width is $\maxz$.
        With this notation, we see that if $\maxz'>\maxz$, then $s\le t\le t+\maxz\le s+\maxz'$ and $\freelevn{\orig}{t,\maxz}{\sA, \beta, \e}\le\freelevn{\orig}{s,\maxz'}{\sA, \beta, \e}$.
        Thus, for any $s\in[a,b]$, $n+a\le n+s\le n+s+\maxz\le n+a+(\maxz+b-a)$ and, therefore, $\sup_{a\le s\le b}\freelevn{\orig}{n+s,\maxz}{\sA, \beta, \e}\le\freelevn{\orig}{n+a,\maxz+b-a}{\sA, \beta, \e}$.
        
        Similarly, for any $s\in[a,b]$, take $j=\lfloor2\maxz^{-1}(s-a)\rfloor$, which satisfies $n+s\le n+a+j\maxz/2\le n+a+(j+1)\maxz/2\le n+s+\maxz$ and $0\le j\le 2\maxz^{-1}(b-a)+1$. Therefore, 
            \[\inf_{a\le s\le b}\freelevn{\orig}{n+s,\maxz}{\sA, \beta, \e}
            \ge\min_{j\in[0,2\maxz^{-1}(b-a)+1]\cap\ZZ}\freelevn{\orig}{n+a+j\maxz/2,\maxz/2}{\sA, \beta, \e}.\]
        The upshot is that, modulo increasing or decreasing the size of $\maxz$, it is enough to prove the claims of the lemma for a fixed $s\in\R$, without taking suprema or infima over $s\in[a,b]$. In the rest, we will go back to omitting the width of the slab from the free energy notation.

         Fix $s\in\R$. We will show that for each $s\in\R$, we have, $\bbP$-almost surely,
         \begin{align}\label{aux-p2l-F}
         1 - \frac{\abs{\xi}_1 \e}{\beta}\le\varliminf_{n \rightarrow \infty} \frac{1}{n} \freelevn{\orig}{n+s}{\sA, \beta, \e}
    \le\varlimsup_{n \rightarrow \infty} \frac{1}{n} \freelevn{\orig}{n+s}{\sA, \beta, \e}\le1.
    \end{align}
        This then proves \eqref{p2l-F}, as explained in the above two paragraphs. 
         
         We start with the lower bound. Write $\xi=\sum_{z\in\rangeA}b_z z$ with coefficients $b_z\ge0$. For $n\in\ZZ_{\ge0}$ let $b_{n,z}= \lceil (n+s) b_z \rceil$ if $z \cdot \uhat \ge 0$ and $b_{n,z} = \lfloor (n+s) b_z \rfloor$ otherwise. Let $v_n = \sum_{z \in \rangeA} b_{n,z} z$. Note that
        $v_n/n\to\xi$ as $n\to\infty$. 
        Since $\xi \cdot \uhat = 1$, 
\[ n+s= (n+s)\sum_{z \in \rangeA} b_z z \cdot \uhat \leq v_n \cdot \uhat \leq (n+s)\sum_{z \in \rangeA} b_z z \cdot \uhat + \sum_{z \in \rangeA} \abs{z\cdot\uhat} < n+s + \maxz.\]
    Thus, $v_n\in\Lt{\sA, n+s}$ for all $n$ and
\begin{align*}
\frac{1}{n} \freelevn{\orig}{n+s}{\sA, \beta, \e} 
&= \frac{1}{n\beta} \log \sum_{v \in \Lt{\sA, n+s}} \prtunr{\orig}{v}{\beta} e^{\beta h \cdot v-\e \abs{v}_1} 
\geq 
\frac{1}{n\beta} \log \prtunr{\orig}{v_n}{\beta} + \frac{\beta h \cdot v_n - \abs{v_n}_1 \e}{n\beta}.
\end{align*}
Since $\uhat - h \in \superDiffUnrA{\beta, \usc}(\xi)$ and $\xi \cdot \uhat = 1$, by Lemma \ref{lem:concaveDualHomogeneous}, $ \ppShapeUnr{\beta}(\xi) = \ppShapeUnrA{\beta, \usc}(\xi)  = (\uhat - h) \cdot \xi = 1 - h \cdot \xi$.  Taking  $n \rightarrow \infty$ and applying \eqref{shape-unr}, 
we obtain, $\bbP$-almost surely,
\[
\varliminf_{n \rightarrow \infty} \frac{1}{n} \freelevn{\orig}{n+s}{\sA, \beta, \e} \geq \ppShapeUnr{\beta}(\xi) + h \cdot \xi - \frac{\abs{\xi}_1 \e}{\beta} = 1 - \frac{\abs{\xi}_1 \e}{\beta}.
\]
The equivalent result at zero temperature is
\[
\varliminf_{n \rightarrow \infty} \frac{1}{n} \freelevn{\orig}{n+s}{\sA, \infty, 0} \geq \varliminf_{n \rightarrow \infty} \frac{1}{n} (\freeunr{\orig}{v_{n}}{\infty} + h\cdot v_{n}) = \ppShapeUnr{\infty}(\xi) + h\cdot\xi = 1.
\]

Next, we derive the almost sure upper bound. Recall that we are now working with a fixed $s$. The upper bound \eqref{shape-UB} 
implies that, with $\bbP$-probability one, there exists $n_0 = n_0(\w, \e)$ such that whenever $\abs{v}_1 \cdot \abs{\uhat}_{\infty} \geq n_0$, 
\[
\frac{1}{\beta}\log \prtunr{\orig}{v}{\beta} \leq \ppShapeUnrA{\beta, \usc}(v) + \frac{\e}{2\beta} \abs{v}_1.
\]
Also note that since $\uhat-h \in \superDiffUnrA{\beta, \usc}(\xi)$, for $v \in L_{\sA, n+s}$
\begin{align*}
&\ppShapeUnrA{\beta, \usc}(v) \leq \ppShapeUnrA{\beta, \usc}(\xi) + (v - \xi) \cdot (\uhat-h)= v \cdot (\uhat-h) \leq n+s + \maxz - v\cdot h. 
\end{align*}
For $n \geq n_0$,
\begin{align*}
\frac{1}{n} \freelevn{\orig}{n+s}{\sA, \beta, \e} &= \frac{1}{n\beta} \log \sum_{v \in L_{\sA, n+s}} \prtunr{\orig}{v}{\beta} e^{\beta h \cdot v -\e \abs{v}_1} \\
&\leq \frac{1}{n\beta} \log \sum_{v \in L_{\sA, n+s}} e^{\beta \ppShapeUnrA{\beta, \usc}(v) + \beta h \cdot v - \e \abs{v}_1 / 2} \\
&\leq \frac{1}{n\beta} \log \sum_{v \in L_{\sA, n+s}} e^{\beta (n+s + \maxz - v\cdot h) + \beta h \cdot v - \e \abs{v}_1 / 2} \\
&\le \frac{n+s+\maxz}{n}  + \frac{1}{n\beta} \log \sum_{v \in \ZZ^d} e^{- \e \abs{v}_1 / 2} \\
&\leq \frac{n+s+\maxz}{n} + \frac{1}{n\beta} \log \Bigl(C \sum_{k = 0}^{\infty}k^d e^{- \e k / 2} \Bigr)
\xrightarrow[n \rightarrow \infty]{} 1.
\end{align*}
The equivalent bound at zero temperature is
\begin{align*}
\frac{1}{n} \freelevn{\orig}{n+s}{\sA, \infty, 0} = \frac{1}{n} \sup_{v \in \Lt{\sA, n+s}} (\freeunr{\orig}{v}{\infty}+h\cdot v) \leq \frac{1}{n} \sup_{v \in \Lt{\sA, n+s}} (\ppShapeUnrA{\infty, \usc}(v) + 1 + h\cdot v) \leq \frac{1}{n} \bigl(n+s +\maxz +1 \bigr) \xrightarrow[n \rightarrow \infty]{} 1.
\end{align*}

Now we turn to proving \eqref{p2l-EF}. As explained at the beginning of the proof, it is enough to prove that for each $s\in\R$,
    \begin{align}\label{aux-p2l-EF}
    1 - \frac{\abs{\xi}_1 \e}{\beta}\le \varliminf_{n \rightarrow \infty} \frac{1}{n} \bbE[\freelevn{\orig}{n+s}{\sA, \beta, \e}] \le \varlimsup_{n \rightarrow \infty} \frac{1}{n} \bbE[\freelevn{\orig}{n+s}{\sA, \beta, \e}]\le1.
    \end{align}
We first prove the upper bound. For this, we  consider the three cases in Condition \ref{VCondition}. 

If Condition \ref{VCondition}\eqref{condBddBelow0} or \eqref{condBddBelowByC} is satisfied, then the upper bound in \eqref{aux-p2l-EF} follows from the bounds 
(\ref{eqn:L1UpperBoundCase1}-\ref{eqn:L1upperBoundBddBelowCZero}), 
the upper bound in \eqref{aux-p2l-F}, and Fatou's Lemma. 


Suppose now that Condition \ref{VCondition}\eqref{condProdMeasure} holds. 
Without any loss of generality, we can assume $r_0 > \maxz$. 
Let $f(\w) = \max_{z \in \rangeA}\pote^-(\w, z) + \max_{z \in \rangeA} \abs{h\cdot z}$. Then, with the notation $\ell = \ell(n+s)$ as introduced above \eqref{eqn:L1UpperBoundCase1Zero}, the
bounds in \eqref{eqn:eqn:L1upperBoundIID} and \eqref{eqn:eqn:L1upperBoundIID2} give 
\begin{align}
\frac1n\freelevn{\orig}{n+s}{\sA, \beta, 0} 
&\leq \frac1n\max_{(n+s) \maxz^{-1} \leq k \leq (n+s + \maxz)\delta^{-1} } \max_{v \in \tilde{L}_{\sA,\orig, n+s}} (\freeres{\orig}{v}{k}{\infty} +h\cdot v) + \frac{1}{n\beta}\log C\ell^{d}\notag\\
&\le\frac{1}{n} \max_{x_{0:\ell} \in \PathsNStep{\orig}{\ell}(\rangeA)}  \sum_{k=0}^{\ell-1} f(\T_{x_k}\w)
+ \frac{1}{n\beta}\log C\ell^{d},\label{aux-F-bnd}
\end{align}
where it is understood that in the zero temperature case, the last term on the right-hand side is $0$.
In the second inequality, we used the fact that $f$ is non-negative, and hence the sum of shifts of $f$ does not get smaller if we add more terms, and the maximum does not get smaller if we drop the requirement that the endpoint $v \in \Lt{\sA, n+s}$. 
We will bound the first term on the right-hand side using lattice animal bounds.

A subset $S$ of $\ZZ^d$ is said to be connected if, for any two points in $S$, there exists a nearest-neighbor path in $S$ that connects them.
Given $n\in\Z_{>0}$, a connected subset of $\ZZ^d$ that has cardinality $n$ is called a lattice animal of size $n$. Let $\sS_n$ denote the set of lattice animals of size $n$. Let $B = \{0, 1, 2, \ldots, r_0-1\}^d$. Then $\{r_0 y + B : y \in \ZZ^d\}$ is a disjoint tiling of $\ZZ^d$.

Consider a path $x_{0:\ell} \in \PathsNStep{\orig}{\ell}(\rangeA)$. Let $y_0, \ldots, y_\ell \in \ZZ^d$ and $u_0,\ldots,u_\ell\in B$ be 
(the unique points) such that $x_k=r_0y_k+u_k$ for each $0 \leq k \leq \ell$. 
It must be that $\abs{y_{k+1} - y_k}_{\infty} \leq 1$, since $r_0>\maxz$.  Although $y_k$ and $y_{k+1}$ might not be nearest neighbors, we can fill in with at most $d-1$ intermediate points to obtain a nearest-neighbor sequence. In total, we start with $\ell+1$ points then add at most $\ell(d-1)$ more. The resulting sequence may have repetitions and fewer than $\ell d+1$ points, but it will be contained in a lattice animal $S \in \sS_{\ell d+1}$.
Since the path $x_{0:\ell}$ does not repeat points, the pairs $(y_k,u_k)$ are distinct. Therefore, 
    \begin{align*}
        \sum_{k=0}^{\ell-1}f(\T_{x_k}\w)
        =\sum_{k=0}^{\ell-1}f(\T_{r_0 y_k+u_k}\w)
        \le\sum_{u\in B}\max_{S \in \sS_{\ell d+1}} \sum_{y \in S} f(\T_{r_0 y + u}\w).
    \end{align*}
We have thus shown that
\begin{align*}
\frac{1}{n} \max_{x_{0:\ell} \in \PathsNStep{\orig}{\ell}(\rangeA)}  \sum_{k=0}^{\ell-1} f(\T_{x_k}\w)
&\leq \frac{\ell}{n} \cdot \frac{1}{\ell} \sum_{u \in B} \max_{S \in \sS_{\ell d+1}} \sum_{y \in S} f(\T_{r_0 y + u}\w).
\end{align*}

Since $\{f(\T_{r_0 y + u}\w) : y \in \ZZ^d\}$ are i.i.d.\ and $f \in L^q(\P)$ for some $q > d$, \cite[Theorem 1.1]{Mar-02-spa} implies that the right-hand side converges both almost-surely and in $L^1(\P)$. This, the upper bound \eqref{aux-F-bnd}, the upper bound in \eqref{aux-p2l-F}, and Fatou's Lemma imply the upper bound in \eqref{aux-p2l-EF} and, therefore, the upper bound in \eqref{p2l-EF} is proved.

We finish the proof by arguing for the lower bound in \eqref{aux-p2l-EF}.
We work out the case of a positive temperature, with the case of zero temperature being similar. Observe that for any $n\in\Z_{>0}$ and $u,v\in\ZZ^d$ such that $v-u\in\RgroupA\cap\Dn{n}$, 
	\begin{align*}
		\prtres{u}{v}{n}{\beta}(\w) 
		&= \rwE_u\Bigl[ e^{- \beta\sum_{k=0}^{n-1} \pote(\T_{X_k} \omega, X_{k+1}-X_k)}\one_{\{X_n=v\}}\Bigr]
		\le \rwE_u\Bigl[ e^{- \beta\sum_{k=0}^{\stoppt{v}-1} \pote(\T_{X_k} \omega, X_{k+1}-X_k)}\one_{\{\stoppt{v}=n\}}\Bigr]\\
		&\le\rwE_u\Bigl[ e^{- \beta\sum_{k=0}^{\stoppt{v}-1} \pote(\T_{X_k} \omega, X_{k+1}-X_k)}\one_{\{\stoppt{v}<\infty\}}\Bigr]
		=\prtunr{u}{v}{\beta}(\w). 	
	\end{align*}
To justify the first inequality note that if $\orig\not\in\Uset$, then $X_n=v$ implies $\stoppt{v}=n$ while if $\orig\in\Uset$, then Condition \ref{VCondition}\eqref{condBddBelow0} must hold, in which case $X_n=v$ implies $\pote(T_{X_k}\w,X_{k+1}-X_k)\ge0$ for all $k\le n$ and, therefore,
$e^{- \beta\sum_{k=\stoppt{v}}^{n-1} \pote(\T_{X_k} \omega, X_{k+1}-X_k)}\le1$. 

Take any $t>0$ such that $\xi/t\in\UsetA$. Let $\hat x_n(\xi/t)$ be the path defined in \cite[(2.1)]{Ras-Sep-14}. Note that this path has increments in $\rangeA$ and satisfies  $\hat x_n(\xi/t)\in\RgroupA\cap\Dn{n}$ for all $n\in\Z_{>0}$ and $\hat x_n(\xi/t)/n\to\xi/t$ as $n\to\infty$. 
This and $\xi\cdot\uhat=1$ imply that there exists a subsequence $m_n$ such that  $\hat x_{m_n}(\xi/t)\in\Lt{\sA, n+s}$ for all $n$ large enough. 
Thus, for $n$ large enough,
\begin{align*}
\frac{1}{n} \freelevn{\orig}{n+s}{\sA, \beta, \e} 
&= \frac{1}{n\beta} \log \sum_{v \in \Lt{\sA, n+s}} \prtunr{\orig}{v}{\beta} e^{\beta h \cdot v-\e \abs{v}_1} 
\geq 
\frac{1}{n\beta} \log \prtunr{\orig}{\hat x_{m_n}(\xi/t)}{\beta} + \frac{\beta h \cdot \hat x_{m_n}(\xi/t) - \abs{\hat x_{m_n}(\xi/t)}_1 \e}{n\beta}\\
&\geq 
\frac{1}{n\beta} \log \prtres{\orig}{\hat x_{m_n}(\xi/t)}{m_n}{\beta}  + \frac{\beta h \cdot \hat x_{m_n}(\xi/t) - \abs{\hat x_{m_n}(\xi/t)}_1 \e}{n\beta}.
\end{align*}

By \cite[Theorem 2.2]{Ras-Sep-14} (and \cite[Theorem 2.4]{Geo-Ras-Sep-16} for the zero temperature case), $m^{-1}\freeres{\orig}{\hat x_m(\xi/t)}{m}{\beta}$ converges both almost surely and in $L^1(\P)$, as $m\to\infty$. Furthermore, dividing    
    \[n+s\le\hat x_{m_n}(\xi/t)\cdot\uhat<n+s+R\]
by $m_n$ and taking $n\to\infty$ implies that $m_n/n\to t$ as $n\to\infty$.
Thus, $(n\beta)^{-1}\freeres{\orig}{\hat x_{m_n}(\xi/t)}{m_n}{\beta}$ converges both almost surely and in $L^1(\P)$, as $n\to\infty$. This allows us to apply Fatou's Lemma and deduce the lower bound in \eqref{aux-p2l-EF} from the one in \eqref{aux-p2l-F}, completing the proof of the lemma.
\end{proof}
    
	For each $n \in \ZZ_{>0}$, let $(U_n)_{n\in\Z_{>0}}$ be a sequence of independent random variables, independent of everything else, and such that for each $n$,  $U_n$ is uniformly distributed on $[0,n]$. Denote the distribution of this sequence by $\bbU$ and let $\bbPU = \bbP \otimes \bbU$ with expectation $\bbEU$. 
	Define $\bfPne$ to be the distribution of 
	\[
	\Bigl(  \omega,  \bigl\{ \BUne{\sA, \beta, \m}(x, x+z) : \sA \in \faces, (\beta,m) \in \DctbA, x \in \ZZ^d, z \in \rangeA \bigr\}\Bigr)
	\]
	induced by $\bbPU$ on $(\Omegahat, \Sighat)$.  
	
	\begin{lemma}
	If $\beta = \infty$ or Condition \ref{VCondition}\eqref{condBddBelowByC} or \eqref{condProdMeasure} is satisfied, take $\e = 0$.  If $\beta < \infty$ and Condition \ref{VCondition}\eqref{condBddBelow0} holds, take $\e > 0$.   Then, the family $\{\bfPne : n \in \ZZ_{>0} \}$ is tight. 
	\end{lemma}
	
	\begin{proof}
	 Take $x \in \ZZ^d$ and $z \in \rangeA$. Let $t_0= \max\{0,x \cdot \uhat,(x+z) \cdot \uhat : z \in \rangeA \}$.  Recall that $\Bn{t,\e}{\sA, \beta, \m}(x,x+z) = 0$ whenever $0<t \leq t_0$. Then for $n > t_0$, use the approximate shift-covariance properties \eqref{eqn:approxShiftCov} and \eqref{eqn:approxShiftCovZeroTemp} to obtain
	\begin{align}
		&\bbEU \bigl[ \BUne{\sA, \beta, \m}(x, x+z) \bigr] \nonumber \\
        &= \frac{1}{n} \int_{t_0}^{n}  \bbE \bigl[ \Bn{s,\e}{\sA, \beta, \m}(x, x+z)  \bigr] \,ds \nonumber \\
		&= \frac{1}{n}  \int_{t_0}^{n} \bbE\bigl[ \freelevn{x}{s}{\sA, \beta, \e}-\freelevn{x+z}{s}{\sA, \beta, \e} - h\cdot z \bigr]\,ds\nonumber \\
		&\leq \frac{1}{n}  \int_{t_0}^{n} \bbE\bigl[ \freelevn{\orig}{s - x \cdot \uhat}{\sA, \beta, \e} - \freelevn{\orig}{s - (x+z) \cdot \uhat }{\sA, \beta, \e} \bigr] \,ds - \frac{(n-t_0) h \cdot z}{n} + \frac{\e}{\beta} (\abs{x}_1 + \abs{x+z}_1) \nonumber \\
		&= \frac{1}{n}  \int_{t_0}^{n} \bbE\bigl[ \freelevn{\orig}{s - x \cdot \uhat}{\sA, \beta, \e}\bigr]\,ds - \frac{1}{n}  \int_{t_0-z\cdot\uhat}^{n-z\cdot\uhat}\bbE\bigl[ \freelevn{\orig}{s - x \cdot \uhat}{\sA, \beta, \e} \bigr]  \,ds - \frac{(n-t_0) h \cdot z}{n} + \frac{\e}{\beta} (\abs{x}_1 + \abs{x+z}_1) \nonumber \\
		&= \frac{1}{n}  \int_{n-z\cdot\uhat}^{n} \bbE\bigl[ \freelevn{\orig}{s - x \cdot \uhat}{\sA, \beta, \e}\bigr]\,ds - \frac{1}{n}  \int_{t_0-z\cdot\uhat}^{t_0}\bbE\bigl[ \freelevn{\orig}{s - x \cdot \uhat}{\sA, \beta, \e} \bigr]  \,ds - \frac{(n-t_0) h \cdot z}{n} + \frac{\e}{\beta} (\abs{x}_1 + \abs{x+z}_1) \nonumber \\
        &= \frac{1}{n}  \int_{-(x+z)\cdot\uhat}^{-x\cdot\uhat} \bbE\bigl[ \freelevn{\orig}{n+s}{\sA, \beta, \e}\bigr]\,ds - \frac{1}{n}  \int_{t_0-(x+z)\cdot\uhat}^{t_0-x\cdot\uhat}\bbE\bigl[ \freelevn{\orig}{s}{\sA, \beta, \e} \bigr] \,ds - \frac{(n-t_0) h \cdot z}{n} + \frac{\e}{\beta} (\abs{x}_1 + \abs{x+z}_1).
    \label{unrBnMean}
 	\end{align}
    In the above computation, we use the convention that if $b<a$, then $\int_a^b=-\int_b^a$. When $\beta = \infty$, we take $\e = 0$, and the last term is understood to be zero.
	
	On the event $\{U_n > t_0\}$,  
	\eqref{unrPosPreLimRecovery} and \eqref{unrZeroPreLimRecovery}  imply that 
		\begin{align}\label{fofo}
		\BUne{\sA, \beta, \m}(x, x+z) \geq -\pote(\T_x \omega, z) + \beta^{-1} \log p(z)\ge-\pote^+(\T_x\omega,z)+\beta^{-1}\log p(z).
		\end{align}
	Again when $\beta = \infty$, the last term is understood to be 0.  Then,
		\begin{align}
		&\bbEU \bigl[ \bigl|\BUne{\sA, \beta, \m}(x, x+z) \bigr| \bigr] 
		= \bbEU \bigl[ \bigl|\BUne{\sA, \beta, \m}(x, x+z) \bigr|\one\{U_n>t_0\} \bigr] \nonumber\\
		&\quad= \bbEU \bigl[ \BUne{\sA, \beta, \m}(x, x+z)\one\{U_n>t_0\} \bigr]-2\bbEU \bigl[\min\bigl\{0,\BUne{\sA, \beta, \m}(x, x+z)\bigr\}\one\{U_n>t_0\} \bigr] \nonumber\\
		&\quad=\bbEU \bigl[ \BUne{\sA, \beta, \m}(x, x+z)\bigr]
		  -2\bbEU \bigl[\min\bigl\{0,\BUne{\sA, \beta, \m}(x, x+z)\bigr\}\one\{U_n>t_0\} \bigr] \nonumber\\
		\begin{split}
		&\quad\le \frac{1}{n}  \int_{-(x+z)\cdot\uhat}^{-x\cdot\uhat} \bbE\bigl[ \freelevn{\orig}{n+s}{\sA, \beta, \e}\bigr]\,ds - \frac{1}{n}  \int_{t_0-(x+z)\cdot\uhat}^{t_0-x\cdot\uhat}\bbE\bigl[ \freelevn{\orig}{s}{\sA, \beta, \e} \bigr] \,ds\\
		&\qquad\qquad- \frac{(n-t_0) h \cdot z}{n} + \frac{\e}{\beta}(\abs{x}_1 + \abs{x+z}_1) + 2 \bbE[\pote^+(\omega, z)] - 2 \beta^{-1} \log p(z).
		\end{split}
		\label{eqn:BUneUpperBound}
		\end{align}

All the terms on the right-hand side except the first are trivially uniformly bounded in $n$. The first term on the right-hand side is also bounded in $n$ by \eqref{aux-p2l-EF}. This proves the claimed tightness.
	\end{proof}

    If $\beta = \infty$ or Condition \ref{VCondition}\eqref{condBddBelowByC} or \eqref{condProdMeasure} is satisfied, then let $\bfPs{0}$ (with expectation $\bfEs{0}$) be any weak subsequential limit point of $\bfPs{n,0}$.  Otherwise, for each $0 < \e \leq 1$, let $\bfPe$ (with expectation $\bfEe$) denote any weak subsequential limit point in $n$ of $\bfPne$.   Recall that $\B{\sA, \beta, \m}(x, x+z)$ is the $(\sA, \beta, \m, x, z)$-coordinate of $\omegahat \in \Omegahat$. By \eqref{fofo} we have that $\bfPne$-almost surely, for all $\sA \in \faces$, $(\beta, \m) \in \DctbA$, $x \in \ZZ^d$, and $z \in \rangeA$,  $\B{\sA, \beta, \m}(x, x+z) \geq -\abs{\pote(\T_x \omega, z)} + \beta^{-1} \log p(z)$.

    Use \eqref{eqn:BUneUpperBound}, Fatou's Lemma, and the upper bound in \eqref{p2l-EF} to obtain 
    \begin{align}
    \bfEe\bigl[\bigl|\B{\sA, \beta, \m}(x, x+z)\bigr| \bigr] 
    &\leq \varliminf_{n \rightarrow \infty} \frac{1}{n}  \Bigl(\int_{-(x+z)\cdot\uhat}^{-x\cdot\uhat } \bbE\bigl[ \freelevn{\orig}{n+s}{\sA, \beta,  \e}\bigr]\,ds - \frac{1}{n}  \int_{t_0-(x+z)\cdot\uhat }^{t_0-x\cdot\uhat }\bbE\bigl[ \freelevn{\orig}{s}{\sA, \beta, \e} \bigr] \,ds - \frac{(n-t_0) h \cdot z}{n} \Bigr) \notag\\
    &\qquad+ \frac{\e}{\beta}(\abs{x}_1 + \abs{x+z}_1) + 2 \bbE[|\pote(\omega, z)|] - 2 \beta^{-1} \log p(z)\notag\\
    &\leq \m \cdot z + \frac{1}{\beta}(\abs{x}_1 + \abs{x+z}_1) + 2 \bbE[|\pote(\omega, z)|] - 2 \beta^{-1} \log p(z).\label{E|B|}
    \end{align}
    The family $\{\bfPe : 0 < \e \leq 1\}$ is therefore tight. 
    Similarly, using Fatou's Lemma with \eqref{unrBnMean} and \eqref{p2l-EF} gives
    \begin{align}
    \bfEe\bigl[\B{\sA, \beta, \m}(x, x+z)\bigr]
    &\leq \uhat \cdot z - h \cdot z + \frac{\e}{\beta}(\abs{x}_1 + \abs{x+z}_1) = \m \cdot z + \frac{\e}{\beta}(\abs{x}_1 + \abs{x+z}_1). \label{eqn:FatouForEpsilon}
    \end{align}
    
Let $\bbPhat$ denote any weak subsequential limit point of $\bfPe$ as $\e\to0$ in $(0,1]$.  If $\beta = \infty$ or Condition \ref{VCondition}\eqref{condBddBelowByC} or \eqref{condProdMeasure} is satisfied, then take $\bbPhat = \bfPs{0}$. 
 
	\begin{lemma}
	$\bbPhat$ is $\That$-invariant.
	\end{lemma}
	
	\begin{proof}
    Let $x,y \in \ZZ^d$, and $z \in \rangeA$.  
    Using \eqref{eqn:approxShiftCov} we get
    \begin{align}
    \Bigl| \Bn{t, \e}{\sA, \beta, \m}(x+y,x+y+z,\T_{-y}\w) - \Bn{t-y\cdot\uhat, \e}{\sA, \beta, \m}(x,x+z,\w) \Bigr| \leq \frac{2 \e \abs{y}_1}{\beta}. \label{eqn:preLimitCov}
    \end{align}
%
    
    Fix $y\in\ZZ^d$. Let $\alpha \geq 0$, $\ell \in \ZZ_{>0}$, and take $f$ to be a bounded, $\alpha$-Lipschitz continuous function of any $\ell$ coordinates of $\omegahat$. These coordinates can be faces $\sA$, vectors $\m$, inverse temperatures $\beta$, vertices $x$, or increments $z$. Let $t_0$ be large enough so that $t_0 > \max\{y\cdot\uhat,x \cdot \uhat,(x+z) \cdot \uhat,(x+y)\cdot\uhat,(x+y+z)\cdot\uhat :z\in\rangeA\}$, 
    where the  maximum goes over all of the at most $\ell$ values of the $x \in \ZZ^d$ coordinates on which $f$ depends.  Let $(n_{\e, k})_{k \in \ZZ_{>0}}$ be the subsequence along which $\bfPs{n,\e}$ converges weakly to $\bfPe$.  Let $\e_j \rightarrow 0$ be the sequence along which $\bfPs{\e_j}$ converges weakly to $\bbPhat$.  Then using continuity of $\That_{y}$, we have 
    \begin{align*}
        \bbEhat[f \circ \That_{y}] &= \lim_{j \rightarrow \infty} \bfEs{\e_j}[f \circ \That_{y}] = \lim_{j \rightarrow \infty} \lim_{k \rightarrow \infty} \bfEs{n_k, \e_j}[f \circ \That_{y}]\\
        &= \lim_{j \rightarrow \infty} \lim_{k \rightarrow \infty} \bfEs{n_k, \e_j}\bigl[f\bigl(\T_y\w,\{\B{\sA, \beta, \m}(x+y,x+y+z)\}\bigr)\bigr]\\
        &= \lim_{j \rightarrow \infty} \lim_{k \rightarrow \infty} \Bigl(\sO(n_k^{-1}) + \frac{1}{n_k} \int_{t_0}^{n_k} \bbE\bigl[f\bigl(\T_y \w, \{\Bn{s, \e_j}{\sA, \beta, \m}(x+y,x+y+z, \w)\}\bigr)\bigr] \,ds \Bigr) \\
        &= \lim_{j \rightarrow \infty} \lim_{k \rightarrow \infty} \frac{1}{n_k} \int_{t_0}^{n_k} \bbE\bigl[f\bigl(\w, \{\Bn{s, \e_j}{\sA, \beta, \m}(x+y,x+y+z, \T_{-y}\w)\}\bigr)\bigr] \,ds \\
        &= \lim_{j \rightarrow \infty} \lim_{k \rightarrow \infty} \frac{1}{n_k} \int_{t_0}^{n_k} \bigl\{\bbE\bigl[f\bigl(\w, \{\Bn{s-y\cdot\uhat, \e_j}{\sA, \beta, \m}(x,x+z,\w)\}\bigr)\bigr] + \alpha\ell\sO(\beta^{-1}\e_j\abs{y}_1)\bigr\} \,ds \\
        &= \lim_{j \rightarrow \infty} \lim_{k \rightarrow \infty} \frac{1}{n_k} \int_{t_0 -y\cdot\uhat}^{n_k-y\cdot\uhat} \bbE\bigl[f\bigl(\w, \{\Bn{s, \e_j}{\sA, \beta, \m}(x,x+z,\w)\}\bigr)\bigr] \,ds \\
        &= \lim_{j \rightarrow \infty} \lim_{k \rightarrow \infty} \frac{1}{n_k} \int_{t_0}^{n_k} \bbE[f(\w, \{\Bn{s, \e_j}{\sA, \beta, \m}(x,x+z,\w)\})] \,ds = \lim_{j \rightarrow \infty} \lim_{k\rightarrow \infty} \bfEs{n_k,\e_j}[f] = \bbEhat[f].
    \end{align*}
    On the fourth line, we used the $\T$-invariance of $\bbP$. On the fifth line, we used \eqref{eqn:preLimitCov} and that $f$ is an $\alpha$-Lipschitz continuous function of $\ell$ coordinates of $\omegahat$.  
    The first equality on the last line follows because the integrals differ by at most $2\abs{y\cdot \uhat}\sup\abs{f}$.
	\end{proof}

    We are now ready to prove the main theorem of this section.
    
    \begin{proof}[Proof of Theorem \ref{thm:Cocycles}]
    	For each $n \in \ZZ_{>0}$, $\e \geq 0$, and $A \in \Sig$, $\bfPne(\wProj \in A) = \bbP(\w \in A)$. Since $\wProj$ is continuous, taking $n\to\infty$ and $\e\to 0$ shows that $\bbPhat$ too has this property, for all closed sets $A$ and hence also for all $A\in\Sig$. Part \eqref{Cocycles.a} is proved.  
    
	Take $x, y \in \ZZ^d$ with $y-x \in \RsemiA$.  Consider a path $x_{0:k} \in \PathsPtPUnr{x}{y}$ of some length $k \in \ZZ_{\ge0}$. For any $t$ large enough that $t> x_i \cdot \uhat$ for all $0 \leq i \leq k$,
	\begin{align*}
		\sum_{i=0}^{k-1} \Bn{t,\e}{\sA, \beta, \m}(x_i, x_{i+1}) &= \sum_{i=0}^{k-1} \left(  \freelevn{x_i}{t}{\sA, \beta, \e} -\freelevn{x_{i+1}}{t}{\sA, \beta, \e} - h\cdot(x_{i+1}-x_i) \right) =  \freelevn{x}{t}{\sA, \beta, \e} - \freelevn{y}{t}{\sA, \beta, \e} - h\cdot (y-x). 
	\end{align*}
	Thus, for any pair of paths $x_{0:k},x'_{0:k'}\in\PathsPtPUnr{x}{y}$, $\bbPU$-almost surely, on the event  
	\[\bigl\{U_n>c''=\max\{x_i \cdot \uhat, x'_j \cdot \uhat :0\le i\le k,0\le j\le k'\}\bigr\},\]
	we have 
	\begin{align*}
		\sum_{i=0}^{k-1} \BUne{\sA, \beta, \m}(x_i, x_{i+1}) = \sum_{i=0}^{k'-1} \BUne{\sA, \beta, \m}(x'_i, x'_{i+1}) . 
	\end{align*}

	From this, we get
	\[\bfPne\Bigl\{\sum_{i=0}^{k-1} \B{\sA, \beta, \m}(x_i, x_{i+1}) = \sum_{i=0}^{k'-1} \B{\sA, \beta, \m}(x'_i, x'_{i+1})\Bigr\}\ge
	\bbPU\{U_n> c''\}\mathop{\longrightarrow}_{n\to\infty}1.\]
	Since the event on the left-hand side is closed and $\bbPhat$ is a weak limit point of $\bfPne$, 
	we conclude that the event has $\bbPhat$-probability one. Similarly, for any $n\in\ZZ_{>0}$ and $\e \geq 0$,
	$\BUne{\sA, \beta, \m}(x,x)=0$, $\bbPU$-almost surely and hence $\B{\sA, \beta, \m}(x,x)=0$, $\bbPhat$-almost surely, for any $x\in\ZZ^d$.
	Consequently, there is a $\That$-invariant event $\Omegahat_0$ with $\bbPhat(\Omegahat_0) = 1$ such that for all $\omegahat\in\Omegahat_0$, $\sA \in \faces$, $(\beta, \m) \in \DctbA$,  and $x,y\in\ZZ^d$ with $y-x\in\RsemiA$, $\B{\sA, \beta, \m}(x,x,\omegahat)=0$ and
	\begin{align}\label{Bdef}
	\B{\sA, \beta, \m}(x, y, \omegahat) = \sum_{i=0}^{k-1} \B{\sA, \beta, \m}(x_i, x_{i+1}, \omegahat)
	\end{align}
	has the same value for all paths $x_{0:k} \in \PathsPtPUnr{x}{y}$.  
	
	For each $z \in \rangeA$ and $x \in \ZZ^d$, set $\B{\sA, \beta, \m}(x, x-z, \omegahat) = -\B{\sA, \beta, \m}(x-z, x, \omegahat)$.  This definition is consistent with the one above when it also happens to be that $-z\in\RsemiA$. 
	With this definition we have that the right-hand side of \eqref{Bdef} has the same value for all paths $x_{0:k}$ with $k\in\ZZ_{\ge0}$, $x_0=x$, $x_k=y$, and $x_{i+1}-x_i\in\rangeA\cup(-\rangeA)$ for $0\le i\le k-1$. Therefore, one can use the right-hand side of \eqref{Bdef} to define $\B{\sA, \beta, \m}(x,y,\omegahat)$ for all $x,y\in\ZZ^d$ with $y-x \in \RgroupA$ and $\omegahat\in\Omegahat_0$. With this definition, we have the cocycle property \eqref{Bcoc}, for all $\omegahat\in\Omegahat_0$, $(\beta, \m) \in \DctbA$, and $u,x,y\in\ZZ^d$ such that $x-u, y-x \in \RgroupA$.

The covariance property \eqref{Bcov} holds, for all $\omegahat\in\Omegahat$ and $v,x,y\in\ZZ^d$ with $y-x\in\rangeA$ simply by the definition of the shift $\That_z$. Then the definition \eqref{Bdef} ensures that this holds for all $\omegahat\in\Omegahat_0$ and $v,x,y\in\ZZ^d$ such that $y-x \in \RgroupA$, without the restriction $y-x\in\rangeA$.

Using a similar argument as above one deduces from \eqref{unrPosPreLimRecovery} and \eqref{unrZeroPreLimRecovery} and passing $n\to\infty$ and $\e\to 0$ that there exists a $\That$-invariant event $\OmegaCoc\subset\Omegahat_0$ with $\bbPhat(\OmegaCoc)=1$ such that \eqref{Brecbeta} and \eqref{Brecinf} hold for all $\omegahat\in\OmegaCoc$, 
$x\in\ZZ^d$, and $(\beta, \m) \in \DctbA$.
	Part \eqref{Cocycles.c} is proved.
	
	Next, we prove part \eqref{Cocycles.b}. 
	By \eqref{fofo} we have that $\bbPU$-almost surely, for all $\sA \in \faces$, $(\beta, \m) \in \DctbA$, $x \in \ZZ^d$, and $z \in \rangeA$,  $\Bn{\e}{\sA, \beta, \m}(x, x+z) \geq -\abs{\pote(\T_x \omega, z)} + \beta^{-1} \log p(z)$.

 
	Recall the vector $\mVec(\B{\sA, \beta, \m})$ from \eqref{meanVec}. Using (\ref{Brecbeta}-\ref{Brecinf}) and applying Fatou's Lemma to  (\ref{E|B|}-\ref{eqn:FatouForEpsilon}) we get that $\B{\sA, \beta, \m}(x, x+z)$ is integrable under $\bbPhat$  and that
	\begin{align}
		-\E[\abs{\pote(\T_x \omega, z)}] + \beta^{-1} \log p(z)\le\bbEhat[\B{\sA, \beta, \m}(x, x+z)] \leq \m\cdot z. \label{FatouUnr}
	\end{align}
	%
	This implies that $\bbEhat[\mVec(\B{\sA, \beta, \m})\cdot \zeta]  \leq \m\cdot \zeta$ for all $\zeta \in \sA$.  In particular, $\bbEhat[\mVec(\B{\sA, \beta, \m})\cdot \xi]  \leq \m \cdot \xi = \ppShapeUnr{\beta}(\xi)$.

    By the variational formula \cite[(2.8)]{Jan-Nur-Ras-22}, $-\ppShapeUnr{\beta}(\xi) \geq -\mVec(\B{\sA, \beta, \m}) \cdot \xi$, $\bbPhat$-almost surely.  Hence $\mVec(\B{\sA, \beta, \m}) \cdot \xi = \ppShapeUnr{\beta}(\xi)$, $\bbPhat$-almost surely. Since $\xi \in \ri \sA$, we can scale $\xi$ by a positive constant to obtain a direction in $\ri \UsetA$, apply \cite[Theorem 6.9]{Roc-70}, and rescale back to $\xi$ to obtain constants $a_z > 0$ for $z\in\rangeA$, such that $\xi = \sum_{z \in \rangeA} a_z z$. Then $\bbPhat$-almost surely,
    \[
    \sum_{z \in \rangeA} a_z z \cdot \mVec(\B{\sA, \beta, \m}) = \mVec(\B{\sA, \beta, \m})\cdot\xi = \ppShapeUnr{\beta}(\xi) = \m\cdot \xi = \sum_{z \in \rangeA} a_z z \cdot \m.
    \]
    Taking the expectation and rearranging, we obtain
    \[
    \sum_{z \in \rangeA} a_z (\m \cdot z -  \bbEhat[\mVec(\B{\sA, \beta, \m})] \cdot z ) = 0.
    \]
	
	By \eqref{FatouUnr}, $\bbEhat[\mVec(\B{\sA, \beta, \m})] \cdot z \leq \m \cdot z$ for each $z \in \rangeA$, so each term in the sum is non-negative.  Therefore each term is exactly 0.  Since $a_z > 0$, it must be that $\bbEhat[\mVec(\B{\sA, \beta, \m})] \cdot z = \m \cdot z$ for each $z \in \rangeA$.  Since $\m$ and $\bbEhat[\mVec(\B{\sA, \beta,\m})]$ are vectors of the linear subspace generated by the steps $z \in \rangeA$, this implies that $m = \bbEhat[\mVec(\B{\sA, \beta,\m})]$, and part \eqref{Cocycles.b} follows. 
	
	It remains to prove part \eqref{Cocycles.d} of the theorem. Recall the definitions of $I^{\ge}_{\sA}$ and $I^{<}_{\sA}$ in \eqref{def:Igeq} and \eqref{def:Ileq}. Observe that the distribution of 
	   \[\Bigl(\{\pote(T_v\w,z):z\in\rangeA, v \in I_{\sA}^{<}\},\bigl\{ \B{\sA, \beta, \m}(x,x+z, \omegahat) : x \in I, z \in \rangeA, (\beta, m) \in \DctbA \bigr\}\Bigr)\]
	   under $\bbPhat$ comes from taking subsequential limits of the distribution of 
	   \begin{align}\Bigl(\{\pote(T_v\w,z):z\in\rangeA, v \in I_{\sA}^{<}\},\bigl\{ \BUne{\sA, \beta, \m}(x,x+z, \omega) : x \in I, z \in \rangeA, (\beta, m) \in \DctbA \bigr\}\Bigr)\label{eq:BusIndep}\end{align}
	   under $\bbPU$, first taking $n\to\infty$ then $\e\to0$. Note that the hypotheses of part \eqref{Cocycles.d} of the theorem imply that the two collections in \eqref{eq:BusIndep} are independent.
       The independence claim comes then from the fact that if $\{(\pote(T_v\w,z))_{z\in\rangeA}:v\in\ZZ^d\}$ are independent under $\bbP$, then the two families in the above display are independent under $\bbPU$.
	The proof of Theorem \ref{thm:Cocycles} is complete.
	\end{proof}
	
		     The same argument that gives \eqref{Brecinf} from \eqref{unrZeroPreLimRecovery} proves the following lemma using Lemma \ref{lm:rec-gen-t}.

		\begin{lemma}\label{lm:rec-gen}
		 Suppose the assumptions of Theorem \ref{thm:Cocycles} hold.
		 There exists a $\That$-invariant event $\Omegahat_0$ with $\bbPhat(\Omegahat_0) = 1$ such that for all $\omegahat\in\Omegahat_0$, each face $\sA\in \faces$, $\m$ such that $(\infty,\m)\in \DctbA$, and any finite set $A\subset\ZZ^d$,
there exist a $y\in A$ and a $z\in\rangeA$ such that $y+z\not\in A$ and
    $\pote(\T_y\w,z)+\B{\sA, \infty, \m}(y,y+z)=0$.
		 \end{lemma}
		 
	We close this section with a result showing a version of weak continuity of covariant cocycles. For $\sA\in\faces$ let 
	$\wt\Omega_\sA=\Omega \times \R^{\ZZ^d\times\rangeA}$,  equipped with the product topology and its Borel $\sigma$-algebra, ${\wt\Sig}_\sA$. 
	For $\wt\w \in \wt\Omega_\sA$, 
	let $\w(\wt\w)$ be the projection to its $\Omega$ coordinate $\w$ and 
	let $\B{\sA}(x, x+z, \wt\w)$ be the $(x, z)$-coordinate of $\wt\w$.  
	
     \begin{lemma}\label{lem:partWeakCont}
    Fix a face $\sA\in \faces$. Suppose that $\pote^-(\w,z)\in L^1(\P)$ for each $z \in \rangeA$.
    Let $(\beta, \m)$ be such that $\beta \in (0,\infty]$, $\ppShapeUnrA{\beta,\usc}$ is deterministic and finite on $\sA$, and there exists a $\xi \in(\ri \sA)\setminus\{\orig\}$ such that $\m \in \subspaceA \cap \superDiffUnrA{\beta,\usc}(\xi)$. Let $\beta_n \in (0,\infty]$ and $\m_n \in \superDiffUnrA{\beta_n, \usc}(\ri\sA)$ be such that $\beta_n \rightarrow \beta$ and $\m_n \rightarrow \m$.  For each $n$, let $\B{\beta_n, \m_n}$ be an $L^1(\bbP)$ shift-covariant $\beta_n$-recovering cocycle on $\sA$ which satisfies the conclusions of Theorem \ref{thm:Cocycles}\eqref{Cocycles.c} and \eqref{Cocycles.b}.  Then there is a subsequence $\{n_k\}_{k \geq 1}$ such that $\B{\beta_{n_k}, \m_{n_k}}$ converges weakly to a random variable $\B{\beta, m}$ on $(\wt\Omega,\wt\Sig)$, which also satisfies the conclusions of Theorem \ref{thm:Cocycles}\eqref{Cocycles.c} and \eqref{Cocycles.b} on $\sA$.
     \end{lemma}
     
     \begin{proof}
          Let $\beta_n \rightarrow \beta$, $m_n \rightarrow m$, and $\B{\beta_n, m_n}$ be as in the statement. Recovery implies that for $x \in \Z^d$ and $z \in \sA$, $\B{\beta_n,m_n}(x,x+z) \geq \pote(\w(\That_x \widehat{\w}),z)-\beta_n^{-1}\log p(z)$; combining this with convergence of the mean vectors $m_n \rightarrow m$, we see that the distributions of $(\w,\B{\beta_n, m_n})\in\wt\Omega$ under $\bbP$ are tight.
          Let $\wt\bbP$ be a weak limit along a subsequence $\{n_k\}_{k \geq 1}$. Then, we can repeat the proof of Theorem \ref{thm:Cocycles}\eqref{Cocycles.b} to show that, under $\wt\bbP$, the coordinate projection $\B{\sA}(x,x+z)$ is integrable with mean $\m\cdot z$. The covariance property $\eqref{Bcov}$ and cocycle property $\eqref{Bcoc}$ hold for each $\B{\beta_n, m_n}$ and are inherited in the limit since coordinate projections and shifts are continuous in the product topology. 

         Define  the function $f:(0,\infty]\times\R^{\Z^d\times\rangeA}\to\R$ by
         \begin{align*}
         f(\beta,B)=\begin{cases}
         \frac{1}{\beta} \log \sum_{z \in \rangeA} p(z) \exp \bigl\{ -\beta B(x,x+z) - \beta \pote(\omega(\wt\T_x \wt\w), z) \bigr\}&\text{if }\beta\in(0,\infty),\\
         f(\infty,B)=-\min_{z \in \rangeA} (B(x,x+z) + \pote(\omega(\wt\T_x \wt\w), z)).
         \end{cases}
         \end{align*}
         Note that $f$ is continuous on $(0,\infty]\times\R^{\Z^d\times\rangeA}$, including at $\beta=\infty$.           

        Thus, the recovery properties \eqref{Brecbeta} and \eqref{Brecinf} that  $\B{\beta_{n_k},\m_{n_k}}$ satisfy for each $k$ can be transferred to $\B{\sA}$ by applying Skorohod's representation theorem.
 \end{proof}

	\section{Semi-infinite path measures}\label{sec:semiinf}
	
	
	We begin by defining semi-infinite polymer measures from a shift-covariant recovering cocycle.  Throughout this section, $(\Omega,\Sig,\bbP,\{\T_x:x\in\ZZ^d\})$ satisfies the general stationary setting described in the first paragraph of Section \ref{sec:rwrp}.  We note that the extended space $(\Omegahat,\Sighat,\bbPhat,\{\That_x:x\in\ZZ^d\})$ from Theorem \ref{thm:Cocycles} satisfies these hypotheses.  We switch to this case in Section \ref{sec:MainThmProofs} where we prove Theorem \ref{thm:CocycleMeasuresFace}.
	
	\begin{definition}\label{pathMeasureDefPos}
		Let $B$ be an $L^1(\bbP)$ $\T$-covariant $\beta$-recovering cocycle on a face $\sA \in \faces$ {\rm(}possibly $\cone$ itself\,{\rm)}.  
		For $x_0\in\ZZ^d$, $0 < \beta < \infty$, and $\w\in\Omega$ such that \eqref{eqn:rec-beta} is satisfied for all $x\in\ZZ^d$,
		define $\Qinf{x_0}{\beta, B, \w}$ to be the probability measure on $\PathsSemiInf{x_0}(\rangeA)$ that is the distribution of the Markov process $X_{0:\infty}$ with transition probability from $x\in\ZZ^d$ to $x+z$, $z\in\rangeA$, given by 
		\[p(z) e^{-\beta B(x, x+z, \w) - \beta \pote(\T_x \w, z)}.\]
	\end{definition} 

For a subset $A\subset\ZZ^d$, let $\prob_A$ be the set of all probability measures on $A$, with the convention that $\prob_\varnothing=\varnothing$. 

    \begin{definition}\label{def:tieBreaker}
        A tie-breaker is an element $\tie\in\Bigl(\prod_{A\subset\range}\prob_A\Bigr)^{\Z^d}$. A covariant tie-breaker is a measurable function $\tie:\Omega\to\Bigl(\prod_{A\subset\range}\prob_A\Bigr)^{\Z^d}$ that is covariant: for any $x,z\in\ZZ^d$ and $A\subset\range$, $\tie_{x,A}^{T_z\w}=\tie_{x+z,A}^\w$. 
    \end{definition}

	\begin{definition}\label{pathMeasureDefZero}
		Let $B$ be an $L^1(\bbP)$ $\T$-covariant $\infty$-recovering cocycle on a face $\sA \in \faces$ {\rm(}possibly $\cone$ itself\,{\rm)}.  
		For $x_0 \in \ZZ^d$, $\w \in \Omega$ such that \eqref{eqn:rec-inf} is satisfied for all $x\in\ZZ^d$, and a tie-breaker $\tie$ define $\Qinf{x_0}{\infty, B, \tie,\w}$ to be  the probability measure on  $\PathsSemiInf{x_0}(\rangeA)$ that is the distribution of the Markov process $X_{0:\infty}$ with transition probability from $x\in\ZZ^d$ to $x+z$, $z\in\rangeA$, given by $\tie_{x,A}(z)$, where
		\[A=\bigl\{z\in\rangeA:\B{}(x, x+z, \w) + \pote(\T_x \w, z)=0 \bigr\}.\]
	\end{definition}
	
	    
\begin{example}\label{ranking}
    Let $\sS$ denote the set of bijections from $\range$ to $\{1,\dotsc,|\range|\}$.
    One natural way to define a covariant tie-breaker is by taking a covariant measurable function $\mathfrak z:\Omega\to\sS^{\Z^d}$ and setting $\tie_{x,A}^\w$ to be a Dirac mass at the step in $A$ that has the lowest $\mathfrak z_x^\w$-value.  
	    In this case, the probability measure $\Qinf{x_0}{\infty, B, \tie,\w}$ is a Dirac mass on a single semi-infinite path in $\PathsSemiInf{x_0}(\rangeA)$, which is the path that out of $x$ follows the increment that satisfies $\B{}(x, x+z, \w) + \pote(\T_x \w, z)=0$ and uses the ranking given by $\mathfrak z_x$ to break the tie if there are multiple such increments. 
\end{example}

	

In the next theorem, we will use a condition to guarantee the transience of semi-infinite geodesic paths at zero temperature.

\begin{condition}\label{BCondition}
Given a face $\sA \in \faces$ {\rm(}possibly $\cone$ itself\,{\rm)} and
an $\infty$-recovering cocycle $B$ on the face $\sA$,
assume that $\bbP$-almost surely, for any finite set $A\subset\ZZ^d$, there exist a $y\in A$ and a $z\in\rangeA$ such that $y+z\not\in A$ and $\pote(\T_y\w,z)+B(y,y+z)=0$.
\end{condition}

\begin{remark}
This condition is necessary for transience at zero temperature since a set $A$ violating the condition acts as a trap from which the semi-infinite path cannot exit, $\Qinf{u}{\infty,B,\tie,\w}$-almost surely. We will show in Lemma \ref{lm:trans.zero} below that the converse is also true. 
Hence, this condition is, in fact, equivalent to having transient paths. 
\end{remark}

\begin{remark}\label{rk:zeroloop}
Lemma \ref{lm:trap} below shows that $\infty$-recovering cocycles that violate the above condition may exist. 
Nevertheless, by Lemma \ref{lm:rec-gen}, Condition \ref{BCondition} is always satisfied for the cocycles $\B{\sA,\beta, \m}$ that we construct in Theorem \ref{thm:Cocycles}.
\end{remark}

The following theorem shows that the measures $\Qinf{u}{\beta,B,\w}$ are consistent with both the restricted-length and unrestricted-length point-to-point measures. It also relates these measures to the Green's function \eqref{g:def}. As the next result is a precursor to Theorem \ref{thm:CocycleMeasuresFace}, it may be helpful to consult the statement and discussion around that theorem to contextualize the notation below.

\begin{theorem}\label{thm:Measures} 
	Fix a face $\sA\in \faces$ {\rm(}possibly $\cone$ itself\,{\rm)}. 
    Assume $\pote$ satisfies Conditions \ref{VCondition} and \ref{classLCondition} on $\sA$. 
    Let $\B{}$ be an $L^1(\bbP)$ $\T$-covariant $\beta$-recovering cocycle on the face $\sA$. 
    There exists a $\T$-invariant event $\Omega_B \subset \Omega$ with $\bbP(\Omega_B) = 1$ such that for all $\w \in \Omega_B$, the following hold.
	\begin{enumerate}[label=\rm(\alph{*}), ref=\rm\alph{*}] \itemsep=2pt 
		\item\label{Measures.consPos} \textup{(Consistency with $\Qunr{x}{y}{\beta, \w}$.)} If $\beta < \infty$, for all non-negative integers $j \leq k \leq n$, points $u, x, y \in \ZZ^d$ with $x-u$ and $y-x$ in  $\RsemiA$,  $x_{0:k} \in \PathsPtPKilled{u}{y}$ with $j = \min\{0 \leq i \leq k : x_i = x\}$, and $x_{k:n} \in \PathsNStep{y}{n-k}(\rangeA)$, then 
	\begin{align} 
	\begin{split}
	&\Qinf{u}{\beta, B, \w}(X_{0:\stoppt{y}+n-k} = x_{0:n} \given \stoppt{x} \leq \stoppt{y} <\infty, X_{0:\stoppt{x}} = x_{0:j}, X_{\stoppt{y}:\stoppt{y}+n-k}=x_{k:n}) \\
		&\qquad=\Qinf{u}{\beta, B, \w}(X_{\stoppt{x}:\stoppt{y}} = x_{j:k} \given \stoppt{x} \leq \stoppt{y} <\infty) = \Qunr{x}{y}{\beta, \w}(X_{0:\stoppt{y}} = x_{j:k}).  
		\end{split}\label{eqn:consist}
	\end{align}

	\item\label{Measures.Greens} \textup{(Consistency with $\greens(x,y)$.)} If $\beta<\infty$, then for each $x, y \in \ZZ^d$ such that $y-x \in \RsemiA$,
    \begin{align}
         \rwE^{\Qinf{x}{\beta,B,\w}}\Bigl[ \sum_{n=0}^\infty \one_{\{X_n = y\}} \Bigr] = \greens(x,y)e^{-\beta B(x,y)}. \label{eqn:greensExpectation}
    \end{align}
    Hence, if $x\ne y$,
    \begin{align}
         \rwE^{\Qinf{x}{\beta,B,\w}}\Bigl[ \sum_{n=0}^\infty \one_{\{X_n = y\}} \Big| \stoppt{y}<\infty \Bigr] = \frac{\greens(x,y)}{\prtunr{x}{y}{\beta,\w}}. \label{eqn:greensExpCond}
    \end{align}

    \item\label{Measures.consPosRes} \textup{(Consistency with $\Qres{x}{y}{n}{\beta, \w}$.)} If $\beta < \infty$, for all non-negative integers $j \leq k \leq n$, points $u, x, y \in \ZZ^d$ such that $x-u \in \DnA{j}$ and $y-x \in \DnA{k-j}$, $x_{0:j} \in \PathsPtPResKM{u}{x}{0}{j}$, $x_{j:k} \in \PathsPtPResKM{x}{y}{j}{k}$, and $x_{k:n} \in \PathsNStep{y}{n-k}(\rangeA)$, 
	\begin{align}
    \begin{split}
		\Qinf{u}{\beta, B, \w}(X_{0:n} = x_{0:n} \given X_{0:j} = x_{0:j}, X_{k:n} = x_{k:n}) &= \Qinf{u}{\beta, B, \w}(X_{j:k} = x_{j:k} \given X_j = x, X_k = y) \\
		&= \Qres{x}{y}{k-j}{\beta, \w}(X_{0:k-j} = x_{j:k}).
    \end{split}\label{eqn:consistRes}
	\end{align} 		
	
	\item\label{Measures.consZero} If $\beta = \infty$, for each $x \in \ZZ^d$ and tie-breaker $\tie$, $\Qinf{x}{\infty, B, \tie,\w}$ is supported on a set of semi-infinite geodesics rooted at $x$ and having increments in $\rangeA$.
	
	\item\label{Measures.LLNLimit} If $\beta<\infty$, then for each $x \in \ZZ^d$, 
	 \[\Qinf{x}{\beta, B, \w}(\text{all limit points of } X_n/n \text{ are contained in } \facetRes{\mVec(B),\sA}) = 1.\]
	Similarly, if $\beta=\infty$, then for each $x \in \ZZ^d$ and tie-breaker $\tie$, 
	\[\Qinf{x}{\beta, B,\tie, \w}(\text{all limit points of } X_n/n \text{ are contained in } \facetRes{\mVec(B),\sA}) = 1.\]
 	\end{enumerate}		
Furthermore, we have the following.
    \begin{enumerate}[label=\rm(\alph{*}), ref=\rm\alph{*},resume] \itemsep=2pt 
	\item\label{Measures.trans.pos} If $\beta<\infty$, then for any {\rm(}and hence all\,{\rm)} $x\in\ZZ^d$, $\Qinf{x}{\beta, B, \w}\{\abs{X_n}_1\to\infty\}=1$ $\bbP$-almost surely if and only if either $\sA\ne\sA_\orig$ or both $\sA=\sA_\orig$ and Condition \ref{trans} holds.
			
	\item\label{Measures.trans.zero} If $\beta=\infty$ and Condition \ref{BCondition} is satisfied, then there exists a covariant tie-breaker $\tie$ such that for $\bbP$-almost every $\w$ and for every $x\in\ZZ^d$, $\Qinf{x}{\infty, B,\tie, \w}\{\abs{X_n}_1\to\infty\}=1$.
			
    \item\label{Measures.Directed} For $\beta\in(0,\infty)$, if  $\Qinf{\orig}{\beta, B, \w}(\abs{X_n}_1\to\infty)=1$, $\bbP$-almost surely, then  for $\bbP$-almost every $\w$, for all $x\in\ZZ^d$,
    \[\Qinf{x}{\beta, B, \w}(X_{0:\infty} \text{ is directed into } \facetUnr{\mVec(B), \sA}) = 1.\]
    Similarly, for $\beta=\infty$ and a covariant tie-breaker $\tie$, if $\Qinf{\orig}{\beta, B,\tie, \w}(\abs{X_n}_1\to\infty)=1$, $\bbP$-almost surely, then  for $\bbP$-almost every $\w$, for all $x\in\ZZ^d$,
    \[\Qinf{x}{\beta, B,\tie, \w}(X_{0:\infty} \text{ is directed into } \facetUnr{\mVec(B),\sA}) = 1.\]
	\end{enumerate}
\end{theorem}

\begin{remark}\label{rk:Qunr->Qres}
If we consider a restricted-length model with $\pote$ satisfying Conditions \ref{VConditionRes} and \ref{classLCondition}, then we can use the space-time cocycles from Remark \ref{rmk:obtainingSTCocycles} to define semi-infinite polymer measures that are consistent with the restricted-length point-to-point polymer measures. Unlike the measures in the above theorem, these polymer measures are not necessarily consistent with the unrestricted-length point-to-point polymer measures. 
\end{remark}
	
We will prove the various claims in the above theorem as separate lemmas, beginning with the consistency with the finite-path measures.  

\begin{lemma}\label{lem:consistencyPos}
		Fix a face $\sA\in \faces$ {\rm(}possibly $\cone$ itself\,{\rm)}.
		Let $0 < \beta < \infty$ and $B$ be an $L^1(\bbP)$ $\T$-covariant $\beta$-recovering cocycle on $\sA$. 
	Take $\w\in\Omega$ such that \eqref{eqn:rec-beta} holds for all $x\in\ZZ^d$.  
	Then for all non-negative integers $j \leq k \leq n$, points $u, x, y \in \ZZ^d$ with $x-u \in \RsemiA$ and $y-x \in \RsemiA$, $x_{0:k} \in \PathsPtPKilled{u}{y}$ with $j = \min\{0 \leq i \leq k : x_i = x\}$, and $x_{k:n} \in \PathsNStep{y}{n-k}(\rangeA)$, 
 \eqref{eqn:consist} holds.
\end{lemma}

\begin{proof} Take $\beta, B, \w,j,k,n,u, x, y, x_{0:n}$ as in the statement. Let $\stoppt{x,y} = \inf\{k \geq \stoppt{x} : X_k = y\}$ where $\inf \varnothing = \infty$.  
Note that $\{\stoppt{x} \leq \stoppt{y}\}$ is in 
the stopped $\sigma$-algebra corresponding to the stopping time $\stoppt{x}$: 
for any integer $t \ge 0$, \[\{\stoppt{x}=t\} \cap \{\stoppt{x} \leq \stoppt{y}\} = \{\stoppt{x}=t\}\cap \{\stoppt{y}\ge t\} \in \sigma(X_k:k\le t).\]


In the second equality of the next computation, use the fact that on the event $\{\stoppt{x} \leq \stoppt{y}\}$, $\stoppt{y} = \stoppt{x,y}$.  In the third equality, use the above measurability observation and the strong Markov property. In the second-to-last equality, use the cocycle property. We have
	\begin{align*}
		&\Qinf{u}{\beta, B, \w}(X_{\stoppt{x}:\stoppt{y}} = x_{j:k} \given \stoppt{x} \leq \stoppt{y} <\infty) \\
		&= \frac{\Qinf{u}{\beta, B, \w}(\stoppt{x} \leq \stoppt{y} <\infty, X_{\stoppt{x}:\stoppt{y}} = x_{j:k})}{\Qinf{u}{\beta, B, \w}(\stoppt{x} \leq \stoppt{y} <\infty) } \\
		&= \frac{\Qinf{u}{\beta, B, \w}(\stoppt{x} \leq \stoppt{y},\stoppt{x,y}<\infty, X_{\stoppt{x}:\stoppt{x,y}} = x_{j:k})}{\Qinf{u}{\beta, B, \w}(\stoppt{x} \leq \stoppt{y},\stoppt{x,y} <\infty) } \\
		&= \frac{\Qinf{u}{\beta, B, \w}(\stoppt{x} \leq \stoppt{y})\Qinf{x}{\beta, B, \w}(\stoppt{y}<\infty, X_{0:\stoppt{y}} = x_{j:k})}{\Qinf{u}{\beta, B, \w}(\stoppt{x} \leq \stoppt{y})\Qinf{x}{\beta, B, \w}(\stoppt{y} <\infty) } \\
		&= \frac{\prod_{i=j}^{k-1} p(x_{i+1}-x_i) e^{-\beta B(x_{i}, x_{i+1}, \w) - \beta \pote(\T_{x_i} \w, x_{i+1}-x_i)}}{\sum_{k=0}^\infty \sum_{\pi_{0:k} \in \PathsPtPKilledRes{x}{y}{k}} \prod_{i=0}^{k-1} p(\pi_{i+1}-\pi_i) e^{-\beta B(\pi_i, \pi_{i+1}, \w) - \beta \pote(\T_{\pi_i} \w, \pi_{i+1}-\pi_i)}}  \\
		&= \frac{e^{-\beta B(x, y, \w)} \prod_{i=j}^{k-1} p(x_{i+1}-x_i) e^{-\beta \pote(\T_{x_i} \w, x_{i+1}-x_i)}}{e^{-\beta B(x, y, \w)}\prtunr{x}{y}{\beta, \w}}  \\
		&= \Qunr{x}{y}{\beta, \w}(X_{0:\stoppt{y}} = x_{j:k}).
	\end{align*}
	Similarly, 
	\begin{align*}
		&\Qinf{u}{\beta, B, \w}(X_{0:\stoppt{y}+n-k} = x_{0:n} \given \stoppt{x} \leq \stoppt{y} <\infty, X_{0:\stoppt{x}} = x_{0:j}, X_{\stoppt{y}:\stoppt{y}+n-k} = x_{k:n}) \\
		&= \frac{\Qinf{x}{\beta, B, \w}(\stoppt{y}<\infty, X_{0:\stoppt{y}+n-k} = x_{j:n})}{\Qinf{x}{\beta, B, \w}(\stoppt{y}<\infty, X_{\stoppt{y}:\stoppt{y}+n-k} = x_{k:n})} \\
		&= \frac{\Qinf{x}{\beta, B, \w}(\stoppt{y}<\infty, X_{0:\stoppt{y}} = x_{j:k}) \Qinf{y}{\beta, B, \w}(X_{0:n-k} = x_{k:n})}{\Qinf{x}{\beta, B, \w}(\stoppt{y}<\infty) \Qinf{y}{\beta, B, \w}(X_{0:n-k} = x_{k:n})} \\
		&= \frac{\Qinf{x}{\beta, B, \w}(\stoppt{y}<\infty, X_{0:\stoppt{y}} = x_{j:k})}{\Qinf{x}{\beta, B, \w}(\stoppt{y}<\infty) } \\
		&= \Qunr{x}{y}{\beta, \w}(X_{0:\stoppt{y}} = x_{j:k}),  
	\end{align*}
    where the last equality comes from the previous calculation.
	\end{proof}

    \begin{lemma}\label{lem:GreensExpectation}
        	Fix a face $\sA\in \faces$ {\rm(}possibly $\cone$ itself\,{\rm)}.
        	Fix $0 < \beta < \infty$ and let $B$ be an $L^1(\bbP)$ $\T$-covariant $\beta$-recovering cocycle on $\sA$.  For each $x, y \in \ZZ^d$ such that $y-x \in \RsemiA$, \eqref{eqn:greensExpectation} and \eqref{eqn:greensExpCond} hold.
    \end{lemma}
    \begin{proof}
        Take $\beta, B, x, y$ as in the statement.  First, use the cocycle property to obtain
        \[
        \Qinf{x}{\beta,B,\w}(X_n = y) = \sum_{x_{0:n} \in \PathsPtPRes{x}{y}{n}} \prod_{i=0}^{n-1} p(x_{i+1} - x_i) e^{-\beta B(x_i, x_{i+1}) - \beta \pote(\T_{x_i} \w, x_{i+1}-x_i)} = \prtres{x}{y}{n}{\beta,\w} e^{-\beta B(x,y)}.
        \]
        Sum over $n$ to obtain \eqref{eqn:greensExpectation}:
        \[
        \rwE^{\Qinf{x}{\beta,B,\w}}\Bigl[ \sum_{n=0}^\infty \one_{\{X_n = y\}} \Bigr] = \sum_{n=0}^{\infty} \Qinf{x}{\beta,B,\w}(X_n = y) = \sum_{n=0}^{\infty} \prtres{x}{y}{n}{\beta,\w} e^{-\beta B(x,y)} = \greens(x,y) e^{-\beta B(x,y)}.
        \]
        
        Similarly, if $x\ne y$,
        \begin{align}\label{Q-aux}
        \Qinf{x}{\beta,B,\w}(\stoppt{y}<\infty) = \sum_{n=0}^{\infty} \sum_{x_{0:n} \in \PathsPtPKilledRes{x}{y}{n}} \prod_{i=0}^{n-1} p(x_{i+1} - x_i) e^{-\beta B(x_i, x_{i+1}) - \beta \pote(\T_{x_i} \w, x_{i+1}-x_i)} = \prtunr{x}{y}{\beta,\w} e^{-\beta B(x,y)}.
        \end{align}
        
        Combining the two previous displays, we obtain \eqref{eqn:greensExpCond}.
    \end{proof}

	We next show consistency with the restricted-length finite quenched polymer measures.	
	\begin{lemma}\label{lem:consistencyPosRes}
		Fix a face $\sA\in \faces$ {\rm(}possibly $\cone$ itself\,{\rm)}.
		Fix $0 < \beta < \infty$ and let $B$ be an $L^1(\bbP)$ $\T$-covariant $\beta$-recovering cocycle on $\sA$. Take $\w\in\Omega$ such that \eqref{eqn:rec-beta} holds for all $x\in\ZZ^d$.  For $j, k, n \in \ZZ_{\ge0}$ with $j \leq k \leq n$, $u, x, y \in \ZZ^d$ such that $x-u \in \DnA{j}$ and $y-x \in \DnA{k-j}$, $x_{0:j} \in \PathsPtPResKM{u}{x}{0}{j}$, $x_{j:k} \in \PathsPtPResKM{x}{y}{j}{k}$, and $x_{k:n} \in \PathsNStep{y}{n-k}(\rangeA)$, 
    \eqref{eqn:consistRes} holds.	
\end{lemma}

\begin{proof}
			Take $\beta, B, \w,j,k,n,u, x, y$, and $x_{0:n}$ as in the statement. Using the Markov property and the cocycle property,
			
			\begin{align*}
					&\Qinf{u}{\beta, B, \w}(X_{0:n} = x_{0:n} \given X_{0:j} = x_{0:j}, X_{k:n} = x_{k:n}) \\
					&= \frac{\Qinf{u}{\beta, B, \w}(X_{0:n} = x_{0:n})}{\Qinf{u}{\beta, B, \w}(X_{0:j} = x_{0:j},X_{k:n} = x_{k:n})} \\
					&= \frac{\Qinf{u}{\beta, B, \w}(X_{0:j} = x_{0:j})\Qinf{x}{\beta, B, \w}(X_{0:k-j} = x_{j:k})\Qinf{y}{\beta, B, \w}(X_{0:n-k} = x_{k:n})}{\Qinf{u}{\beta, B, \w}(X_{0:j} = x_{0:j})\Qinf{x}{\beta, B, \w}(X_{k-j} =y)\Qinf{y}{\beta, B, \w}(X_{0:n-k} = x_{k:n})} \\
					&= \frac{\prod_{i=j}^{k-1} p(x_{i+1}-x_i) e^{-\beta B(x_i, x_{i+1}, \w) - \beta \pote(\T_{x_i} \w, x_{i+1}-x_i)}}{ \sum_{\pi_{0:k-j} \in \PathsPtPRes{x}{y}{k-j}} \prod_{i=0}^{k-j-1} p(\pi_{i+1}-\pi_i) e^{-\beta B(\pi_i, \pi_{i+1}, \w) - \beta \pote(\T_{\pi_i} \w, \pi_{i+1}-\pi_i)}} \\ 
					&= \frac{e^{-\beta B(x, y, \w)} \prod_{i=j}^{k-1} p(x_{i+1}-x_i) e^{-\beta \pote(\T_{x_i} \w, x_{i+1}-x_i)}}{ e^{-\beta B(x, y, \w)} \prtres{x}{y}{k-j}{\beta, \w}} \\
					&=\Qres{x}{y}{k-j}{\beta, \w}(X_{0:k-j} = x_{j:k}).\qedhere
				\end{align*}

\end{proof}
	
Now we work out the consistency in the zero temperature case. Recall that the definition of the quenched measures, in this case, requires a tie-breaker $\tie$. 
	
	\begin{lemma}\label{lem:consistencyZero}
			Fix a face $\sA\in \faces$ {\rm(}possibly $\cone$ itself\,{\rm)}.
			Let $B$ be an $L^1(\bbP)$ $\T$-covariant $\infty$-recovering cocycle on $\sA$.   
		Let $\w\in\Omega$ be such that \eqref{eqn:rec-inf} is satisfied for all $x\in\ZZ^d$.
		Then, for each $x \in \ZZ^d$, for any tie-breaker $\tie$, $\Qinf{x}{\infty, B,\tie, \w}$ is supported on a set of semi-infinite geodesics rooted at $x$ and having increments in $\rangeA$.  
	\end{lemma}
	
	\begin{proof}  
	The measure $\Qinf{x}{\infty, B,\tie, \w}$ only allows steps $z \in \rangeA$, so the paths stay within $x + \sA$.  
	Next, observe that $\Qinf{x}{\infty, B,\tie, \w}$-almost surely, $B(X_i, X_{i+1}) = -\pote(\T_{X_i} \w, X_{i+1}-X_i)$  for all $i \in \ZZ_{\ge0}$. We will show that this implies that $X_{0:\infty}$ is a geodesic.
	
	Let $x_{0:\infty}\in\PathsSemiInf{x}(\rangeA)$ be a semi-infinite path such that $B(x_i, x_{i+1}) = -\pote(\T_{x_i} \w, x_{i+1}-x_i)$  for each $i \in \ZZ_{\ge0}$.  Let $0 \leq j < k$, $n \in \ZZ_{>0}$, and $y_{0:n} \in \PathsPtPRes{x_j}{x_k}{n}$.  Using the cocycle and recovery properties,
	\begin{align*}
		-\sum_{i=j}^{k-1} \pote(\T_{x_i} \w, x_{i+1}-x_i)&= \sum_{i=j}^{k-1} B(x_i, x_{i+1}) = B(x_j, x_k) \\
		&= \sum_{i=0}^{n-1} B(y_i, y_{i+1}) \geq -\sum_{i=0}^{n-1} \pote(\T_{y_i} \w, y_{i+1}-y_i).
	\end{align*}  
Since this holds for all lengths $n$ and all paths $y_{0:n}$ from $x_j$ to $x_k$, $x_{j:k}$ is a geodesic from $x_j$ to $x_k$.  Since this holds for all $0 \leq j < k$, $x_{0:\infty}$ is a semi-infinite geodesic rooted at $x$.  
\end{proof}

Now that the consistency properties have been proven, we turn to directedness. 
We first address the question of transience versus recurrence of the semi-infinite polymers. Recall that $\Uset_\orig$ is the unique face of $\Uset$ that contains $\orig$ in its relative interior and that  $\range_{\orig}=\range\cap\Uset_\orig$.

\begin{lemma}\label{lm:trans.pos}
	Fix a face $\sA\in \faces$ {\rm(}possibly $\cone$ itself\,{\rm)}.
	Assume \eqref{V>0} holds and $\beta\in(0,\infty)$.
Let $B$ be an $L^1(\bbP)$ $\T$-covariant $\beta$-recovering cocycle on $\sA$. 
The following are equivalent:
\begin{enumerate}[label={\rm(\roman*)}, ref={\rm\roman*}]   \itemsep=2pt
    \item\label{lm:trans.i} Either  $\sA\ne\sA_\orig$ or both $\sA=\sA_\orig$  and Condition \ref{trans} holds.
    \item\label{lm:trans.ii} There exists an $x\in\ZZ^d$ such that $\bbP$-almost surely $\Qinf{x}{\beta, B, \w}\{\abs{X_n}_1\to\infty\}=1$.
    \item\label{lm:trans.iii} $\bbP$-almost surely, for all $x\in\ZZ^d$, $\Qinf{x}{\beta, B, \w}\{\abs{X_n}_1\to\infty\}=1$.
\end{enumerate}
\end{lemma}

\begin{proof}
Note that \eqref{lm:trans.ii} and \eqref{lm:trans.iii} are equivalent due to the shift-invariance of $\bbP$ and the shift-covariance of $B$. We, therefore, fix an $x\in\ZZ^d$ and prove that \eqref{lm:trans.i} is equivalent to \eqref{lm:trans.ii} with this choice of $x$.

From \eqref{eqn:greensExpectation} and a standard fact about time-homogeneous Markov chains, we see that the Markov chain $\Qinf{x}{\beta, B, \w}$ is recurrent if and only if $\greens(x,x)=\infty$. Let $\sigma_1$ be the time of first return of $X_n$ to the starting point: $\sigma_1 = \inf \{k \geq 1 : X_k = X_0\}$. Lemma 6.2 in \cite{Jan-Nur-Ras-22} gives 
    \begin{equation}\label{eqn:gr}
    \greens(x,x)=\frac1{1-\rwE_x\Bigl[e^{-\beta\sum_{k=0}^{\sigma_1-1}\pote(\T_{X_k}\w,X_{k+1}-X_k)}\one\{\sigma_1<\infty\}\Bigr]}\le\frac1{\rwP_x(\sigma_1=\infty)}\,.\end{equation}
If the reference random walk $\rwP_\orig$ is transient, then the above implies that $\greens(x,x)<\infty$ 
and, therefore, the Markov chain $\Qinf{x}{\beta, B, \w}$ is transient. This includes the case $\sA\ne\sA_\orig$ because, in that case, $\rangeA\setminus\range_\orig\ne\varnothing$, and once the reference random walk takes a step in $\rangeA\setminus\range_\orig$, it will never return to its starting point.


Suppose that $\rwP_\orig$ is recurrent and, in particular, $\sA=\sA_\orig$. Assume Condition \ref{trans} holds. We will prove that, $\bbP$-almost surely, $\Qinf{x}{\beta, B, \w}$ is again transient, for all $x\in\ZZ^d$. 

Without loss of generality, let $x = \orig$.  
By Condition \ref{trans}, with $\bbP$-probability one, there exist $y\in\Rsemi_{\sA_\orig}$, $z_0 \in \range_{\orig}$, and $\e > 0$ such that $\pote(\T_y\w, z_0) \ge \e$. \cite[Lemma 3.4]{Jan-Nur-Ras-22} shows that  $y$ can be written as $y = \sum_{z \in \range_\orig} \gamma_z z$ with $\gamma_z \in \ZZ_{\ge 0}$ and $\gamma_z \le C |y|_1$ for all $z \in \range_\orig$ and a finite positive constant $C$.  Use this to obtain an admissible path $y_{0:n}$ from $\orig$ to $y$ of length $n = \sum_{z \in \range_\orig} \gamma_z \leq C |\range_\orig| \cdot |y|_1$.  Next, we will extend this path to obtain a loop from $\orig$ to $\orig$.  Let $y_{n+1} = y + z_0$.  Then, use the same argument to construct a path $y_{n+1:m}$ from $y+z_0$ to $\orig$, with length $m-n-1 \leq C' |\range_\orig| \cdot |y|_1$.  Then $y_{0:m}$ is a loop which starts at $\orig$, travels to $y$, takes a step to $y+z_0$, and returns to $\orig$.  The length $m \leq C_1 |y|_1$, where $C_1 = (C+C')|\range_\orig|+1$.  Let $\kappa = \min_{z \in \range_\orig} p(z)$.  Then, $\rwP_\orig(X_{0:m} = y_{0:m}) \geq \kappa^{C_1 |y|_1}$. 

For the loop constructed in the above paragraph,
\[\exp\Bigl\{ -\beta\sum_{i=0}^{m-1} \pote(\T_{y_i} \w, y_{i+1}-y_i) \Bigr\} \leq \exp\{ -\beta\pote(\T_{y} \w, z_0) \} \leq e^{-\beta\e}.\]  
On the other hand, for any loop $y'_{0:\ell}$ from $\orig$ to $\orig$, \[\exp\Bigl\{ -\beta\sum_{i=0}^{\ell-1} \pote(\T_{y'_i} \w, y'_{i+1}-y'_i) \Bigr\} \leq 1\] because $\pote(\T_{v} \w, z) \geq 0$ for all $v \in \Rsemi_{\sA_\orig}$ and $z \in \range_\orig$.  Therefore,
\begin{align*}
\rwE_\orig \Bigl[ e^{-\beta\sum_{i=0}^{\sigma_1-1} \pote(\T_{X_k}\w, X_{k+1}-X_k) } \one_{\{ \sigma_1 < \infty \}} \Bigr] &\leq \rwP_\orig(X_{0:m} \neq y_{0:m}) + e^{-\beta\e} \rwP_\orig(X_{0:m} = y_{0:m}) \\
&= 1 - (1 - e^{-\beta\e}) \rwP_\orig(X_{0:m}=y_{0:m}) \leq 1 - (1 - e^{-\beta\e}) \kappa^{C_1 |y|_1}.
\end{align*}
Using \eqref{eqn:gr}, we see that $\bbP$-almost surely,
\[
g(\orig,\orig) \leq \kappa^{-C_1 |y|_1} (1-e^{-\beta\e})^{-1}
\]
and $\Qinf{\orig}{\beta, B, \w}$ is transient.

If, on the other hand, Condition \ref{trans} is violated, then with positive $\bbP$-probability, $\pote(\T_y\w,z)=0$ for all 
$y\in\Rsemi_{\sA_\orig}$ and $z\in\rangeA$. On this event, $\Qinf{\orig}{\beta, B, \w}$-almost surely,
$\pote(\T_{X_k}\w,X_{k+1}-X_k)=0$ for all $k\in\ZZ_{\ge0}$ and  \eqref{eqn:gr} gives $\greens(\orig,\orig)=(1-\rwP_\orig(\sigma_1<\infty))^{-1}=\infty$. 
The lemma is proved.
\end{proof}

\begin{lemma}\label{lm:trans.zero}
	Fix a face $\sA\in \faces$ {\rm(}possibly $\cone$ itself\,{\rm)}.
	Assume \eqref{V>0} holds and $\beta=\infty$.
Let $B$ be an $L^1(\bbP)$ $\T$-covariant $\infty$-recovering cocycle on $\sA$. Assume $B$ satisfies Condition \ref{BCondition}.
Then there exists a covariant tie-breaker $\tie$ such that, $\bbP$-almost surely, for any $x\in\ZZ^d$, $\Qinf{x}{\infty, B,\tie, \w}\{\abs{X_n}_1\to\infty\}=1$.
\end{lemma}

Before we prove the lemma, we need a definition and an intermediate result. 
Given an $\infty$-recovering cocycle $B$ on a face $\sA\in\faces$ we will say that  $x$ and $y$ in $\ZZ^d$ communicate, and will write $x	\leftrightsquigarrow y$, if there exist paths $x_{0:k}\in\PathsPtPUnr{x}{y}(\rangeA)$ and $y_{0:\ell}\in\PathsPtPUnr{y}{x}(\rangeA)$, $k,\ell\in\Z_{>0}$, such that
$\pote(\T_{x_i}\w,x_{i+1}-x_i)=B(x_i,x_{i+1})=0$ and 
$\pote(\T_{y_j}\w,y_{j+1}-y_j)=B(y_j,y_{j+1})=0$ for all integers $i\in[0,k-1]$ and $j\in[0,\ell-1]$. Note that $\leftrightsquigarrow$ is symmetric and transitive, but it is not necessarily reflexive. Namely, $x\leftrightsquigarrow x$ holds if and only if there is a non-empty admissible loop from $x$ to $x$ along which $\pote$ and $B$ vanish.  In other words, $x$ does not communicate with itself if and only if it does not communicate with any site $y\in\ZZ^d$. We will call such a site \emph{non-essential} and the sites that do communicate with themselves \emph{essential}. Hence, $\leftrightsquigarrow$ is an equivalence relation on the set of essential sites. 

%


\begin{lemma}\label{lm:trans.aux}
	Fix a face $\sA\in \faces$ {\rm(}possibly $\cone$ itself\,{\rm)}.
	Assume \eqref{V>0} holds and $\beta=\infty$.
Let $B$ be an $\infty$-recovering cocycle on $\sA$. If there exists a path $x_{0:k}\in\PathsPtPRes{x}{x}{k}$ with $x\in\ZZ^d$ and $k\in\Z_{>0}$, such that  \begin{align}\label{rec-temp}
    \pote(\T_{x_i}\w,x_{i+1}-x_i)+B(x_i,x_{i+1})=0\quad\text{for all integers }i \in [0,k-1],
    \end{align}
    then $x$ is essential.
\end{lemma}

\begin{proof}
Sum \eqref{rec-temp} over $i$ and use the cocycle property and $B(x,x)=0$ to get that  $\sum_{i=0}^{k-1}\pote(\T_{x_i}\w,x_{i+1}-x_i)=0$.
Next, observe that the steps of the loop $x_{0:k}$ are all in $\range_\orig$. By \eqref{V>0}, we get that $\pote(\T_{x_i}\w,x_{i+1}-x_i)=0$ along the loop, which, in turn, implies that $B(x_i,x_{i+1})=0$ along the loop. Consequently, $x\leftrightsquigarrow x$, and the lemma is proved.
%
\end{proof}

\begin{proof}[Proof of Lemma \ref{lm:trans.zero}]
We begin by defining the covariant tie-breaker. 
To this end, fix a bijection $\mathfrak z$ from $\range$ to $\{1,\dotsc,|\range|\}$.
For any non-essential point $x$  and subset $A\subset\range$, we define $\tie_{x,A}$ as a Dirac mass at the step in $A$ that has the minimum $\mathfrak z$ value. 

Fix another bijection $\mathfrak z'$ from $\ZZ^d$ to $\Z_{>0}$. 
If $x$ is an essential point, then let $A$ be its equivalence class. We treat first the case where $A$ is finite. Consider the pairs $(y,z)$ that appear in Condition \ref{BCondition} for this set $A$. 
Among these pairs, choose the ones where $y$ has the smallest $\mathfrak{z}'$ value. If there are multiple such pairs, select the unique pair with the lowest $\mathfrak{z}$ value for $z$.

Next, for a set $C\subset\range$ containing  $z$, we define $\tie_{y,C}$ as a Dirac mass at $z$. If $z\not\in C$, let $\tie_{y,C}$ be a Dirac mass at the point in $C$ with the lowest $\mathfrak z$ value.

To define the tie-breaker at points in $A\setminus{\{y\}}$, we first construct a spanning tree of $A$ as follows. We start by selecting the point $u\in A\setminus{\{y\}}$ with the smallest $\mathfrak{z}'$ value among all such points. We then choose the first step $z_1\in\range_\orig$ of an admissible path from $u$ to $y$ that remains entirely within $A$ and avoids $u$ after leaving it. If there are multiple choices for $z_1$, we select the one with the smallest $\mathfrak{z}$ value.

We continue by selecting the first step $z_2\in\range_\orig$ of an admissible path from $u+z_1$ to $y$ that remains entirely within $A$, does not visit $u$, and avoids $u+z_1$ after leaving it. We iterate this process until we reach $y$, constructing a path from $u$ to $y$ that forms a branch of the spanning tree.

Next, we select a point $v$ in $A$ that is not on the just constructed path and that has the smallest $\mathfrak{z}'$ value among all such points. We then construct a path from $v$ to $y$, similarly to how we constructed the path from $u$ to $y$, except that if this second path intersects with the first one, we stop its construction at the intersection point. We continue this process until we have exhausted all points in $A$, resulting in a spanning tree of $A$ with all paths leading to the root $y$.

For a point $u\in A$, let $z_1\in\range_\orig$ be the first step on the path from $u$ to $y$ in the spanning tree constructed above. If $C\subset\range$ and $z_1\in C$, define $\tie_{u,C}$ to be a Dirac mass at $z_1$. Otherwise, we define $\tie_{u,C}$ to be a Dirac mass at the element in $C$ with the smallest $\mathfrak{z}$ value. Note that this second case is irrelevant because we know that $\pote(\T_u\w,z_1)+B(u,z_1)=0$.

Next, consider the case when $A$ is infinite. In this scenario, we will construct a spanning forest that contains semi-infinite coalescing paths. We will provide a brief overview of the construction, as it is similar to the one described above for the case of a finite $A$.

We begin by selecting the point $u\in A$ with the smallest $\mathfrak{z}'$ value. Since $A$ is infinite, we can find a sequence of points $x_n\in A$ such that $\abs{x_n}_1\geq n$ for each $n$. For every $n$, there exists an admissible path from $u$ to $x_n$ that does not return to $u$ after leaving it. By restricting to a subsequence $n_j$ if necessary, we can ensure that all paths from $u$ to $x_{n_j}$ pass through the same step $z_1\in\range_\orig$. If there are multiple options, we choose the $z_1$ with the smallest $\mathfrak{z}$ value.

Continuing the construction, we obtain a semi-infinite path that starts at $u$, remains entirely within $A$, and contains no loops. Once this path is constructed, we select a point $v\in A$ that is not on the path and construct another semi-infinite path that remains within $A$. If this second path intersects with the first one, we stop its construction at that intersection point. We repeat this process until all points in $A$ have been exhausted.

Using the spanning forest we constructed, we define the tie-breaker $\tie_x$ for points $x\in A$ similarly to when $A$ is finite.

The shift-covariance of $\tie$ is immediately evident from the construction and the shift-covariance of $B$. Additionally, the construction guarantees the following properties hold with $\Qinf{x}{\infty, B,\tie, \w}$-probability one:
\begin{enumerate}[label={\rm(\roman*)}, ref={\rm\roman*}]   \itemsep=2pt
\item If $X_{0:\infty}$ enters a finite equivalence class $A$, it will exit $A$ without forming a loop inside it.
\item Once $X_{0:\infty}$ exits an equivalence class $A$, it cannot return to $A$ because doing so would create a loop that remains entirely within $A$.
\item If $X_{0:\infty}$ enters an infinite equivalence class $A$, it remains within $A$ and does not form any loops inside $A$.
\item By Lemma \ref{lm:trans.aux}, the path does not form loops that include non-essential points.
\end{enumerate}

These properties imply that, almost surely, under $\Qinf{x}{\infty, B,\tie, \w}$, $X_{0:\infty}$ does not revisit any previously visited points, and as a result, $\abs{X_n}_1$ tends to infinity.
\end{proof}

\begin{lemma}\label{lm:trap}
Assume $\sA\in\faces$ is such that $\orig\in\ri\UsetA$. Fix an $\w\in\Omega$. 
Assume that $\pote(\T_v\w,z)\ge0$ for all $v\in\RgroupA$ and $z\in\rangeA$. Assume that there exist a $k\in\ZZ_{>0}$ and a loop  $x_{0:k}\in\PathsPtPRes{\orig}{\orig}{k}$ with $\pote(\T_{x_i}\w,x_{i+1}-x_i)=0$ for each $i\in[0,k)$.  Then $B(x,y)=\freeunr{x}{\orig}{\infty}(\w)-\freeunr{y}{\orig}{\infty}(\w)$, $x,y\in\RgroupA$, is an $\infty$-recovering cocycle. For any tie-breaker $\tie$,     
\[\Qinf{\orig}{\infty, B, \tie,\w}\Big\{\pote(\T_{X_i}\w,X_{i+1}-X_i)=0\text{ for all }i\in\Z_{\ge0}\Bigr\}=1.\]
\end{lemma}

\begin{remark}
This lemma shows that Condition \ref{BCondition} is violated when, with positive probability, a loop $x_{0:k}$ as in the statement of the lemma exists and there is a finite set of sites $y\in\RgroupA$ for which there exist an $\ell\in\ZZ_{>0}$ and a path $y_{0:\ell}\in\PathsPtPRes{\orig}{y}{\ell}$, such that $\pote(\T_{y_i}\w,y_{i+1}-y_i)=0$ for each $i\in[0,\ell)$.
This situation can arise in various scenarios, including the standard first-passage percolation model described in Example \ref{examples}\eqref{itm:randomGrowth}, when the edge weights are i.i.d.\ and non-negative, and there exists a positive probability of encountering a zero edge weight. 
Nevertheless, Lemma \ref{lm:rec-gen} asserts that even in such cases, Condition \ref{BCondition} remains valid for the specific cocycles $\B{\sA,\beta, \m}$ constructed in Theorem \ref{thm:Cocycles}.
\end{remark}

\begin{proof}[Proof of Lemma \ref{lm:trap}]
The cocycle property is immediate from the definition of $B$. For the recovery property, write
\begin{align*}
\min_{z \in \rangeA} \{ B(x, x+z) + \pote(\T_x \w, z) \} 
&=\freeunr{x}{\orig}{\infty}(\w)-\max_{z \in \rangeA}\{\freeunr{x+z}{\orig}{\infty}(\w) - \pote(\T_x \w, z)\}\\
&=\freeunr{x}{\orig}{\infty}(\w)-\max_{z \in \rangeA}\sup_{k \in \ZZ_{\ge0}} \sup_{x_{0:k}\in \PathsPtPRes{x+z}{\orig}{k}}\Bigl\{ \sum_{i=0}^{k-1} (-\pote(\T_{x_i}\w, x_{i+1}-x_i))- \pote(\T_x \w, z)\Bigr\}\\
&=\freeunr{x}{\orig}{\infty}(\w)-\sup_{k \in \ZZ_{>0}} \sup_{x_{0:k}\in \PathsPtPRes{x}{\orig}{k}}\sum_{i=0}^{k-1} (-\pote(\T_{x_i}\w, x_{i+1}-x_i)).
\end{align*}
The right-hand side vanishes when $x\ne\orig$ because, to go from $x$ to $\orig$, the path needs at least one step, and therefore, the supremum over $k>0$ is the same as the supremum over $k\ge0$, which gives $\freeunr{x}{\orig}{\infty}(\w)$. When $x=\orig$, both terms on the right-hand side vanish because of the assumptions that the potential is non-negative (and hence both terms are suprema of non-positive quantities) and that there exists an admissible loop from $\orig$ to $\orig$, of length at least one, which has a total weight of zero. Either way, the right-hand side vanishes, and $B$ is an $\infty$-recovering cocycle. 

To see the second claim first note that Definition \ref{pathMeasureDefZero} implies that for any tie-breaker $\tie$,
\begin{align}\label{loop-aux}
\Qinf{\orig}{\infty, B, \tie,\w}\Big\{B(X_i,X_{i+1})+\pote(\T_{X_i}\w,X_{i+1}-X_i)=0\text{ for all }i\in\Z_{\ge0}\Bigr\}=1.
\end{align}
Take $k\in\ZZ_{\ge0}$ and $\e>0$. There exists a path $Y^\e_{k:\ell}\in\PathsPtPResKM{X_k}{\orig}{k}{L}$ with 
$L\in\ZZ_{\ge k}$ and 
\[\sum_{j=k}^{L-1} (-\pote(\T_{Y^\e_j}\w, Y^\e_{j+1}-Y^\e_j))\ge \freeunr{X_k}{\orig}{\infty}(\w)-\e.\]
Rearranging and using the definition of $B$, we get
\[B(X_k,\orig)+\sum_{j=k}^{L-1} \pote(\T_{Y^\e_j}\w, Y^\e_{j+1}-Y^\e_j)\le\e.\]
Together with \eqref{loop-aux}, the cocycle property, and $B(\orig,\orig)=0$, this gives that 
$\Qinf{\orig}{\infty, B, \tie,\w}$-almost surely, 
\[\sum_{i=0}^{k-1}\pote(\T_{X_i}\w,X_{i+1}-X_i)
+\sum_{j=k}^{L-1}\pote(\T_{Y^\e_j}\w,Y^\e_{j+1}-Y^\e_j)\le\e.\]
Since the potential is non-negative, this implies 
    \[0\le\sum_{i=0}^{k-1}\pote(\T_{X_i}\w,X_{i+1}-X_i)\le\e.\]
Taking $\e\to0$ and using again the fact that the potential is non-negative gives that $\pote(\T_{X_i}\w,X_{i+1}-X_i)=0$ for all $i\in[0,k)$. Since $k$ was arbitrary, the lemma is proved.
\end{proof}

Now, we  address the directedness of the cocycle polymer measures. Recall the set $\facetUnr{\mVec(B),\sA}$ defined in \eqref{def:facetUnr}. 

\begin{theorem}\label{thm:dir}
Assume the setting of Theorem \ref{thm:Measures}. 
Let $B$ be a $\beta$-recovering cocycle on the face $\sA$ with $\mVec(B)$ as in \eqref{meanVec}. If $\beta=\infty$, then let $\tie$ be a covariant tie-breaker and, for the purpose of this theorem, abbreviate $\Qinf{\orig}{\beta, B,\w}=\Qinf{\orig}{\beta, B,\tie,\w}$. Assume that $\Qinf{\orig}{\beta, B, \w}\{\abs{X_n}_1\to\infty\}=1$, $\bbP$-almost surely.
Then $\bbP$-almost surely for any $x \in \ZZ^d$, 
\[\Qinf{x}{\beta, B, \w}\big(X_{0:\infty}\text{ is directed into }\facetUnr{\mVec(B),\sA}\big)=1.\] 
\end{theorem}

\begin{proof}
By the $\T$-covariance of $B$ and the $\T$-invariance of $\P$, we can take $x = \orig$ without any loss of generality. Fix $\kappa>0$.   
Let 
    \[\facetUnr{\mVec(B),\sA,\kappa}=\bigl\{\zeta\in\sA:\abs{\zeta}_1=1\text{ and }\exists\xi\in\facetUnr{\mVec(B),\sA}\text{ with }\abs{\zeta-\xi}_1<\kappa \bigr\}.\]    
Note that $\bbP$-almost surely, for any $\zeta\in\sA$, $\mVec(B)\cdot\zeta\ge\ppShapeUnr{\beta}(\zeta)$ by \cite[Theorem 2.14]{Jan-Nur-Ras-22}.  Since $\mVec(B)\cdot\zeta$ is continuous on all of $\sA$, $\ppShapeUnr{\beta}$ is continuous on $\ri \sA$, and $\ppShapeUnrA{\beta, \usc}$ is the unique continuous extension of $\ppShapeUnrA{\beta}$ to all of $\sA$, 
we conclude that $\mVec(B)\cdot\zeta\ge\ppShapeUnrA{\beta, \usc}(\zeta)$. 
Therefore, with $\bbP$-probability one, for all $\zeta \in \sA$ with $\abs{\zeta}_1 = 1$, $\zeta\not\in\facetUnr{\mVec(B),\sA,\kappa}$ implies $\mVec(B)\cdot\zeta>\ppShapeUnrA{\beta, \usc}(\zeta)$.
Since $\{\zeta\in\sA\setminus\facetUnr{\mVec(B),\sA,\kappa}:\abs{\zeta}_1=1\}$ is a compact set and $\ppShapeUnrA{\beta, \usc}$ is continuous on $\sA$, we see that 
\[\inf\big\{\mVec(B)\cdot\zeta-\ppShapeUnrA{\beta, \usc}(\zeta):\zeta\in\sA\setminus\facetUnr{\mVec(B),\sA,\kappa}\text{ with }\abs{\zeta}_1=1\big\}>0,\]
$\bbP$-almost surely. 

Fix any $\delta\in(0,1)$. The above implies that there exists a deterministic $\e>0$ such that 
\[\bbP\Bigl\{\mVec(B)\cdot\zeta-\ppShapeUnrA{\beta, \usc}(\zeta)\ge 3\e\ \forall\zeta\in\sA\setminus\facetUnr{\mVec(B),\sA,\kappa}\text{ with }\abs{\zeta}_1=1\Bigr\}>1-\frac\delta3.\]
The shape theorem for $B$ \cite[Theorem B.3]{Jan-Ras-18-arxiv} implies that there exists an $L_0>0$ such that for any $L\ge L_0$ we have
\[\bbP\Bigl\{B(\orig,y)\ge\mVec(B)\cdot y-\e\abs{y}_1\ \forall y\text{ with }\abs{y}_1\ge L\Bigr\}>1-\frac\delta3.\]
Applying \eqref{shape-UB} to $\freeunr{\orig}{y}{\beta}$ we get that there exists an $L_1\ge L_0$ such that 
\[\bbP\Bigl\{\freeunr{\orig}{y}{\beta}\le \ppShapeUnrA{\beta, \usc}(y)+\e\abs{y}_1\ \forall y\text{ with }\abs{y}_1\ge L_1\Bigr\}>1-\frac\delta3.\]
Putting all of the above together, we get
\begin{equation}\label{eqn:F-Bbound}\bbP\Bigl\{\freeunr{\orig}{y}{\beta}-B(\orig,y)\le-\e\abs{y}_1\ \forall y\text{ with }\abs{y}_1\ge L_1\text{ and }\frac{y}{\abs{y}_1}\in\sA\setminus\facetUnr{\mVec(B),\sA,\kappa}\Bigr\}>1-\delta.\end{equation}
At zero temperature, Lemma \ref{lem:consistencyZero} implies that, $\Qinf{\orig}{\infty, B, \w}$-almost surely, $X_{0:n}$ is an unrestricted-length geodesic, so $\freeunr{\orig}{X_n}{\infty} = \sum_{i=0}^{n-1} \pote(\T_{X_{i}}\w, X_{i+1}-X_i) = \sum_{i=0}^{n-1} B(X_i, X_{i+1}) = B(\orig, X_n)$.
This and the assumption that $\abs{X_n}_1\to\infty$, $\Qinf{\orig}{\infty, B, \w}$-almost surely, imply that on the event in the above probability
we have
\begin{align}\label{Q-F}
\Qinf{\orig}{\infty, B, \w}\Bigl\{\exists n_0 : \frac{X_n}{\abs{X_n}_1} \in \facetUnr{\mVec(B),\sA,\kappa}\ \forall n\ge n_0\Bigr\}=1.
\end{align}
To derive this at positive temperature, first observe that from \eqref{Q-aux},
    \[\prtunr{\orig}{y}{\beta}e^{-\beta B(\orig,y)}=\Qinf{\orig}{\beta, B, \w}(\stoppt{y}<\infty).\]
Thus, we have 
\[\bbP\Bigl\{\Qinf{\orig}{\beta, B, \w}(\stoppt{y}<\infty)\le e^{-\beta \e\abs{y}_1}\ \forall y\text{ with }\abs{y}_1\ge L_1\text{ and }\frac{y}{\abs{y}_1}\in\sA\setminus\facetUnr{\mVec(B),\sA,\kappa}\Bigr\}>1-\delta.\]
Again, since $|X_n|_1\to\infty$, $\Qinf{x}{\beta, B, \w}$-almost surely, Borel-Cantelli implies that on the event in the above probability, \eqref{Q-F} holds again.

Taking $\delta\to0$ implies that \eqref{Q-F} is a full $\bbP$-probability event. This, in turn, implies that 
\[\P\Bigl\{\Qinf{\orig}{\beta, B, \w}\Big(X_{0:\infty}\text{ is directed into }\overline{\facetUnr{\mVec(B),\sA,\kappa}}\Big)=1\Bigr\}=1,\]
where $\overline{\facetUnr{\mVec(B),\sA,\kappa}}$ is the closure of $\facetUnr{\mVec(B),\sA,\kappa}$. Since $\facetUnr{\mVec(B),\sA}$ is closed, $\bigcap_{\kappa>0}\overline{\facetUnr{\mVec(B),\sA,\kappa}}=\facetUnr{\mVec(B),\sA}$. Taking $\kappa\to0$ in the above display finishes the proof.
\end{proof}

 Next, we turn to the question of determining the possible limit points of $X_n / n$ under the polymer measure. We achieve this by proving a large deviation principle. 
Fix a face $\sA \in \faces$.  For each $\xi \in \UsetA \cap \QQ^d$, let $\{y_n(\xi)\}_{n \in \ZZ_{\ge0}}$ be a path with $y_0(\xi)=\orig$, $y_n(\xi)-y_{n-1}(\xi) \in \rangeA$ for all $n\in\Z_{>0}$, and $y_{kj} = kj\xi$ for all $k \in \ZZ_{\ge0}$ and some $j=j(\xi) \in \ZZ_{>0}$.

\begin{lemma}\label{lem:upperBoundCompact} Fix $\beta \in (0,\infty]$ and $\sA \in \faces$.
Assume Condition \ref{classLCondition} holds.  Let $B$ be an $L^1(\bbP)$ $\T$-covariant $\beta$-recovering cocycle on the face $\sA$. 
For $\bbP$-a.e.\ $\w$, all subsets $K \subset \RR^d$, and all $\delta > 0$, if $\beta < \infty$, 
	\begin{align*}
		&\varlimsup_{n \rightarrow \infty} \frac{1}{n} \log \rwE_{\orig} \left[ e^{-\beta B(\orig, X_n) - \beta \sum_{i=0}^{n-1} \pote(\T_{X_i} \w, X_{i+1}-X_i)} \one\{X_n/n \in K\cap\UsetA\} \right] \\
		&\leq \sup_{\xi \in \QQ^d \cap K_\delta \cap \UsetA} \varlimsup_{n \rightarrow \infty} \frac{1}{n} \log \rwE_{\orig} \left[ e^{-\beta B(\orig, X_n) - \beta \sum_{i=0}^{n-1} \pote(\T_{X_i} \w, X_{i+1}-X_i)} \one\{X_n = y_n(\xi) \} \right]
	\end{align*}
where $K_\delta = \{\zeta \in \RR^d : \exists \zeta' \in K \text{ with } |\zeta - \zeta'|_1 \le \delta \}$.  If $\beta = \infty$,
\begin{align*}
    \varlimsup_{n \rightarrow \infty} \frac{1}{n} \max_{x \in \DnA{n}\cap nK} \bigl( \freeres{\orig}{x}{n}{\infty} - B(\orig, x) \bigr) &\leq \sup_{\xi \in \QQ^d \cap K_\delta \cap \UsetA} \varlimsup_{n \rightarrow \infty} \frac{1}{n} \bigl( \freeres{\orig}{y_n(\xi)}{n}{\infty}   - B(\orig, y_n(\xi))\bigr).
\end{align*}
\end{lemma}

%
%
%
%
%
\begin{proof}
    We follow the strategy of the proof of Lemma 2.9 in \cite{Ras-Sep-14}.
  	Let $R = \max\{|z|_1 : z \in \rangeA \}$.  Order $\rangeA^{\id}$ as $\{z_1,\dotsc,z_m\}$ and for each of these $z_i$ let $\widehat{z}_i \in \rangeA\setminus\rangeA^{\id}$ be as in Condition \ref{classLCondition}.  If  $\rangeA^{\id}=\varnothing$, then set $m=1$ and take any $\widehat{z}_1 \in \rangeA$. Let $\eps \in \bigl(0, \delta/(4R|\rangeA|)\bigr)$.  Let $j$ be an integer with $j \geq \max\{\frac{1+2\eps}{\eps}|\range|, (2\varepsilon)^{-1}\}$.  Because $(mj)^{-1} \DnA{mj}$ is finite, there exists an integer $b$ such that $y_{kb} = kb\xi$ for all $k \in \ZZ_{\ge0}$ and $\xi \in (mj)^{-1} \DnA{mj}$.  
  
  Next, we construct a path from each $x \in \DnA{n} \cap nK$ to a multiple of a point $\xi(n,x) \in K_\delta \cap (mj)^{-1} \DnA{mj}$.  To do this, start by writing $x = \sum_{z \in \rangeA} a_z z$ where each $a_z \in \ZZ_{\ge0}$ and $\sum_{z \in \rangeA} a_z = n$.  Let $k_n = \lceil \frac{(1+2\eps)n}{mj} \rceil$ and $s_z^{(n)} = \lceil \frac{ja_z}{(1+2\eps)n} \rceil$. Then
  \[
  \frac{a_z}{n} - \frac{s_z^{(n)}}{j} = \frac{a_z}{n} \Bigl(1 - \frac{ns_z^{(n)}}{ja_z}\Bigr) \geq \frac{a_z}{n} \Bigl(1-\frac{1}{1+2\eps}\Bigr) - \frac{1}{j}.
  \]
  Similarly, 
  \[
  \frac{a_z}{n} - \frac{s_z^{(n)}}{j} \leq \frac{a_z}{n} \left(1-\frac{1}{1+2\eps}\right).
  \]
  Summing over $z$, we get 
  \begin{align}\label{auxaux000}
    \frac{\eps}{1+2\eps} \leq 1-\frac{1}{1+2\eps} - \frac{|\range|}{j} \leq 1 - \frac{1}{j} \sum_{z \in \rangeA} s_z^{(n)} \leq 1 - \frac{1}{1+2\eps} < \frac{\delta}{2R\abs{\rangeA}}\le\frac{\delta}{2R}.
  \end{align}
  
  Also,
    \[\abs{n^{-1}a_z-j^{-1}s_z^{(n)}}\le\max\Bigl\{\frac{a_z}n\Bigl(1-\frac1{1+2\e}\Bigr),\frac1j\Bigr\}\le 2\e.\] 
 Thus,
  \begin{align}\label{auxaux001}
  \begin{split}
  \Bigl| \frac{1}{j} \sum_{z \in \rangeA} s_z^{(n)} z - \frac{x}{n} \Bigr|_1 = \Bigl| \frac{1}{j} \sum_{z \in \rangeA} \Bigl( s_z^{(n)} z - \frac{j}{n} a_z z \Bigr) \Bigr|_1 
  \leq R \sum_{z \in \rangeA} |j^{-1} s_z^{(n)} - n^{-1} a_z| \leq 2R|\rangeA|\varepsilon < \frac{\delta}{2}.
  \end{split}
  \end{align}

Define
\begin{align}\label{jxi}
\xi(n,x) = j^{-1} \sum_{z \in \rangeA} s_z^{(n)} z + \sum_{i=1}^m m^{-1}\Bigl(1-j^{-1}\sum_{z \in \rangeA} s_z^{(n)}\Bigr) \widehat{z}_i.
\end{align}
\eqref{auxaux000} and \eqref{auxaux001} imply that $\xi(n,x)\in K_\delta \cap (mj)^{-1} \DnA{mj}$
 and the path that takes $ms_z^{(n)}$ $z$-steps for each $z\in\rangeA$ then $(j-\sum_{z \in \rangeA} s_z^{(n)})$  $\widehat{z}_i$-steps, for each $i\in\{1,\dotsc,m\}$, is an admissible path that takes a total of $mj$ steps to go from $\orig$ to $mj\xi(n,x)$.

Consider next the following admissible path starting at the origin $\orig$. It begins by taking $a_z$ $z$-steps for each $z\in\rangeA$, leading to the point $x$. Next, the path takes $(mk_n s_z^{(n)}-a_z)$ $z$-steps for each $z\in\rangeA$. Then, it proceeds with $(k_n j - \sum_{z \in \rangeA} k_n s_z^{(n)})$ $\widehat z_i$-steps, for each $i\in\{1,\dotsc,m\}$. This entire path is admissible because, first, $mk_n s_z^{(n)} \geq m\cdot\frac{(1+2\eps)n}{mj} \cdot \frac{j a_z}{(1+2\eps)n} = a_z$ and second, as previously observed, $k_n j\ge\sum_{z \in \rangeA} k_n s_z^{(n)}$.

The above path consists of a total of $k_n mj$ steps and ends at 
$k_n mj \xi(n,x)$.
Notably, the last $k_n mj - n$ steps give an admissible path from $x$ to $k_n mj \xi(n,x)$.
For each $i\in\{1,\dotsc,m\}$, the number of $\widehat{z}_i$ steps in the path is at least 
\[
k_n j\Bigl(1 - j^{-1} \sum_{z \in \rangeA} s_z^{(n)} \Bigr) \geq \frac{(1+2\eps)n}{mj} \cdot j \cdot \frac{\eps}{1+2\eps} = nm^{-1}\eps.
\]
Let $\ell_n=\lceil\frac{k_n}b\rceil$ so that $(\ell_n-1)b < k_n \leq \ell_n b$.  Repeating the steps of the path below \eqref{jxi} that goes from $\orig$ to $mj \xi(n,x)$, one gets a path that goes from $k_n mj\xi(n,x)$ to $\ell_n b mj \xi(n,x)$.  This requires repeating each step $\ell_n b - k_n<b$ times.   Thus, the number of steps to go from $x$ to $\ell_n bmj \xi(n,x)$ is 
$r_n = (k_n mj - n)+(\ell_n b - k_n)mj = \ell_n b mj - n$.  Since $b$ is a function solely dependent on $mj$ and independent of $n$,
\begin{align*}
r_n 
\leq \Bigl(\frac{k_n}{b} + 1\Bigr)bmj - n
\leq \Bigl(\frac{(1+2\eps)n}{mj} + 1\Bigr)mj + bmj - n  
= 2\eps n + (b+1)mj \leq 3 \eps n
\end{align*}
for all $n\ge n_0$ with $n_0$  depending on a $mj$ and $\eps$.  Denote the steps we constructed to go from $x$ to $\ell_n b mj \xi(n,x)$ by $(u_1, \ldots, u_{r_n})$ and denote by $\mb{u}(n,x)$ the path that starts at $\orig$ and takes these exact steps.  

We are ready to estimate:  
\begin{align*}
	&\frac{1}{n} \log \rwE_{\orig} \left[ e^{-\beta B(\orig, X_n) - \beta \sum_{i=0}^{n-1} \pote(\T_{X_i} \w, X_{i+1}-X_i)} \one\{X_n/n \in K\cap\UsetA\} \right] \\ 
	&= \frac{1}{n} \log \sum_{x \in \DnA{n} \cap nK} \rwE_{\orig} \left[ e^{-\beta B(\orig, x) - \beta \sum_{i=0}^{n-1} \pote(\T_{X_i} \w, X_{i+1}-X_i)} \one\{X_n = x\} \right] \\
	&\leq \max_{x \in \DnA{n} \cap nK} \frac{1}{n} \log \rwE_{\orig} \left[ e^{-\beta B(\orig, x) - \beta \sum_{i=0}^{n-1} \pote(\T_{X_i} \w, X_{i+1}-X_i)} \one\{X_n = x\} \right] + \frac{C \log n}{n} \\
	&\leq \max_{x \in \DnA{n} \cap nK} \frac{1}{n} \log \rwE_{\orig} \left[ e^{-\beta B(\orig, \ell_n bmj\xi(n,x)) - \beta \sum_{i=0}^{n+r_n-1} \pote(\T_{X_i} \w, X_{i+1}-X_i)} \one\{X_n = x, Z_{n+1:r_n+n} = u_{1:r_n}\} \right] \\
	&\qquad- \frac{r_n}{n} \log \min_{z \in \rangeA} p(z)
	+ \max_{x \in \DnA{n}\cap nK} \frac{\beta B(x, \ell_n bmj\xi(n,x))}{n}\\
	&\qquad+ \max_{x \in \DnA{n} \cap nK} \frac{\beta}{n} \sum_{i=0}^{r_n-1} \pote^+(\T_{x + u_1 + \ldots + u_i} \w,u_{i+1}) + \frac{C \log n}{n} \\ 
	&\leq \max_{x \in \DnA{n} \cap nK} \frac{1}{n} \log \rwE_{\orig} \left[ e^{-\beta B(\orig, \ell_n bmj\xi(n,x)) - \beta \sum_{i=0}^{n+r_n-1} \pote(\T_{X_i} \w, X_{i+1}-X_i)} \one\{X_{\ell_n bmj} = \ell_n b mj \xi(n,x)\} \right] \\
	&\qquad- \frac{r_n}{n} \log \min_{z \in \rangeA} p(z) + \max_{x \in \DnA{n}\cap nK} \frac{\beta B(x, \ell_n bmj\xi(n,x))}{n}\\
	&\qquad+ \max_{x \in \DnA{n} \cap nK} \frac{\beta}{n} \sum_{i=0}^{r_n-1} \pote^+(\T_{x + u_1 + \ldots + u_i} \w,u_{i+1}) + \frac{C \log n}{n}. 
\end{align*}
Above, we used the notation $Z_{i+1}=X_{i+1}-X_i$ and took the convention that an empty sum is $0$. Hence, $x+u_1+\ldots+u_i=x$ when $i=0$.
%

We now explain why all of the quantities on the right-hand side converge to 0 when we send $n\rightarrow \infty$ and then $\eps \rightarrow 0$. 
Since $\xi(n,x)\in K_\delta \cap (mj)^{-1} \DnA{mj}\subset \QQ^d \cap K_\delta \cap \UsetA$ and $y_{\ell_n bmj}(\xi(n,x))=\ell_n bmj\xi(n,x)$,
we have that for $n$ large enough, the first maximum on the right-hand side is bounded above by 
\[
(1+3\eps)\sup_{\xi \in \QQ^d \cap K_\delta \cap \UsetA} \varlimsup_{n \rightarrow \infty} \frac{1}{n} \log \rwE_{\orig} \left[ e^{-\beta B(\orig, X_n) - \beta \sum_{i=0}^{n-1} \pote(\T_{X_i} \w, X_{i+1}-X_i)} \one\{X_n = y_n(\xi) \} \right].
\]
For the second term, use $r_n \leq 3\eps n$.  By the shape theorem for cocycles, \cite[Theorem B.3]{Jan-Ras-18-arxiv}, in the limit, the next maximum is bounded in absolute value 
by $\beta (1+3\eps) 3\eps R  |\mVec(B)|_1$, where $\mVec(B)$ is the random vector which satisfies \eqref{meanVec}.  
The last term goes to $0$.  

For the third maximum, notice that the particular order of the steps in $\mb{u}(n,x)$ does not matter to this point.  For each $i\in\{1,\dotsc,m\}$, the ratio of $z_i$ steps to $\widehat{z}_i$ steps is at most $r_n / (m^{-1}n\eps) \leq 3m$.  So order the steps of $\mb{u}(n,x)$ in this way.  
First, alternate between $\widehat{z}_1$ steps and $z_1$ steps in such a way that there are no more than $3m$ consecutive $z_1$ steps.  Continue this process until we have exhausted all $z_1$ steps and their corresponding $k_n j\Bigl(1 - j^{-1} \sum_{z \in \rangeA} s_z^{(n)} \Bigr)$ $\widehat{z}_1$ steps.   Repeat this procedure for $i=2,\dotsc,m$. It is important to note that if $\widehat{z}_i=\widehat{z}_{i'}$ for some distinct $i$ and $i'$, we only exhaust the corresponding number of $\widehat{z}_i$ steps when we exhaust the $z_i$ steps, leaving the $\widehat{z}_{i'}$ steps to be paired with the $z_{i'}$ steps.

Next, choose an ordering for the set $\rangeA \setminus \rangeA^{\id}$, denoted as $\{z'_1, \ldots, z'_c\}$.  Continue ordering the steps of $\mb{u}(n,x)$ by using all of the $z'_1$ steps first, followed by the $z'_2$ steps, and so on for the remaining steps. Note that this includes  $(m k_n s_{\widehat z_i}^{(n)} - a_{\widehat z_i})$  $\widehat{z}_i$ steps, for each $i\in\{1,\dotsc,m\}$ from the first part of the path from $x$ to $k_n mj\xi(n,x)$, which have not been used in the above part of the procedure, when we exhausted the $z_i$ steps.

Recall that $\T_z$ is the identity map for all $z\in\rangeA^{\id}$. 
Using the above ordering, we may bound
\begin{align*}
\max_{x \in \DnA{n} \cap nK} \frac{1}{n} \sum_{k=0}^{r_n-1} \pote^+(\T_{x + u_1 + \ldots + u_k} \w,u_{k+1})
&\leq \frac{|\rangeA|}{n} \max_{x \in \DnA{n} \cap nK} \max_{y \in x+\mb{u}(n,x)} \max_{z \in \rangeA \setminus \rangeA^{\id}} \sum_{k=0}^{r_n-1} \pote^+(\T_{y+ k z}\w,z)\\
&\quad+\sum_{i=1}^m \frac{3m}n\max_{x \in \DnA{n} \cap nK} \max_{y \in x+\mb{u}(n,x)}\sum_{k=0}^{r_n-1} \pote^+(\T_{y+ k{\widehat z}_i}\w,z_i).
\end{align*}
Since $\pote^+(\w,z) \in \classL_{z,\rangeA}$ for all $z\in\rangeA$ and $\pote^+(\w,z_i)\in\classL_{\widehat z_i,\rangeA}$, for each $i\in\{1,\dotsc,m\}$, the right-hand side converges to $0$ when we take $n\to\infty$ and then $\eps\to\infty$.  

At zero temperature, the equivalent estimate is 
\begin{align*}
    &\frac{1}{n} \max_{x \in \DnA{n} \cap nK} \bigl( \freeres{\orig}{x}{n}{\infty} - B(\orig, x) \bigr) \\
    &\leq \frac{1}{n} \max_{x \in \DnA{n} \cap nK} \bigl( \freeres{\orig}{x}{n}{\infty} - \sum_{i=0}^{r_n-1} \pote(\T_{x+u_1+\ldots+u_i}\w, u_{i+1})  - B(\orig, \ell_n bmj\xi(n,x)) \bigr)  \\
    &\qquad \qquad + \max_{x \in \DnA{n}\cap nK} \frac{ B(x, \ell_n bmj\xi(n,x))}{n} + \max_{x \in \DnA{n} \cap nK} \frac{1}{n} \sum_{i=0}^{r_n-1} \pote^+(\T_{x + u_1 + \ldots + u_i} \w,u_{i+1}) \\
    &\leq \frac{1}{n} \max_{x \in \DnA{n} \cap nK} \bigl( \freeres{\orig}{\ell_n bmj\xi(n,x)}{\ell_n bmj}{\infty}   - B(\orig, \ell_n bmj\xi(n,x)) \bigr)  \\
    &\qquad \qquad + \max_{x \in \DnA{n}\cap nK} \frac{ B(x, \ell_n bmj\xi(n,x))}{n} + \max_{x \in \DnA{n} \cap nK} \frac{1}{n} \sum_{i=0}^{r_n-1} \pote^+(\T_{x + u_1 + \ldots + u_i} \w,u_{i+1}).
\end{align*}
The two terms on the last line converge to 0 after sending $n \rightarrow \infty$ and $\e \rightarrow 0$, as in the positive temperature case.  The term in the second-to-last line is bounded above by
\begin{align*}
    &(1+3\e) \sup_{\xi \in \QQ^d \cap K_\delta \cap \UsetA} \varlimsup_{n \rightarrow \infty} \frac{1}{n} \bigl( \freeres{\orig}{y_n(\xi)}{n}{\infty}   - B(\orig, y_n(\xi))\bigr).
\end{align*}
Take $\e\to0$.
\end{proof}

Recall the restricted-length limiting quenched point-to-point free energy $\ppShapeRes{\beta}$ in \eqref{ppShapeRes}, its restriction $\ppShapeResA{\beta}$ to the face $\sA$, and the upper semicontinuous regularization $\ppShapeResA{\beta, \usc}$.  Define 
\[I_{B}(\xi) = \begin{cases}\beta \mVec(B) \cdot \xi - \beta \ppShapeResA{\beta, \usc}(\xi)&\text{if $\xi \in \UsetA$,}\\ \infty&\text{otherwise.}\end{cases}\]

\begin{theorem}\label{thm:LDP}
	Assume the setting of Theorem \ref{thm:Measures} with $\beta < \infty$. 
    There exists an event $\Omega_B \subset \Omega$ with $\bbP(\Omega_B) = 1$ such that for each $\w \in \Omega_B$, the distribution of $X_n / n$ under $\Qinf{x}{\beta, B, \w}$ satisfies a large deviation principle with rate function $I_B$.
	This means the following bounds hold
	\begin{align}
		&\varlimsup_{n \rightarrow \infty} n^{-1} \log \Qinf{x}{\beta, B, \w}(X_n / n \in K) \leq -\inf_{\zeta \in K} I_{B}(\zeta) \text{ for closed } K \subset \RR^d\quad\text{and} \\
		&\varliminf_{n \rightarrow \infty} n^{-1} \log \Qinf{x}{\beta, B, \w}(X_n / n \in O) \geq -\inf_{\zeta \in O} I_{B}(\zeta) \text{ for open } O \subset \RR^d.\label{LDP-lower}
	\end{align}
\end{theorem}


\begin{proof}
    We follow the strategy of the proof of \cite[Theorem 4.1]{Ras-Sep-14}.
	First, observe that for $y-x\in\RsemiA$, using the cocycle property, we get
	\begin{align}
		\Qinf{x}{\beta, B, \w}(X_n = y) &= \sum_{x_{0:n} \in \PathsPtPRes{x}{y}{n}} p(x_{0:n}) e^{-\beta \sum_{i=0}^{n-1}  B(x_i, x_{i+1}, \w) - \beta \sum_{i=0}^{n-1} \pote(\T_{x_i} \w, x_{i+1}-x_i)} \notag\\
		&= e^{-\beta B(x, y, \w)} \sum_{x_{0:n} \in \PathsPtPRes{x}{y}{n}} p(x_{0:n}) e^{ -\beta \sum_{i=0}^{n-1}  \pote(\T_{x_i} \w, x_{i+1}-x_i)}\notag \\
		&= e^{-\beta B(x, y, \w)} \prtres{x}{y}{n}{\beta, \w}.\label{QeBZ}
	\end{align}

Then using \cite[Theorem B.3]{Jan-Ras-18-arxiv} and Theorem \ref{Thm:ShapeRes}, we have for $\bbP$-almost every $\w$, for any sequence  $x_n\in\DnA{n}+x$, $n\in\ZZ_{>0}$, such that $x_n/n\to\xi\in\ri\UsetA$ as $n\to\infty$, 
\[
\frac{1}{n} \log \Qinf{x}{\beta, B, \w}(X_n = x_n)  = -\frac{1}{n} \beta B(x, x_n, \w) + \frac{1}{n} \log \prtres{x}{x_n}{n}{\beta, \w} \mathop{\longrightarrow}_{n\rightarrow \infty} -\beta \mVec(B) \cdot \xi + \beta \ppShapeResA{\beta, \usc}(\xi) = -I_{B}(\xi).
\]

Let $O \subset \RR^d$ be an open set. Take $\xi \in O\cap\ri\UsetA$ and take $x_n$ as above.  
Then $x_n \in nO$ for all $n$ large enough and
\begin{align*}
	\varliminf_{n\rightarrow \infty} \frac{1}{n} \log \Qinf{x}{\beta, B, \w}(X_n/n \in O) 
	\geq \varliminf_{n\rightarrow \infty} \frac{1}{n} \log \Qinf{x}{\beta, B, \w}(X_n =x_n) 
	= -I_{B}(\xi). 
\end{align*}
By Theorem \ref{Thm:ShapeRes}, $I_B$ is continuous on $\UsetA$ and, therefore, the bound, in fact, holds for all $\xi\in O\cap\UsetA$.
This bound also holds trivially for $\xi \in O \setminus (\UsetA)$. 
Take the supremum over all $\xi \in O$ to get \eqref{LDP-lower}.

Next, Lemma \ref{lem:upperBoundCompact} says that for $\bbP$-almost every $\w$, for any closed set $K\subset \RR^d$ and $\delta > 0$, for any $x\in\ZZ^d$, 
\begin{align}
	\varlimsup_{n\rightarrow \infty} \frac{1}{n} \log \Qinf{x}{\beta, B, \w}\Bigl(\frac{X_n}{n} \in K \Bigr)
	&= \varlimsup_{n\rightarrow \infty} \frac{1}{n} \log \rwE_{\orig} \Bigl[ e^{-\beta B(\orig, X_n) - \beta \sum_{i=0}^{n-1} \pote(\T_{X_i} \w, X_{i+1}-X_i)} \one\Bigl\{\frac{X_n}n \in K\cap\UsetA\Bigr\} \Bigr]\notag\\
\begin{split}
	&\leq \sup_{\xi \in \QQ^d \cap K_\delta \cap \UsetA} \varlimsup_{n \rightarrow \infty} \frac{1}{n} \log \Qinf{x}{\beta, B, \w}(X_n = y_n(\xi)) \\
	&=- \inf_{\xi \in \QQ^d \cap K_\delta \cap \UsetA} I_{B}(\xi) \leq -\inf_{\xi \in K_\delta \cap \UsetA} I_{B}(\xi).
\end{split}\label{QLDP-aux}
\end{align}

Since \eqref{QLDP-aux} holds for all $\delta >0$, $K_\delta \cap \UsetA$ is compact and decreases to $K\cap\UsetA$ as $\delta$ decreases to $0$, and since $I_{B}$ is lower semicontinuous,  we obtain the bound 
\begin{align*} 
	&\varlimsup_{n\rightarrow \infty} \frac{1}{n} \log \Qinf{x}{\beta, B, \w}\Bigl(\frac{X_n}{n} \in K \Bigr) \leq -\lim_{\delta \rightarrow 0} \inf_{\xi \in K_\delta \cap \UsetA} I_{B}(\xi)
	= -\inf_{\xi \in K \cap \UsetA} I_{B}(\xi) = -\inf_{\xi \in K} I_{B}(\xi).\qedhere
\end{align*}
\end{proof}

We are now ready to prove Theorem \ref{thm:Measures}.

\begin{proof}[Proof of Theorem \ref{thm:Measures}]
Parts \eqref{Measures.consPos}, \eqref{Measures.Greens}, \eqref{Measures.consPosRes},  \eqref{Measures.consZero}, \eqref{Measures.trans.pos}, \eqref{Measures.trans.zero}, and \eqref{Measures.Directed} were proved in Lemmas \ref{lem:consistencyPos}, \ref{lem:GreensExpectation}, \ref{lem:consistencyPosRes},  \ref{lem:consistencyZero}, \ref{lm:trans.pos}, \ref{lm:trans.zero}, and Theorem \ref{thm:dir}, respectively. 
Next, we will prove part \eqref{Measures.LLNLimit} in the positive temperature case.

Write the complement $\RR^d \setminus \facetRes{\mVec(B),\sA}$ as a countable union of non-decreasing compact sets $K_j$.  Explicitly, we define
\[
K_j = \{v \in \RR^d : |v|_1 \leq j \text{ and } d(v, \facetRes{\mVec(B),\sA}) \geq j^{-1} \}.
\]
Then in fact $\RR^d \setminus \facetRes{\mVec(B),\sA} = \bigcup_{j=1}^\infty \mathring{K_j}$, where $\mathring{A}$ is the interior of the set $A\subset\RR^d$.  

We now show that for each $j$, $\Qinf{x}{\beta, B, \w}(X_n/n \text{ has no limit points in } \mathring{K_j}) = 1$.  From the large deviation principle in Theorem \ref{thm:LDP}, 
\begin{align*}
	\varlimsup_{n \rightarrow \infty} n^{-1} \log \Qinf{x}{\beta, B, \w}(X_n / n \in K_j) \leq -\inf_{\zeta \in K_j} I_{B}(\zeta).
\end{align*}
The rate function $I_{B}(\zeta) = 0$ precisely when $\zeta \in \facetRes{\mVec(B),\sA}$.  Because $K_j$ is compact, $K_j \subset \RR^d \setminus \facetRes{\mVec(B),\sA}$, and $I_B$ is continuous, 
we see that $-\inf_{\zeta \in K_j} I_{B}(\zeta) < 0$. 
Therefore $\Qinf{x}{\beta, B, \w}(X_n/n \in K_j)$ is summable in $n$, and by  Borel-Cantelli, $\Qinf{x}{\beta, B, \w}(X_n/n \in K_j \text{ for only finitely many } n) = 1$.  Therefore, $X_n/n$ has no limit points in $\mathring{K_j}$, $\Qinf{x}{\beta, B, \w}$-almost surely.  

Then by a union bound in $j$, since $K_j$ are non-decreasing, 
we see $\Qinf{x}{\beta, B, \w}$-almost surely, $X_n/n$ has no limit points in $\bigcup_{j=1}^{\infty} \mathring{K_j} = \RR^d \setminus \facetRes{\mVec(B),\sA}$.  So we conclude that $\Qinf{x}{\beta, B, \w}(X_n \text{ is directed into } \facetRes{\mVec(B),\sA}) = 1$.


Lastly, we prove part \eqref{Measures.LLNLimit} in the zero temperature case.  For any sequence $x_n \in \DnA{n}+x$, $n\in\ZZ_{>0}$, such that $x_n/n\to\xi\in\ri\UsetA$ as $n\to\infty$, the cocycle and recovery properties imply that $n^{-1} B(x,x_n) \geq n^{-1}\freeres{x}{x_{n}}{n}{\infty}$.  
The shape theorem \cite[Theorem B.3]{Jan-Ras-18-arxiv} and Theorem \ref{Thm:ShapeRes} give that $\mVec(B) \cdot \xi \geq \ppShapeResA{\infty, \usc}(\xi)$, for all $\xi \in \ri\UsetA$.  The continuity of $\ppShapeResA{\infty, \usc}$ extends this inequality to all  $\xi \in \UsetA$.   

Consider $X_n$ under $\Qinf{x}{\infty, B,\tie, \w}$. Using 
$n^{-1}B(x, X_{n}) = n^{-1} \freeres{x}{X_{n}}{n}{\infty}$, the shape theorem \cite[Theorem B.3]{Jan-Ras-18-arxiv}, and Theorem \ref{Thm:ShapeRes}, we get $\mVec(B) \cdot \zeta \leq \ppShapeResA{\infty, \usc}(\zeta)$, for any limit point $\zeta$ of $X_n/n$. 
Combined with the inequality above, we see $\mVec(B) \cdot \zeta = \ppShapeResA{\infty, \usc}(\zeta)$, which then implies that $\zeta \in\facetResZero{\mVec(B),\sA}$. 
\end{proof}

\section{Proof of Theorem \ref{thm:CocycleMeasuresFace}}\label{sec:MainThmProofs}

   Apply Theorem \ref{thm:Cocycles} to obtain, for each $\sA \in \faces$, $(\beta, \m) \in \DctbA$, and $x \in \ZZ^d$, the $L^1(\bbPhat)$, $\That$-covariant, $\beta$-recovering cocycle $\B{\sA, \beta, \m}(x, y, \omegahat)$ on $\sA$ with mean $\bbEhat[\mVec(\B{\sA, \beta,\m})] = m$. 
Next, apply Theorem \ref{thm:Measures} to obtain the semi-infinite path measures $\Qinf{x}{\sA, \beta, \B{\sA,\beta, \m}, \omegahat}$ for triples $(\sA,\beta,\m)$ with $\beta<\infty$.
For the triples $(\sA,\infty,\m)$, Lemma \ref{lm:rec-gen} verifies that Condition \ref{BCondition} holds for $\B{\sA, \infty,\m}$. Therefore, Theorem \ref{thm:Measures} produces the measures $\Qinf{x}{\sA, \infty, \B{\sA,\infty, \m},\tie, \omegahat}$, where $\tie$ is the tie-breaker mentioned in part \eqref{Measures.trans.zero} of the theorem. We will abbreviate this family of measures as $\{\Qinf{x}{\sA, \beta, \m, \omegahat}:\sA\in\faces,(\beta,m)\in\DctbA,x\in\ZZ^d\}$. Let $\Omegahat_0=\Omegahat_{\coc}\cap\bigcap_{\sA\in\faces,(\beta,m)\in\DctbA}\Omegahat_{\B{\sA,\beta, \m}}$, where $\Omegahat_{\coc}$ is the full-probability event from Theorem \ref{thm:Cocycles}\eqref{Cocycles.c} and $\Omegahat_{\B{\sA,\beta, \m}}$ are the full-probability events from Theorem \ref{thm:Measures}. Then the measures $\Qinf{x}{\sA, \beta, \m, \omegahat}$ satisfy the consistency properties in parts \eqref{CocycleMeasuresFace.consPos}, \eqref{CocycleMeasuresFace.Greens}, \eqref{CocycleMeasuresFace.consPosRes}, and \eqref{CocycleMeasuresFace.consZero}, for all $\omegahat\in\Omegahat_0$. 

    
    By parts \eqref{Measures.trans.pos}, \eqref{Measures.trans.zero}, and \eqref{Measures.Directed} of Theorem \ref{thm:Measures}, we have that under the conditions of Theorem \ref{thm:CocycleMeasuresFace}\eqref{CocycleMeasuresFace.trans}, for all $x \in \ZZ^d$ and $\omegahat \in \Omegahat_0$, 
        \[\Qinf{x}{\sA, \beta, \m, \omegahat}\{\abs{X_n}_1\to\infty\}=\Qinf{x}{\sA, \beta, \m, \omegahat}\{X_{0:\infty} \text{ is directed into } \facetUnr{\mVec(\B{\sA, \beta,\m})}\} = 1.\]
        (Recall that Lemma \ref{lm:rec-gen} verifies that Condition \ref{BCondition} holds for $\B{\sA, \infty,\m}$.)
        By Theorem \ref{thm:Measures}\eqref{Measures.LLNLimit}, for all $x\in\ZZ^d$ and $\omegahat\in\Omegahat_0$,
         \[\Qinf{x}{\sA, \beta, \m, \omegahat}\{\text{all limit points of } X_n/n \text{ are contained in } \facetRes{\mVec(\B{\sA, \beta,\m})}\} = 1.\]
        By Theorem \ref{thm:Measures}\eqref{Cocycles.b} and \eqref{mIsExtreme}, $\bbEhat[\mVec(\B{\sA, \beta,\m})] = \m \in \ext(\subspaceA \cap \superDiffUnrA{\beta, \usc}(\xi))$ for some $\xi \in \ri \sA$. By Lemma \ref{lem:E[m(B)]}, $\mVec(\B{\sA, \beta,\m}) \in \subspaceA \cap \superDiffUnrA{\beta, \usc}(\xi)$, $\bbPhat$-almost surely.  It must be then that $\m = \bbEhat[\mVec(\B{\sA, \beta,\m})] = \mVec(\B{\sA, \beta,\m})$, $\bbPhat$-almost surely by the definition of an extreme point.
      So there exists a $\That$-invariant event of $\bbPhat$-probability one, on which $\facetRes{\mVec(\B{\sA, \beta,\m})} = \facetRes{\m,\sA}$ and $\facetUnr{\mVec(\B{\sA, \beta,\m})} = \facetUnr{\m,\sA}$. Let $\Omegahat_{\dir}$ be the intersection of this event with $\Omegahat_0$. All the claims of the theorem are now verified, but with the quenched measures being measurable functions of $\omegahat$.

    It remains to construct the family of measures on the original space $\Omega$, using a standard argument in measure theory. To this end, recall that $\pi_{\Omega}$ is the projection from $\Omegahat$ to $\Omega$.  By \cite[Example 10.4.11]{Bog-07}, there exist a $\T$-invariant Borel set $\OmegaReg \subset \Omega$ and a family of regular conditional distributions $\mu_{\w}(\cdot) = \bbPhat(\cdot \given \pi_{\Omega}^{-1}(\w))$ such that $\bbP(\OmegaReg) = 1$ and $\mu_{\w}(\pi_{\Omega}^{-1}(\w)) = 1$ for all $\w \in \OmegaReg$.  For $\bbP$-a.e.\ $\w \in \Omega$, $\mu_\w(\widehat{\Omega}_{\dir}) = 1$ since
    \[
    \int \mu_\w(\widehat{\Omega}_{\dir}) \,\bbP(d\w) = \bbPhat(\widehat{\Omega}_{\dir}) = 1.
    \]
    Define the event $\OmegaDir = \OmegaReg \cap \{\w \in \Omega : \mu_\w(\widehat{\Omega}_0) = 1\}$. Then $\bbP(\OmegaDir) = 1$.  For each $\w \in \OmegaDir$, $\mu_\w(\pi_{\Omega}^{-1}(\w) \cap \widehat{\Omega}_{\dir}) = 1$, so there exists $\omegahat \in \widehat{\Omega}_0$ such that $\pi_{\Omega}(\omegahat) = \w$.  For each $\w \in \OmegaDir$, define $\Qinf{x}{\sA, \beta, \m, \w} = \Qinf{x}{\sA, \beta, \m, \omegahat}$. This family satisfies all the desired claims.\hfill\qed

\appendix

\section{Basic convex analysis and linear algebra facts}\label{app:conv}

In this short appendix, we recall some convex analysis facts and prove some lemmas that are of use to us.

For a convex set $K$, $A$ is called a face of $K$  if for all $\xi, \eta \in K$ and $t \in (0,1)$, $t\xi + (1-t)\eta \in A$ implies that $\xi, \eta \in A$.  The intersection of faces is a face.  $K$ itself is a face, and, by   \cite[Corollary 18.1.3]{Roc-70}, all other faces are contained in the relative boundary of $K$.  Extreme points are the zero-dimensional faces. If $\xi \in A$ can be written as a convex combination of $\eta, \zeta \in K$ then $\eta, \zeta \in A$.  The relative interiors of the non-empty faces of $K$ form a partition of $K$ by \cite[Theorem 18.2]{Roc-70}.  Thus, every $\xi \in K$ has a unique face $K_{\xi}$ such that $\xi \in \ri K_{\xi}$. If $K$ is in the convex set (respectively convex cone) generated by a set $R$, then by \cite[Theorem 18.3]{Roc-70} a face $A$ of $K$ is in the convex set (convex cone) generated by $R \cap A$.

\begin{lemma}[Euler's homogeneous concave function 
theorem]\label{lem:concaveDualHomogeneous}
    Let $\sX$ and $\sY$ be real vector spaces with a bilinear function $\langle \cdot, \cdot \rangle : \sX \times \sY \rightarrow \R$.  Let $f : \sX \rightarrow [-\infty, \infty)$ be a proper, concave, positively 1-homogeneous function.  The superdifferential at $x \in \sX$ is
    \[
        \partial f(x) = \{ y \in \sY : \forall u \in \sX,   f(u)  \le f(x) + \langle u-x, y \rangle\}.
    \]
    For $x \in \sX$ and $y \in \sY$, 
    if $y \in \partial f(x)$ then $f(x) = \langle x, y \rangle$.
    \end{lemma}
    \begin{proof}
    Let $x \in \sX$ and $y \in \sY$ such that $y \in \partial f(x)$.  Let $\lambda > 0$ and $u = \lambda x$.  By homogeneity,
    \[
    (\lambda - 1) f(x) = f(u) - f(x) \leq \langle u, y \rangle - \langle x, y \rangle = (\lambda - 1) \langle x, y \rangle.
    \]
    With $\lambda < 1$, this implies $f(x) \geq \langle x, y \rangle$.  With $\lambda > 1$, this implies $f(x) \leq \langle x, y \rangle$.  Therefore, $f(x) = \langle x, y \rangle.$
    \end{proof}
    
\begin{lemma}\label{lm:z.uhat=1}
Let $d\in\Z_{>0}$ and $\range\subset\Z^d$. The following are equivalent:
\begin{enumerate}[label=\rm(\alph{*}), ref=\rm\alph{*}] \itemsep=2pt 
    \item\label{same length} For any $x,y\in\Z^d$ with $y-x\in\Rgroup$, all paths in $\PathsPtPUnr{x}{y}$ have the same length.
    \item\label{z.uhat=1} There exists a vector $\uhat\in\R^d$ such that $\uhat\cdot z=1$ for all $z\in\range$.
\end{enumerate}
\end{lemma}

\begin{proof}
Suppose \eqref{z.uhat=1} holds. Then any path $x_{0:k}\in\PathsPtPUnr{x}{y}$ satisfies $(y-x)\cdot\uhat = \sum_{i=0}^{k-1} (x_{i+1}-x_i)\cdot\uhat = k$.  Thus, \eqref{same length} holds.

Now suppose \eqref{same length} holds. Let $\mathbb V$ denote the linear span of $\range$.
Let $k\in[1,d]$ be the dimension of this vector space and let $z_1,\dotsc,z_k\in\range$ be a basis for it. Augment this set to a basis $\{z_1,\dotsc,z_k,z_{k+1},\dotsc,z_d\}$ of $\R^d$, where $z_{k+1},\dotsc,z_d$ are also integer vectors. Let $A$ be the unique invertible linear transformation such that $A z_i=e_i$ for $1\le i\le d$.
 In the standard basis, the matrix of $A$ is the inverse of the matrix $B=[z_1,\dotsc,z_d]$; hence this matrix has rational entries. 
 
 Take $z\in\range$. There is a unique set of numbers $\{a_1,\dotsc,a_k\}$ such that $z=\sum_{i=1}^k a_i z_i$. Applying $A$ gives $Az=\sum_{i=1}^k a_i e_i$. Since the left-hand side has rational coordinates, we get that $a_i$ is rational for  $1\le i\le k$.  Let $n\in\Z_{>0}$ be such that $na_i\in\Z$ for all $i\le k$. 

Since $nz+\sum_{i=1}^k na_i^- z_i=\sum_{i=1}^k na^+_i z_i$, \eqref{same length} implies that $n+\sum_{i=1}^k na_i^-=\sum_{i=1}^k na_i^+$, which implies that $\sum_{i=1}^k a_i=1$. Consequently, we have $z-z_1=\sum_{i=2}^k a_i(z_i-z_1)$. We have thus shown that  the linear span of $\{z-z_1:z\in\range\}$ has dimension at most $k-1$. As this is a subspace of $\mathbb V$ and $\mathbb V$ has dimension $k$, there exists a vector $\uhat'\in\mathbb V\setminus\{\orig\}$ that is orthogonal to $z-z_1$ for all $z\in\range$. This implies that $\uhat'\cdot z=\uhat'\cdot z'$ for all $z,z'\in\range$. Denote this common number by $c$. If $c=0$, then $\uhat'$ is perpendicular to all $z\in\range$ and is hence perpendicular to $\mathbb V$. However, then $\uhat'$ is perpendicular to itself, and thus $\uhat'=\orig$. This contradicts the choice of $\uhat'$ and proves that $c\ne0$.  Taking $\uhat=c^{-1}\uhat'$ satisfies \eqref{z.uhat=1}.
\end{proof}

Let $\range\subset\ZZ^d$ and $\cone = \Bigl\{\sum_{z \in \range} b_z z : b_z \in \Rnonneg\Bigr\}$.
Fix a face $\sA$ of $\cone$ and let $\stRgroupA$ be the subgroup of $\Z^{d+1}$ generated by $\{\langle z,1\rangle:z\in\rangeA\}$. Let $\RgroupA^{(0)}$ be the group generated by $\{z-z':z,z'\in\rangeA\}$. Take $z_0\in\rangeA$ and for $j\in\Z$ let $\RgroupA^{(j)}=jz_0+\RgroupA^{(0)}$. Note that $\RgroupA^{(j)}$ does not depend on the choice of $z_0$.

\begin{lemma}\label{lm:stRgroupA}
We have
\[\stRgroupA=\bigcup_{j\in\Z}\bigl\{\langle x,j\rangle:x\in\RgroupA^{(j)}\bigr\}.\]
\end{lemma}

\begin{proof}
    If $x\in\RgroupA^{(0)}$, then $x=\sum_{i=1}^k(z_i-z_i')$ for some $z_1,\dotsc,z_k,z_1',\dotsc,z_k'\in\rangeA$ and,  consequently, 
    \[\langle x,0\rangle=\sum_{i=1}^k(\langle z_i,1\rangle-\langle z_i',1\rangle)\in\stRgroupA.\]
    If $j\in\Z$, then 
    \[\langle jz_0+x,j\rangle=\langle x,0\rangle+j\langle z_0,1\rangle\in\stRgroupA.\]
    For the other direction, take $x\in\ZZ^d$ and $j\in\ZZ$ such that $\langle x,j\rangle\in\stRgroupA$. Then 
        \[\langle x,j\rangle=\sum_{i=1}^k\langle z_i,1\rangle-\sum_{i=1}^{\ell}\langle z_i',1\rangle,\]
        for some $z_1,\dotsc,z_k,z_1',\dotsc,z'_\ell\in\rangeA$ and $k,\ell\in\Z_{\ge0}$ with $k-\ell=j$. 
        From this, we get
            \[x=\sum_{i=1}^k z_i+\ell z_0-kz_0-\sum_{i=1}^{\ell}z_i'+jz_0\in\RgroupA^{(j)}.\qedhere\]
\end{proof}

\begin{proof}[Proof of \eqref{comp:c}]
        Note that $\overline\m\in\partial\overline\Lambda^{\beta,\usc}_{\overline\sA}(\langle\xi,1\rangle)$ is equivalent to having
    	\[t\overline\Lambda^{\beta,\usc}_{\overline\sA}(\langle\zeta,1\rangle)-\overline\Lambda^{\beta,\usc}_{\overline\sA}(\langle\xi,1\rangle)\le\m\cdot(t\zeta-\xi)+c(t-1)\]
    	for all $\zeta\in\Uset'$ and $t>0$. Rearranging and reverting back to restricted-length gives
        \[t \bigl(\ppShapeResU{\beta,\usc}(\zeta)-\m\cdot\zeta-c\bigr) \leq \ppShapeResU{\beta,\usc}(\xi) -m\cdot\xi-c.\]
        Taking $t\to0$ and $t\to\infty$ gives 
        \[ 
        \ppShapeResU{\beta,\usc}(\zeta) - \m\cdot\zeta\le c \leq \ppShapeResU{\beta,\usc}(\xi) - \m\cdot\xi,
        \]
        for all $\zeta\in\Uset'$. This gives 
        \[c=\ppShapeResU{\beta,\usc}(\xi) - \m\cdot\xi\]
        and implies that $\m$ is in the superdifferential at $\xi$ of the concave function that is equal to $\Lambda^{\beta,\usc}_{\sA}$ on $\Uset'$ and is set to $-\infty$ outside $\Uset'$.
\end{proof}

We close this section the proof of an observation made just prior to Theorem \ref{Thm:ShapeUnr}.

\begin{lemma}\label{lem:pathrem}
Fix $x \in \cone$ and let $\sA$ be the unique face of $\cone$ for which $x \in\ri \sA$. Let $x_{0:n}$  be any path with $x_0 = \orig$, $x_n = x$, and $z_i=x_{i}-x_{i-1}\in\range$ for $i =1,\dots, n$. Then $x_{i}\in \sA$ for all $i=0,\dots,n$.
\end{lemma}
\begin{proof}
    We begin by noting that $\orig \in \sA$. If $\cone \neq \bbR^d$, then $\cone$ is polyhedral (\cite[Theorem 19.1]{Roc-70}), i.e.~equal to the intersection of finitely many closed half-spaces whose boundary hyperplanes pass through the origin. The facets of such a set are obtained by intersecting with these hyperplanes and so in particular $\orig\in \sA$.
    
    Next, denote by $S$ the set of steps $x_{i}-x_{i-1}= z_i\in \range$, $i=1,\dots, n$ used in this path. By definition, $x$ lies in the relative interior of the convex hull of $n S$, which is a convex set in $\sC$. Because $x$ lies in the relative interior of the face $\sA$, it follows from \cite[Theorem 18.1]{Roc-70} that the convex hull of $n S$ is a subset of $\sA$. For each $i=1,\dots, n$, $x_i$ is a convex combination of $\orig$ and points in $n S$, so $x_i\in\sA$ for all such $i$.
\end{proof}

\section{Gibbs consistency}\label{app:consistency}

In this appendix, we discuss the Gibbs consistency of the various polymer measures to place our results within the framework of Gibbs specifications and solutions to the Dobrushin-Landford-Ruelle (DLR) equations.

As stated in the next lemma, the restricted-length finite path measures are consistent in the sense of conditioning,
indicating that this family forms a Gibbs specification. 
	See Section 2.4 in \cite{Jan-Ras-18-arxiv}.  
	
	\begin{lemma}\label{lem:consistencyFinite}
		For $j, k, n \in \ZZ_{\ge0}$ with $j \leq k \leq n$, for $u, v, x, y \in \ZZ^d$ such that $x-u \in \Dn{j}$, $y-x \in \Dn{k-j}$, and $v-y \in \Dn{n-k}$, and for $\beta \in (0,\infty)$, $\omega \in \Omega$, $x_{0:j} \in \PathsPtPResKM{u}{x}{0}{j}$, $x_{j:k} \in \PathsPtPResKM{x}{y}{j}{k}$, and $x_{k:n} \in \PathsPtPResKM{y}{v}{k}{n}$,
		\begin{align*}
			\Qres{u}{v}{n}{\beta, \omega}(X_{0:n} = x_{0:n} \given X_{0:j} = x_{0:j}, X_{k:n} = x_{k:n}) &= \Qres{u}{v}{n}{\beta, \omega}(X_{j:k} = x_{j:k} \given X_j = x, X_k = y) \\
			&= \Qres{x}{y}{k-j}{\beta, \w}(X_{0:k-j} = x_{j:k}).
		\end{align*} 		
	\end{lemma}
\begin{proof}
		Let $j, k, n, u, x, y, v, \beta, \omega$ be as in the statement. Let $x_{0:j} \in \PathsPtPResKM{u}{x}{0}{j}$, $x_{j:k} \in \PathsPtPResKM{x}{y}{j}{k}$, and $x_{k:n} \in \PathsPtPResKM{y}{v}{k}{n}$. 
		Then
		\begin{align*}
			&\Qres{u}{v}{n}{\beta, \omega}(X_{0:n} = x_{0:n} \given X_{0:j} = x_{0:j}, X_{k:n} = x_{k:n}) \\
			&\qquad= \frac{\Qres{u}{v}{n}{\beta, \omega}(X_{0:n} = x_{0:n})}{\Qres{u}{v}{n}{\beta, \omega}(X_{0:j} = x_{0:j}, X_{k:n} = x_{k:n})} \\
			&\qquad= \frac{\prod_{i=0}^{n-1} p(x_{i+1}-x_i) e^{-\beta \pote(T_{x_i} \w, x_{i+1}-x_i)}}{\prod_{i=0}^{j-1} p(x_{i+1}-x_i) e^{-\beta \pote(T_{x_i} \w, x_{i+1}-x_i)} \prtres{x}{y}{k-j}{\beta, \w}  \prod_{i=k}^{n-1} p(x_{i+1}-x_i) e^{-\beta \pote(T_{x_i} \w, x_{i+1}-x_i)}} \\
			&\qquad= \frac{\prod_{i=j}^{k-1} p(x_{i+1}-x_i) e^{-\beta \pote(T_{x_i} \w, x_{i+1}-x_i)}}{\prtres{x}{y}{k-j}{\beta, \w}} \\
			&\qquad=\Qres{x}{y}{k-j}{\beta, \w}(X_{0:k-j} = x_{j:k}).
		\end{align*}
	
	Similarly,
	\begin{align*}
		\Qres{u}{v}{n}{\beta, \omega}(X_{j:k} = x_{j:k} \given X_j = x, X_k = y) 
		&=\Qres{x}{y}{k-j}{\beta, \w}(X_{0:k-j} = x_{j:k}).\qedhere
	\end{align*}
	\end{proof}	
 
    Let $\Uset$ be the convex hull of $\range$, with $\ri\Uset$ denoting its relative interior.  The next result shows that the unrestricted-length finite path measures are also Gibbs consistent if $\orig \not\in\Uset$. 
    
    \begin{lemma}\label{lem:consistencyFiniteUnr}
        Assume $\orig\not\in\Uset$. Let $\beta \in (0,\infty)$, $\w \in \Omega$, and $u,v,x,y \in \ZZ^d$ such that $x-u, y-x, v-y \in \Rsemi$.  Let $x_{0:n} \in \PathsPtPKilled{u}{v}$, such that $x_j = x$ and $x_k = y$ for some $0 \leq j \leq k \leq n$. Then,
        \begin{align*}
            \Qunr{u}{v}{\beta,\w}(X_{0:\stoppt{v}} = x_{0:n} \given X_{0:\stoppt{x}} = x_{0:j}, X_{\stoppt{y}:\stoppt{v}} = x_{k:n}) &= \Qunr{u}{v}{\beta,\w}(X_{\stoppt{x}:\stoppt{y}} = x_{j:k} \given \stoppt{x} \leq \stoppt{y} < \infty)\\
            &= \Qunr{x}{y}{\beta,\w}(X_{0:\stoppt{y}} = x_{j:k}).
        \end{align*}
    \end{lemma}
    \begin{proof}
        Assume $\orig \not\in\Uset$ and let $\beta \in (0,\infty)$, $\w \in \Omega$, and $u,v,x,y \in \ZZ^d$ such that $x-u, y-x, v-y \in \Rsemi$.  Let $x_{0:n} \in \PathsPtPKilled{u}{v}$ such that $x_j = x$ and $x_k = y$. Since there are no loops and $y-x \in \Rsemi$, it must be that $0 \leq j \leq k \leq n$.  Furthermore, on the event $\{\stoppt{x} < \infty, \stoppt{y}<\infty\}$, it must be that $\stoppt{x} \leq \stoppt{y} \leq \stoppt{v}$, for otherwise we can construct a loop. Then,
        \begin{align*}
            &\Qunr{u}{v}{\beta,\w}(X_{0:\stoppt{v}} = x_{0:n} \given X_{0:\stoppt{x}} = x_{0:j}, X_{\stoppt{y}:\stoppt{v}} = x_{k:n}) \\
            &\qquad= \frac{\Qunr{u}{v}{\beta,\w}(X_{0:\stoppt{v}} = x_{0:n})}{\Qunr{u}{v}{\beta,\w}(X_{0:\stoppt{x}} = x_{0:j}, X_{\stoppt{y}:\stoppt{v}} = x_{k:n})} \\
            &\qquad= \frac{\prod_{i=0}^{n-1} p(x_{i+1}-x_i) e^{-\beta\pote(\T_{x_i}\w,x_{i+1}-x_i)}}{\prod_{i=0}^{j-1} p(x_{i+1}-x_i) e^{-\beta\pote(\T_{x_i}\w,x_{i+1}-x_i)} \prtunr{x}{y}{\beta,\w}  \prod_{i=k}^{n-1} p(x_{i+1}-x_i) e^{-\beta\pote(\T_{x_i}\w,x_{i+1}-x_i)}} \\
            &\qquad= \frac{\prod_{i=j}^{k-1} p(x_{i+1}-x_i) e^{-\beta\pote(\T_{x_i}\w,x_{i+1}-x_i)}}{\prtunr{x}{y}{\beta,\w} } \\
            &\qquad= \Qunr{x}{y}{\beta,\w}(X_{0:\stoppt{y}} = x_{j:k}).
        \end{align*}

        Similarly, 
        \begin{align*}
            &\Qunr{u}{v}{\beta,\w}(X_{\stoppt{x}:\stoppt{y}} = x_{j:k} \given \stoppt{x} \leq \stoppt{y}<\infty) \\
            &\qquad= \frac{\Qunr{u}{v}{\beta,\w}(\stoppt{x}\leq\stoppt{y}<\infty, X_{\stoppt{x}:\stoppt{y}} = x_{j:k})}{\Qunr{u}{v}{\beta,\w}(\stoppt{x} \leq \stoppt{y}<\infty)} \\
            &\qquad= \frac{\prtunr{u}{x}{\beta,\w} \prod_{i=j}^{k-1} p(x_{i+1}-x_i) e^{-\beta\pote(\T_{x_i}\w, x_{i+1}-x_i)} \prtunr{y}{v}{\beta,\w}}{\prtunr{u}{x}{\beta,\w} \prtunr{x}{y}{\beta,\w} \prtunr{y}{v}{\beta,\w}}\\
            &\qquad= \frac{\prod_{i=j}^{k-1} p(x_{i+1}-x_i) e^{-\beta\pote(\T_{x_i}\w, x_{i+1}-x_i)}}{ \prtunr{x}{y}{\beta,\w} }\\
            &\qquad= \Qunr{x}{y}{\beta,\w}(X_{0:\stoppt{y}} = x_{j:k}).\qedhere
        \end{align*}
    \end{proof}
  
    The measures $\Qunr{u}{v}{\beta, \w}$  are not consistent in general if $\orig \in \Uset$, but they are asymptotically consistent as $|v|_1 \rightarrow \infty$.  
    To see this, we introduce some more notation and definitions.  
	Let $x-u, y-x, v-y \in \Rsemi$ with $y \neq v$. 
		  Define the partition functions
		\begin{align*}
		&\prtunr{u}{x}{\beta, \w}(\stoppt{x} \le \min\{\stoppt{y}, \stoppt{v}\})= \rwE_u\Bigl[ e^{- \beta\sum_{i=0}^{\stoppt{x}-1} \pote(\T_{X_i} \omega, X_{i+1}-X_i)}\one_{\{\stoppt{x}<\infty, \stoppt{x} \leq \min\{\stoppt{y}, \stoppt{v}\}\}}\Bigr] \text{ and } \\ 
		&\prtunr{x}{y}{\beta, \w}(\stoppt{y} \le \stoppt{v}) = \rwE_x\Bigl[ e^{- \beta\sum_{i=0}^{\stoppt{y}-1} \pote(\T_{X_i} \omega, X_{i+1}-X_i)}\one_{\{\stoppt{y}<\infty, \stoppt{y} \leq \stoppt{v}\}}\Bigr].
		\end{align*}
		Take $j \leq k$ in $\ZZ$ and $x_{j:k} \in \PathsPtPResKM{x}{y}{j}{k}$ such that $x_i\not\in\{y,v\}$ for all integers $i\in[j,k)$ and $x_k=y$. Then,
		\begin{align*}
			&\Qunr{u}{v}{\beta, \w}(X_{\stoppt{x}:\stoppt{y}}=x_{j:k},\stoppt{x}=j,\stoppt{y}=k) = \frac{\prtunr{u}{x}{\beta, \w}(\stoppt{x} \le \min\{\stoppt{y}, \stoppt{v}\}) e^{-\beta \sum_{i=j}^{k-1} \pote(\T_{x_i} \w, x_{i+1}-x_i) } \prtunr{y}{v}{\beta, \w}}{\prtunr{u}{v}{\beta, \w}} 
		\end{align*}
		whereas,
		\begin{align*}
			 &\Qunr{u}{v}{\beta, \w}(\stoppt{x} \le \stoppt{y} \le \stoppt{v}) \cdot \Qunr{x}{y}{\beta, \w}(X_{0:k-j}=x_{j:k}) \\
			 &\qquad= \frac{\prtunr{u}{x}{\beta, \w}(\stoppt{x} \le \min\{\stoppt{y}, \stoppt{v}\}) \prtunr{x}{y}{\beta, \w}(\stoppt{y} \le \stoppt{v}) \prtunr{y}{v}{\beta, \w}}{ \prtunr{u}{v}{\beta, \w}} \cdot \frac{e^{-\beta \sum_{i=j}^{k-1}  \pote(\T_{x_i}\w, x_{i+1}-x_i)}}{ \prtunr{x}{y}{\beta, \w}}.
		\end{align*}
	
		Thus, we do not get exact consistency, since $ \prtunr{x}{y}{\beta, \w}(\stoppt{y} \le \stoppt{v}) / \prtunr{x}{y}{\beta, \w}$ is not equal to 1 until we take $|v|_1 \rightarrow \infty$.  When $\orig\not\in\Uset$, this issue is resolved since when $v-y \in \Rsemi$ having $\stoppt{y}<\infty$ implies $\stoppt{y} \leq \stoppt{v}$.  

    It is shown in Theorem \ref{thm:CocycleMeasuresFace} above that there exist measures on semi-infinite paths that are consistent with the point-to-point measures in the sense that 
			\[\Qinf{u}{\beta, \w}(X_{\stoppt{x}:\stoppt{y}} = x_{j:k} \given \stoppt{x} \leq \stoppt{y} <\infty) = \Qunr{x}{y}{\beta, \w}(X_{0:\stoppt{y}} = x_{j:k}).\]
			
At zero temperature, the consistency properties become the following facts about geodesics. Given a geodesic $x_{0:n}$, for any integers $0\le j\le k\le n$, $x_{j:k}$ optimizes the passage time among all paths in $\PathsPtPRes{x_j}{x_k}{k-j}$ in the restricted-length case, and among all paths in $\PathsPtPUnr{x_j}{x_k}$ in the unrestricted-length case. The question about the existence of semi-infinite polymer measures becomes one of the existence of restricted-length or, respectively, unrestricted-length
 \emph{semi-infinite geodesics}. These are semi-infinite admissible paths $x_{0:\infty}$ with the property that for all integers $0 \leq j < k$, $x_{j:k}$ is a restricted-length or, respectively, an unrestricted-length geodesic from $x_j$ to $x_k$.  

    \section{Shape theorems}\label{sec:shapeThms}
In this appendix, we prove the shape theorems which were stated in Section \ref{sub:FE}. These results play a key role in the body of the paper when we prove the duality between the Busemann function mean vector and the direction of the associated Gibbs measure.

\begin{proof}[Proof of Theorem \ref{Thm:ShapeUnr}]
    The finiteness of $\ppShapeUnrA{\beta}$ comes from \cite[Theorem 3.10]{Jan-Nur-Ras-22}, after adjusting their proof as in Lemma \ref{lem:upperBoundCompact} to accommodate our weaker Condition \ref{classLCondition}. 
    As a finite concave function, $\ppShapeUnrA{\beta}$ is continuous on the convex open set $\ri\sA$. Lower semicontinuity implies that $\ppShapeUnrA{\beta}$ is bounded
     below (with a bound that can depend on $\w$), uniformly over any bounded subset of $\sA$.
    Then \cite[Theorem 10.3]{Roc-70} implies that  $\ppShapeUnrA{\beta}$ has a unique continuous
extension from the relative interior to the whole of $\sA$. 
The general argument on page 726 of \cite{Ras-Sep-14} shows that this extension agrees with the upper semicontinuous regularization. (In \cite{Ras-Sep-14}, the argument is made for functions on $\sU$, but it works word for word for functions on $\sA$.) 

    
    The second inequality in \eqref{shape-unr} is implied by \eqref{shape-UB}. 
	The argument for both the first inequality in \eqref{shape-unr} and for \eqref{shape-UB} follows the proof of \cite[Theorem 3.10]{Jan-Nur-Ras-22}, with minor modifications that we will highlight. 
	Consider the following two cases. In Case 1, assume that with positive probability there exists an $\e>0$ and a sequence $x_n\in\RsemiA$ such that $\abs{x_n}_1\to\infty$ and 
	$\freeunr{\orig}{x_n}{\beta} - \ppShapeUnrA{\beta, \usc}(x_n)\ge\e\abs{x_n}_1$ for all $n$. In Case 2, assume that with positive probability there exists a $\xi\in\sA$ with $\abs{\xi}_1=1$ and a sequence $x_n\in\RsemiA$ such that $\abs{x_n}_1\to\infty$, $x_n/\abs{x_n}_1\to\xi$, and 
	\begin{align}\label{shape-aux1}
	\abs{x_n}_1^{-1}\freeunr{\orig}{x_n}{\beta} - \ppShapeUnr{\beta}(\xi)\le-\e\quad\text{for all $n$.}
	\end{align}
    In the first case,  
	we can follow the approach in \cite{Jan-Nur-Ras-22} and extract a subsequence (which we denote as $x_n$ again) such that $x_n/\abs{x_n}_1$ converges to some $\xi\in\sA$. Since $\ppShapeUnrA{\beta,\usc}$ is continuous on $\sA$ we have
	\begin{align}\label{shape-aux2}
	\abs{x_n}_1^{-1}\freeunr{\orig}{x_n}{\beta} - \ppShapeUnrA{\beta, \usc}(\xi)\ge\e/2\quad\text{for all $n$ large enough.}	
	\end{align}
	
	Now, in either case, 
	follow the argument in \cite[Theorem 3.10]{Jan-Nur-Ras-22}, specifically below their equation (3.10), adjusting it as shown in the proof of Lemma \ref{lem:upperBoundCompact} to accommodate our weaker Condition \ref{classLCondition}. Following this adjusted argument, we reach the conclusion that $\bbP$-almost surely, on the event where either \eqref{shape-aux1} holds or \eqref{shape-aux2} holds, we have
	\[-\e_2+\ppShapeUnr{\beta}(\xi)
	\le\varliminf_{n\to\infty}\abs{x_n}^{-1}\freeunr{\orig}{x_n}{\beta}
	\le\varlimsup_{n\to\infty}\abs{x_n}^{-1}\freeunr{\orig}{x_n}{\beta}\le \ppShapeUnr{\beta}\Bigl(\xi+\e_1\!\!\!\sum_{z\in\rangeA}\!\!\!z\Bigr)+\e_2,\]
	where $\e_2>0$ is arbitrary, and $\e_1>0$ can be chosen to be arbitrarily small, depending on $\e_2$. 
	Take $\e_1\to0$ then $\e_2\to0$ and use the facts that $\ppShapeUnr{\beta}\Bigl(\xi+\e_1\sum_{z\in\rangeA}z\Bigr)=\ppShapeUnrA{\beta,\usc}\Bigl(\xi+\e_1\sum_{z\in\rangeA}z\Bigr)$ and $\ppShapeUnrA{\beta,\usc}$ is continuous on $\sA$ to get that 
	\[\ppShapeUnr{\beta}(\xi)
	\le\varliminf_{n\to\infty}\abs{x_n}^{-1}\freeunr{\orig}{x_n}{\beta}\le
	\varlimsup_{n\to\infty}\abs{x_n}^{-1}\freeunr{\orig}{x_n}{\beta}\le \ppShapeUnrA{\beta, \usc}(\xi),\] 
	which contradicts both \eqref{shape-aux1} and \eqref{shape-aux2}. This proves the desired inequalities. 
	The ergodicity claim is already in Theorem 3.10 of \cite{Jan-Nur-Ras-22}. The theorem is proved.
\end{proof}

\begin{proof}[Proof of Theorem \ref{Thm:ShapeRes}]
	Write the restricted-length model as an unrestricted-length model as in Remark \ref{rk:resAsUnr1}. Then the conditions on $\pote^+$ in the statement of the theorem imply that $\overline\pote^+$ satisfies the hypotheses of Theorem \ref{Thm:ShapeUnr} and, as explained in Remark \ref{rk:resAsUnr2}, $\overline\Lambda^{\beta,\usc}_{\overline\sA}(\langle \zeta,t\rangle)=t\ppShapeResU{\beta,\usc}(\zeta/t)$ for all $\zeta\in\sA$ and $t>0$ with $\zeta/t\in\sU'$.
 %
%
	Applying \eqref{shape-UB} and \eqref{shape-unr} to the unrestricted-length model gives \eqref{shape-res-UB} and \eqref{shape-res} for the restricted-length model.
%
	The ergodicity of $\bbP$ under $\{\T_z : z \in \rangeA \}$ is equivalent to its ergodicity under $\{\T_{\langle z, 1\rangle} : \langle z, 1\rangle \in \overline\range_{\overline A}\}$, which by Theorem \ref{Thm:ShapeUnr} implies that $\overline\Lambda^{\beta, \usc}_{\overline\sA}$ is deterministic on $\overline\sA$ and, therefore, in this case, $\ppShapeResU{\beta, \usc}$ is deterministic on $\sU'$.  
\end{proof}

\section{Relationship between restricted and unrestricted-length}\label{app:resunr}

We prove a theorem relating the restricted-length and unrestricted-length limiting free energies in the directed setting where $\orig \not\in\Uset$. This connection is also known to hold in the case of the standard first-passage percolation model. See \cite[Equation (2.37)]{Kri-Ras-Sep-23}.
It is natural to expect that our theorem continues to hold under appropriate hypotheses when the model has loops, i.e., $\orig \in \Uset$. We leave this for future work. 

\begin{theorem}\label{thm:resAndUnrFreeE}
Fix a face $\sA \in \faces$ {\rm(}possibly $\cone$ itself\,{\rm)}. Assume $\orig \not\in \UsetA$, Conditions \ref{VCondition} and \ref{classLCondition} hold on $\sA$, and $\bbP$ is ergodic under $\{\T_x : x \in \RgroupA\}$.  Then for each $\beta \in (0,\infty]$ and $\xi \in (\ri\sA)\setminus\{\orig\}$, 
\begin{equation}\label{eqn:resUnrRelation}
    \ppShapeUnr{\beta}(\xi) = \sup_{\substack{s > 0 \\ \xi/s \in \UsetA}} \{ s \ppShapeRes{\beta,\usc}(\xi / s)\}, \text{ and the supremum is achieved.}
\end{equation}

\end{theorem}
\begin{proof}
    For each $\zeta \in \Uset$ and $n \in \ZZ_{> 0}$, define the lattice point
    $\widehat{x}_n(\zeta)$ as in \cite[(2.1)]{Ras-Sep-14}. The point $\widehat{x}_n(\zeta)$ approximates $n\zeta$.  In particular, $\widehat{x}_n(\zeta) / n \rightarrow \zeta$ as $n \rightarrow \infty$, and if $\zeta \in \UsetA$, then $\widehat{x}_n(\zeta) \in \DnA{n}$.   
    
    For the lower bound, note that there are no loops since $\orig \not\in \UsetA$.  Let $\xi\in(\ri\sA)\setminus\{\orig\}$.  For any $s > 0$ such that $\xi / s \in \ri\UsetA$, 
    \begin{align*}
        \frac{1}{n} \log \prtunr{\orig}{\widehat{x}_{\lfloor sn \rfloor}(\xi / s) }{\beta} = \frac{1}{n} \log \sum_{k=0}^{\infty} \prtres{\orig}{\widehat{x}_{\lfloor sn \rfloor}(\xi / s)}{k}{\beta} &\geq \frac{1}{n} \log \prtres{\orig}{\widehat{x}_{\lfloor sn\rfloor}(\xi / s) }{\lfloor sn \rfloor}{\beta}.  
    \end{align*}
    The equivalent bound at zero temperature is $n^{-1} \freeunr{\orig}{\widehat{x}_{\lfloor sn \rfloor}(\xi/s) }{\infty} \ge n^{-1} \freeres{\orig}{\widehat{x}_{\lfloor sn \rfloor}(\xi/s) }{\lfloor sn \rfloor}{\infty}$.  Take limits, apply Theorems \ref{Thm:ShapeUnr} and \ref{Thm:ShapeRes}, and take the supremum over $s > 0$ to get 
    \begin{align}\label{LBaux}
   \ppShapeUnrA{\beta}(\xi)=\ppShapeUnrA{\beta,\usc}(\xi) \ge  \sup_{\substack{s > 0 \\ \xi/s \in \ri\UsetA}} \{ s \ppShapeRes{\beta}(\xi / s)\}=\sup_{\substack{s > 0 \\ \xi/s \in \UsetA}} \{ s \ppShapeRes{\beta,\usc}(\xi / s)\}.
    \end{align}
    


    For the upper bound, first note that by \cite[Lemma A.1]{Jan-Nur-Ras-22} there exist a $\uhat \in \RR^d$ and a $\delta > 0$ such that $\uhat \cdot z \geq \delta > 0$ for all $z \in \rangeA$.  
    Any admissible path from $\orig$ to $x_n$ must take at least $\abs{x_n}_1\min_{z\in\rangeA}\abs{z}_1^{-1}$ steps and at most $\delta^{-1} x_n \cdot \uhat$ steps. Thus, for $n$ large enough, the number of steps is between $\e n$ and $Cn$, 
    for some finite positive constants $C = C(\xi)$ and $\e=\e(\xi)$.  
    For such $n$, we have
    \begin{align*}
        \frac{1}{n} \log\prtunr{\orig}{x_n}{\beta} 
        &= \frac{1}{n} \log \rwE_{\orig}\bigl[ e^{-\beta\sum_{i=0}^{\stoppt{x_n}-1} \pote(\T_{X_i} \w, X_{i+1} - X_i)} \one_{\{\stoppt{x_n} < \infty\}} \bigr] \\
        &= \frac{1}{n} \log \sum_{k=0}^\infty \rwE_{\orig}\bigl[ e^{-\beta\sum_{i=0}^{k-1} \pote(\T_{X_i} \w, X_{i+1} - X_i)} \one_{\{\stoppt{x_n} = k\}} \one_{\{X_k = x_n\}} \bigr] \\
        &\leq \frac{1}{n} \log \sum_{\e n\le k\le Cn} \prtres{\orig}{x_n}{k}{\beta} \\
        &\leq \max_{\e n\leq k \leq Cn}  \frac{1}{n} \log\prtres{\orig}{x_n}{k}{\beta} + \frac{1}{n} \log Cn.
    \end{align*}
    The equivalent bound at zero temperature is
    \[
    \frac{1}{n} \freeunr{\orig}{x_n}{\infty} \le \max_{\e n \leq k \leq Cn} \frac{1}{n} \freeres{\orig}{x_n}{k}{\infty}. 
    \]
    In either case, let $k_n\in[\e n,Cn]$ be the integer that achieves the maximum. Then
        \[\frac{1}{n} \freeunr{\orig}{x_n}{\beta} =  \frac{k_n}n\cdot\frac{1}{k_n} \freeres{\orig}{x_n}{k_n}{\beta}.\]
    Take a subsequence such that $k_n/n$ converges to some $t\in[\e,C]$. In particular, $k_n\to\infty$ and $x_n/k_n\to\xi/t\in\UsetA$ as $n\to\infty$.  Take $n\to\infty$ along this subsequence and apply Theorems \ref{Thm:ShapeUnr} and \ref{Thm:ShapeRes} to get 
    \[
    \ppShapeUnrA{\beta}(\xi) \le t \ppShapeResA{\beta,\usc}(\xi / t)
    \le \sup_{\substack{s > 0 \\ \xi/s \in \UsetA}} s \ppShapeResA{\beta,\usc}(\xi / s).
    \]
Together with the lower bound \eqref{LBaux} we get that the inequalities above are, in fact, all equalities.
\end{proof}

\begin{remark}
Theorem \ref{thm:CocycleMeasuresFace} shows that the set $\facetRes{\m,\sA}$ must be non-empty. One can see this directly if one first proves the relationship in \eqref{eqn:resUnrRelation} on $\sA$.  
        Indeed, if $\m \in \superDiffUnrA{\beta,\usc}(\xi)$ for  $\xi \in(\ri \sA)\setminus\{\orig\}$, then  $\m\cdot \xi = \ppShapeUnrA{\beta, \usc}(\xi)$, and  
        using \eqref{eqn:resUnrRelation}, one would get
        \[
        \m\cdot \xi = \ppShapeUnrA{\beta, \usc}(\xi) = \ppShapeUnrA{\beta}(\xi) = s \ppShapeResA{\beta, \usc}(\xi / s)
        \]
        for some $s > 0$ with $\xi / s \in \UsetA$.  Therefore, $\ppShapeResA{\beta,\usc}(\xi / s) = \m \cdot \xi / s$, and $\facetRes{\m,\sA}$ is non-empty.  
\end{remark}

\subsection*{Acknowledgment} 
The authors extend their gratitude to the anonymous referees for their meticulous and detailed reviews, as well as for their insightful and constructive comments.

	\bibliographystyle{abbrv}  
	\bibliography{firasbib2010}

\end{document}